\documentclass[10pt,aps,prb, preprint,reqno]{amsart}
\usepackage{amsmath}
\usepackage{amssymb}
\usepackage{yfonts}
\usepackage{hyperref}
\usepackage{mathrsfs}
\usepackage {latexsym}
\usepackage{graphics}
\usepackage{txfonts}
\usepackage{enumitem}
\allowdisplaybreaks

\newtheorem{theorem}{Theorem}[section]
\newtheorem{remark}{Remark}[section]
\newtheorem{lemma}[theorem]{Lemma}

\newtheorem{proposition}[theorem]{Proposition}

\newtheorem{define}{Definition}[section]

\allowdisplaybreaks 

\begin{document}
\title[Non-uniqueness in law for 3D Navier-Stokes equations]{Non-uniqueness in law of three-dimensional Navier-Stokes equations diffused via a fractional Laplacian with power less than one half}
 
\author{Kazuo Yamazaki}  
\address{Texas Tech University, Department of Mathematics and Statistics, Lubbock, TX, 79409-1042, U.S.A.; Phone: 806-834-6112; Fax: 806-742-1112; E-mail: (kyamazak@ttu.edu)}
\date{}
\maketitle

\begin{abstract}
Non-uniqueness of three-dimensional Euler equations and Navier-Stokes equations forced by random noise, path-wise and more recently even in law, have been proven by various authors. We prove non-uniqueness in law of the three-dimensional Navier-Stokes equations forced by random noise and diffused via a fractional Laplacian that has power between zero and one half. The solution we construct has H$\ddot{\mathrm{o}}$lder regularity with a small exponent rather than Sobolev regularity with a small exponent. For the power sufficiently small, the non-uniqueness in law holds at the level of Leray-Hopf regularity. In particular, in order to handle transport error, we consider phase functions convected by not only a mollified velocity field but a sum of that with a mollified Ornstein-Uhlenbeck process if noise is additive and a product of that with a mollified exponential Brownian motion if noise is multiplicative. 
\vspace{5mm}

\end{abstract}
\footnote{2010MSC : 35A02; 35Q30; 35R60}

\section{Introduction}\label{Introduction}

\subsection{Motivation from physics and mathematics}\label{Motivation from physics and mathematics}
Various ways to describe dissipation have been proposed by atmospheric scientists and geophysicists (e.g., frictional dissipation in \cite{PBH00}). In particular, in models such as surface quasi-geostrophic equations, diffusion in a form of a fractional Laplacian appears naturally (e.g., \cite{C02}); specifically, $(-\Delta)^{m}$ for $m \in \mathbb{R}_{+}$ as a Fourier operator with a symbol $\lvert \xi \rvert^{2m}$ so that $\widehat{(-\Delta)^{m} f}(\xi) = \lvert \xi \rvert^{2m}\hat{f}(\xi)$ for any integrable function $f$ on $\mathbb{R}^{n}$ or $\mathbb{T}^{n} = [-\pi, \pi]^{n}, n\in\mathbb{N} \triangleq \{1, 2, \hdots \}$. Introduced for the first time by Lions \cite[p. 263]{L59} who subsequently in \cite[Equ. (6.164) on p. 97]{L69} claimed the uniqueness of its Leray-Hopf weak solution (see Definition \ref{weak and Leray-Hopf}) when $m\geq \frac{1}{2} + \frac{n}{4}$, the generalized Navier-Stokes (GNS) equations \eqref{deterministic GNS} that has diffusion in the form of $-(-\Delta)^{m}$ (so that it recovers the classical NS equations when $m = 1$) has captured the interests of mathematicians for more than sixty years. Except logarithmic improvements in the case of smooth initial data that was initiated by Tao \cite{T09} (also \cite{BMR14}), Lions' threshold of $\frac{1}{2} + \frac{n}{4}$ remains unbroken. On the other hand, non-uniqueness of Leray-Hopf weak solutions to the GNS equations \eqref{deterministic GNS} when $m =1$ was famously conjectured by Ladyzhenskaya \cite{L67} and remains open. Analogous statements may be made for the NS equations forced by random noise that have received much attention from researchers for more than half a century since the work of Novikov \cite{N65} (e.g.,  \cite{CGV14} for the GNS equations forced by random noise). In particular, failure of path-wise uniqueness of Leray-Hopf weak solution to the GNS equations forced by random noise \eqref{stochastic GNS} with exponent $m= 1$ remains open. This research direction concerning non-uniqueness has received special attention from the general community of stochastic partial differential equations and significant progress has been made for a certain heat equation (e.g., \cite{M96, MMP14, MS93}); however, extending the techniques developed therein to the GNS equations that are non-linear and non-local seems to be completely out of reach. In this manuscript, we prove non-uniqueness, not only path-wise but even in law (see Definition \ref{path-wise and in law}), of the GNS equations with exponent $m \in (0, \frac{1}{2})$ forced by random noise \eqref{stochastic GNS} at the level of spatial regularity $C_{x}^{\gamma}$, $\gamma > 0$ sufficiently small (see Theorems \ref{Theorem 2.1}-\ref{Theorem 2.4}). Consequences of our results include non-uniqueness in law of the GNS equations forced by random noise \eqref{stochastic GNS} at the level of Leray-Hopf regularity when $m$ is sufficiently small (see Remark \ref{Remark 2.1}). In what follows, we elaborate to make these statements precise. 

\subsection{Previous works}\label{Previous works}
We denote $\partial_{t}\triangleq \frac{\partial}{\partial t}$, velocity and pressure fields, and viscosity by 
$u: \mathbb{R}_{+} \times \mathbb{T}^{n} \mapsto \mathbb{R}^{n}, \pi: \mathbb{R}_{+} \times \mathbb{T}^{n} \mapsto \mathbb{R}$, and $\nu \geq 0$, respectively. Then the GNS equations read 
\begin{equation}\label{deterministic GNS}
\partial_{t} u + (u\cdot\nabla)u + \nabla \pi + \nu (-\Delta)^{m} u = 0, \hspace{3mm} \nabla\cdot u = 0, \hspace{3mm} \text{ for } t > 0. 
\end{equation} 
The case $m = 1, \nu > 0$ gives the classical NS equations while $\nu = 0$ the Euler equations. 
\begin{define}\label{weak and Leray-Hopf}
(E.g., \cite[Def. 3.5 and 3.6]{BV19b}) Suppose $\nu > 0$. If $u(t, \cdot)$ is weakly divergence-free, mean-zero, satisfies \eqref{deterministic GNS} distributionally and $\lVert u(t) \rVert_{L_{x}^{2}}^{2} + 2\nu \lVert u \rVert_{L_{t}^{2} \dot{H}_{x}^{m}}^{2} \leq \lVert u(0) \rVert_{L_{x}^{2}}^{2}$ for any $t \in [0,T]$, then $u \in C_{\text{weak}}^{0} ([0,T]; L_{x}^{2}) \cap L^{2}([0,T]; \dot{H}_{x}^{m})$ is a Leray-Hopf weak solution of \eqref{deterministic GNS}. On the other hand, if $u(t, \cdot)$ is weakly divergence-free, mean-zero, and satisfies \eqref{deterministic GNS} distributionally for any $t \in [0,T]$, then $u \in C_{T}^{0} L_{x}^{2}$ is a weak solution of \eqref{deterministic GNS}. 
\end{define}
The global existence of a Leray-Hopf weak solution to \eqref{deterministic GNS} in case $m =1$ is classical \cite{H51, L34}, while the case $m \in (0,1)$ can be found in \cite[The. 1.1]{CDD18}. Next, let us consider 
\begin{equation}\label{stochastic GNS}
du + [\nu (-\Delta)^{m} u + \text{div} (u\otimes u) + \nabla \pi] dt = F(u) dB, \hspace{3mm} \nabla\cdot u = 0, \hspace{3mm} \text{ for } t > 0. 
\end{equation} 

\begin{define}\label{path-wise and in law}
Suppose that $F$ is a certain operator (to be stated precisely in Section \ref{Preliminaries}) and $B$ is a Brownian motion. Then the existence of a Leray-Hopf weak solution to \eqref{stochastic GNS} in case $m = 1$, i.e., at the regularity of $L_{t}^{\infty} L_{x}^{2} \cap L_{t}^{2} \dot{H}_{x}^{1}$ and the energy inequality, was proven in \cite{FR08} (see \cite[Def. 3.3]{FR08}); cf. \cite{FG95} in which the existence of a weak solution to \eqref{stochastic GNS} in case $m = 1$ was proven but without the energy inequality (see \cite[Def. 3.1]{FG95}). If for any solution $(u, B)$ and $(\tilde{u}, \tilde{B})$ with same initial distributions, defined potentially on different filtered probability spaces, $\mathcal{L}(u) = \mathcal{L} (\tilde{u})$ holds, where $\mathcal{L}(v)$ represents the law of $v$, then uniqueness in law holds for \eqref{stochastic GNS}. On the other hand, if for any solutions $(u,B)$ and $(\tilde{u}, B)$ with common initial data defined on same probability space, $u(t) = \tilde{u}(t)$ for all $t$ with probability one, then path-wise uniqueness holds for \eqref{stochastic GNS}. While uniqueness in law does not imply path-wise uniqueness (see \cite[Exa. 2.2]{C03}), Yamada-Watanabe theorem implies the converse. Moreover, if a solution is adapted to the canonical right continuous filtration generated by $B$ and augmented by all the negligible sets, then it is a strong solution. By \cite[The. 3.2]{C03}, existence of a strong solution and uniqueness in law together imply path-wise uniqueness. 
\end{define} 
We point out that a typical proof of path-wise uniqueness, when possible, is similar to the deterministic case. For such a reason and more, a general consensus has been to devote effort to prove uniqueness in law for \eqref{stochastic GNS} (see \cite[p. 878--879]{DD03}), until the recent breakthrough developments of convex integration, which we review next. 

Gromov \cite[Par. 2.4]{G86} considered the $C^{1}$ isometric embedding theorem due to Nash \cite{N54} and Kuiper \cite{K55} as a primary example of homotopy-principle and developed convex integration technique. M$\ddot{\mathrm{u}}$ller and $\check{\mathrm{S}}$verak applied this technique to prove the existence of unexpected solutions to various equations in \cite{MS98}, and extended it to Lipschitz mappings in \cite{MS03}. De Lellis and Sz$\acute{\mathrm{e}}$kelyhidi Jr. \cite{DS09} extended the technique and proved the global existence of a weak solution to $n$-dimensional ($n$D) Euler equations, $n \geq 2$, in $L_{t,x}^{\infty}$ with compact support in space-time, extending the previous works of \cite{S93, S97} in case of $n = 2$ and regularity only in $L_{t,x}^{2}$. Improving the convex integration technique was further motivated in effort to prove the negative direction of Onsager's conjecture \cite{O49} (see \cite{CET94, E94} for positive direction) and wealth of remarkable results flourished (e.g., \cite{BDIS15, DS10, DS12, DS13}) until Isett \cite{I18} provided its complete resolution. Although the convex integration technique was limited to the Euler equations up to this point, Buckmaster and Vicol \cite{BV19a} introduced new tool called intermittent Beltrami flows and proved the non-uniqueness of 3D NS equations.  This inspired many variations: non-uniqueness of 3D GNS equations for $m \in [1, \frac{5}{4})$ \cite{LT20} and furthermore in the class of weak solutions with bounded kinetic energy, integrable vorticity that are smooth outside a fractal set of singular times with Hausdorff dimension strictly less than one \cite{BCV18}; non-uniqueness of 2D GNS equations for $m \in (0,1)$ \cite{LQ20} (and Boussinesq system \cite{LTZ20}). 

The implications of convex integration reached the stochastic case as well: non-uniqueness path-wise of compressible Euler \cite{BFH20, CFF19} (\cite{HZZ20} for non-uniqueness in law); non-uniqueness in law of stochastic NS equations \eqref{stochastic GNS} with $n = 3, m = 1$ in \cite{HZZ19}, $n = 3, m \in (\frac{13}{20}, \frac{5}{4})$ in \cite{Y20a}, $n = 2, m \in (0, 1)$ in \cite{Y20c} (and stochastic Boussinesq system in \cite{Y21a}). A natural question is whether such non-uniqueness results can be extended to the case $n = 3, m \in (0, \frac{13}{20}]$. The heart of the matter in the proof is the careful adaptation of convex integration technique to the stochastic case and upon a close inspection, it turns out that the convex integration part of \cite{Y20a} cannot be extended to the case $m \leq \frac{13}{20}$. The previous works in the deterministic case (e.g., \cite[The. 1.5]{BCV18} and \cite[The. 1]{LT20}) also required $m \geq 1$. This direction of research was partially explored by the authors in \cite{CDD18, D19} who proved the non-uniqueness of Leray-Hopf weak solution to the 3D GNS equations for $m \in (0, \frac{1}{3})$ (see \cite[The. 1.2]{CDD18} and \cite[The. 1.2]{D19}) while commenting without providing a full proof that appropriate modifications of their arguments can prove the non-uniqueness of weak solution for $m \in (0, \frac{1}{2})$ (see \cite[Cor. 2.3]{CDD18}, \cite[p. 337]{D19}). While this raises hope that appropriately adapting the proofs within \cite{CDD18, D19} to the probabilistic settings of \cite{HZZ19} can lead to the non-uniqueness in law of the 3D stochastic GNS equations \eqref{stochastic GNS} with $m \in (0, \frac{1}{2})$, unfortunately, major obstacles arise, which the author has not been able to resolve directly. The main iteration schemes within \cite{CDD18, D19} consist of of an estimate of a convective derivative, and therefore a time derivative, of Reynolds stress, e.g., $\lVert \partial_{t} \mathring{R}_{q+1} + v_{q+1} \cdot \nabla \mathring{R}_{q+1} \rVert_{0}$ in \cite[Equ. (68)]{CDD18} (see also \cite[Equ. (15)]{CDD18}, \cite[Equ. (5.43)]{D19}). The respective Reynolds stress $\mathring{R}_{q+1}$ for the stochastic GNS equations \eqref{stochastic GNS} in an additive noise case for example is given in \eqref{estimate 115} that consists of $R_{\text{com2}}$ defined in \eqref{estimate 51} which in turn consists of an Ornstein-Uhlenbeck process $z$ that is only in $C_{t}^{\alpha}$ for $\alpha < \frac{1}{2}$ by Proposition \ref{[Pro. 4.4, Y20a]} (see \cite[Rem. 1.2]{Y20c} for similar explanation). Therefore, modifying the arguments in \cite{CDD18, D19} suitably to the stochastic case seems  very difficult. 

\section{Statement of main results}
Despite the obstacles aforementioned in Subsection \ref{Previous works}, we obtain the following results; for simplicity we assume hereafter that $\nu = 1$ in \eqref{stochastic GNS} and denote an adjoint operator by an asterisk. 
\begin{theorem}\label{Theorem 2.1}
Suppose that $n = 3, m \in (0, \frac{1}{2}), F \equiv 1, B$ is a $GG^{\ast}$-Wiener process on $(\Omega, \mathcal{F}, \textbf{P})$, and $\text{Tr} ((-\Delta)^{\frac{5}{2} - m + 2 \sigma} GG^{\ast}) < \infty$ for some $\sigma > 0$. Then given $T > 0, K > 1$, and $\kappa \in (0,1)$, there exist $\gamma \in (0,1)$ and a $\textbf{P}$-almost surely (a.s.) strictly positive stopping time $\mathfrak{t}$ such that 
\begin{equation}\label{estimate 1}
\textbf{P} ( \{ \mathfrak{t} \geq T \}) > \kappa 
\end{equation} 
and the following is additionally satisfied. There exists an $(\mathcal{F}_{t})_{t\geq 0}$-adapted process $u$ that is a weak solution to \eqref{stochastic GNS} starting from a deterministic initial condition $u^{\text{in}}$, satisfies 
\begin{equation}\label{estimate 2}
\text{esssup}_{\omega \in \Omega} \lVert u(\omega) \rVert_{C_{\mathfrak{t}} C_{x}^{\gamma}} < \infty, 
\end{equation} 
and on the set $\{\mathfrak{t} \geq T \}$, 
\begin{equation}\label{estimate 7}
\lVert u(T) \rVert_{L_{x}^{2}} > K \lVert u^{\text{in}} \rVert_{L_{x}^{2}} + K ( T \text{Tr} (GG^{\ast} ))^{\frac{1}{2}}. 
\end{equation} 
\end{theorem} 
\begin{theorem}\label{Theorem 2.2}
Suppose that $n = 3, m \in (0, \frac{1}{2}), F \equiv 1, B$ is a $GG^{\ast}$-Wiener process on $(\Omega, \mathcal{F}, \textbf{P})$, and $Tr ((-\Delta)^{\frac{5}{2} - m + 2 \sigma} GG^{\ast}) < \infty$ for some $\sigma > 0$. Then non-uniqueness in law holds for \eqref{stochastic GNS} on $[0,\infty)$. Moreover, for all $T > 0$ fixed, non-uniqueness in law holds for \eqref{stochastic GNS} on $[0,T]$. 
\end{theorem} 

\begin{theorem}\label{Theorem 2.3}
Suppose that $n = 3, m \in (0, \frac{1}{2}), F(u) = u$, and $B$ is a $\mathbb{R}$-valued Wiener process on $(\Omega, \mathcal{F}, \textbf{P})$. Then given $T > 0, K > 1,$ and $\kappa \in (0,1)$, there exist $\gamma \in (0,1)$ and a $\textbf{P}$-a.s. strictly positive stopping time $\mathfrak{t}$ such that \eqref{estimate 1} holds and the following is additionally satisfied. There exists an $(\mathcal{F}_{t})_{t\geq 0}$-adapted process $u$ which is a weak solution to \eqref{stochastic GNS} starting from a deterministic initial condition $u^{\text{in}}$, satisfies \eqref{estimate 2}, and on the set $\{\mathfrak{t} \geq T \}$, 
\begin{equation}\label{estimate 68}
\lVert u(T) \rVert_{L_{x}^{2}} > K e^{\frac{T}{2}} \lVert u^{\text{in}} \rVert_{L_{x}^{2}}. 
\end{equation} 
\end{theorem} 

\begin{theorem}\label{Theorem 2.4}
Suppose that $n = 3, m \in (0, \frac{1}{2}), F(u) = u$, and $B$ is a $\mathbb{R}$-valued Wiener process on $(\Omega, \mathcal{F}, \textbf{P})$. Then non-uniqueness in law holds for \eqref{stochastic GNS} on $[0,\infty)$. Moreover, for all $T > 0$ fixed, non-uniqueness in law holds for \eqref{stochastic GNS} on $[0,T]$. 
\end{theorem} 

We emphasize that the spatial regularity $C_{x}^{\gamma}$ for $\gamma > 0$ in \eqref{estimate 2} is higher than $H_{x}^{\gamma}$ of the  solutions constructed in previous works such as \cite{HZZ19, Y20a}. 

\begin{remark}\label{Remark 2.1}
From \eqref{estimate 62}-\eqref{[Equ. (47), Y20a]} in case of Theorems \ref{Theorem 2.1}-\ref{Theorem 2.2} and \eqref{estimate 160} in case of Theorems \ref{Theorem 2.3}-\ref{Theorem 2.4}, we see that the only condition on $\gamma$ is that $\gamma < \beta$; we choose not to pursue the explicit lower bound of this $\beta \in (0, \frac{1}{2})$ because it is taken to be quite small in the proofs of Theorems \ref{Theorem 2.1} and \ref{Theorem 2.3}. Nonetheless, because $C_{t}C_{x}^{\gamma} \subset L_{t}^{2}C_{x}^{\gamma}$, we see that there exists $m \in (0, \frac{1}{2})$ such that the non-uniqueness in law stated in Theorems \ref{Theorem 2.2} and \ref{Theorem 2.4} hold for solutions at the level of Leray-Hopf regularity. Proving the non-uniqueness of Leray-Hopf weak solution requires additionally showing that appropriate energy inequality holds, and that seems difficult (cf. analogous situation for the deterministic Hall-magnetohydrodynamics system \cite{D19a}). 
\end{remark}

\begin{remark}
There are two reasons to believe that extensions of Theorems \ref{Theorem 2.1}-\ref{Theorem 2.4} to higher spatial regularity beyond \eqref{estimate 2} or $m \geq \frac{1}{2}$ will require new ideas. The first reason is simply technical; in order to handle the diffusive term $(-\Delta)^{m} u$ in \eqref{stochastic GNS}, we rely on Lemma \ref{[The. 1.4, RS16]} and its hypothesis requires that $2m + \epsilon \leq 1$ for some $\epsilon > 0$ and consequently $m < \frac{1}{2}$ (see \eqref{estimate 35} and \eqref{estimate 117}). Second, the solution to the GNS equations \eqref{deterministic GNS} possesses scaling-invariance of $u_{\lambda}(t,x) \triangleq \lambda^{2m-1} u(\lambda^{2m}t, \lambda x)$ for any $\lambda > 0$, and it follows from the definition of H$\ddot{o}$lder semi-norm that $C^{1-2m}(\mathbb{T}^{3})$ is a critical space; i.e., $\lVert u_{\lambda} (t) \rVert_{C_{x}^{1-2m}} = \lVert u(\lambda^{2m} t) \rVert_{C_{x}^{1-2m}}$ (cf. \cite[Sec. 5]{W04}). We note that local well-posedness of GNS equations in critical Besov spaces have been studied in \cite{MY06, W04, W05}; however, in this largest critical Besov space $B_{\infty, \infty}^{1-2m} = C^{1-2m}$ (see \cite[p. 99]{BCD11}), we were able to locate only \cite{YZ12} in which Yu and Zhai  proved the local existence and uniqueness of solution to the deterministic GNS equations in $B_{\infty,\infty}^{1-2m} = C^{1-2m}$ but only in case $m \in (\frac{1}{2}, 1)$ so that the non-uniqueness result in case $m < \frac{1}{2}$ can be seen as a complimentary result to \cite{YZ12} and the case $m = \frac{1}{2}$ remains intriguingly open. We were not able to locate in the literature an extension of \cite{YZ12} to the case $m \in (0,\frac{1}{2}]$; our work does not cross out such possibilities because $1-2m> 0$ for any $m \in (0, \frac{1}{2})$ and $\gamma$ in \eqref{estimate 2} is quite small. In fact, while $\gamma \in (0,\beta)$ from \eqref{estimate 62}, the estimate of \eqref{estimate 32} in the proof of Theorem \ref{Theorem 2.1} requires $\beta < \frac{1}{3} (1-2m-\epsilon)$  where $\epsilon \in (0, 1-2m)$ and therefore $\gamma < \frac{1}{3} (1-2m-\epsilon) < 1-2m$ as expected (and identically in \eqref{estimate 118} within the proof of Theorem \ref{Theorem 2.3}).  In this perspective, it seems that one should not expect any better regularity than what we achieved in case $m < \frac{1}{2}$ but close to $\frac{1}{2}$. This is in sharp contrast to the solutions $u \in H^{s}(\mathbb{T}^{3})$ for $s > 0$ quite small which were previously constructed in case $m = 1$ (e.g., \cite{BV19a, HZZ19}) that has potential to rise to the level of $H^{s}(\mathbb{T}^{3})$ for any $s < \frac{1}{2}$ as the relevant critical space in this case is $H^{\frac{1}{2}}(\mathbb{T}^{3}) = B_{2,2}^{\frac{1}{2}}(\mathbb{T}^{3})$. 

Heuristically, our proofs of Theorems \ref{Theorem 2.1}-\ref{Theorem 2.4} consists of extending ``upward'' to the GNS equations the approach on the Euler equations in \cite[Sec. 5]{BV19b} which applied convex integration at level of $C_{t,x}$ to give a simple proof of \cite[The. 1.1]{DS13} similarly to how \cite{CDD18, D19} extended the work of \cite{BDIS15} on the Euler equations. Simultaneously, we must adapt such arguments to a probabilistic setting from \cite{HZZ19} while facing major difficulty due to a transport error within the Reynolds stress on which we will elaborate in Remarks \ref{transport error} and \ref{transport error again}. Our proof can be readily simplified to prove analogous results in the deterministic case as well, and therefore gives a new simple proof of the non-uniqueness of weak solution to the 3D GNS equations \eqref{deterministic GNS} when $m \in (0, \frac{1}{2})$. 
\end{remark} 
\begin{remark}
As aforementioned, non-uniqueness in law of the GNS equations \eqref{stochastic GNS} in \cite{Y20a, Y20c} were successfully extended to the Boussinesq system \cite{Y21a}. An attempt at extensions of Theorems \ref{Theorem 2.1}-\ref{Theorem 2.4} to the Boussinesq system was countered by a surprising but somewhat inherent difficulty. In the  Boussinesq system, the equation of velocity field \eqref{stochastic GNS} contains $\theta e^{3}$ where $\theta: \mathbb{R}_{+} \times \mathbb{T}^{3} \mapsto \mathbb{R}$ represents temperature and $e^{3}$ the standard basis of $\mathbb{R}^{3}$. Consequently, $\mathring{R}_{q+1}$ in \eqref{estimate 115} would consist of $(\theta_{l} - \theta_{q+1})e^{3}$ where $\theta_{l}$ is $\theta_{q}$ after mollification in space-time (see \cite[Equ. (93), (116), and  (117a)]{Y21a}). Although the iteration scheme in \cite{Y21a} required only $\lVert \mathring{R}_{q+1} \rVert_{C_{t}L_{x}^{1}}$ (see \cite[Equ. (60b)]{Y21a}), those in the current manuscript will require $\lVert \mathring{R}_{q+1}\rVert_{C_{t,x}}$ (see \eqref{[Equ. (40c), Y20a]}). Considering that one can apply mollifier estimates to $\theta_{l} - \theta_{q}$, we can split
\begin{equation*}
\lVert \mathcal{R} ( ( \theta_{l} - \theta_{q+1}) e^{3}) \rVert_{C_{t,x}} \leq \lVert \mathcal{R} ((\theta_{l} - \theta_{q}) e^{3}) \rVert_{C_{t,x}} + \lVert \mathcal{R} ( ( \theta_{q} - \theta_{q+1}) e^{3}) \rVert_{C_{t,x}}
\end{equation*} 
(see \cite[Equ. (183)]{Y21a}) and reduce the workload of $\lVert \mathcal{R} ((\theta_{l} - \theta_{q+1}) e^{3}) \rVert_{C_{t,x}}$ to $\lVert \mathcal{R} ((\theta_{q} - \theta_{q+1}) e^{3}) \rVert_{C_{t,x}}$ where $\mathcal{R}$ is a  divergence-inverse operator (see Lemma \ref{divergence inverse}). Now one way to proceed is, similarly to \cite[Equ. (126)]{Y21a}, to rely on $\dot{W}^{1,4}(\mathbb{T}^{3}) \hookrightarrow C(\mathbb{T}^{3})$ and obtain
\begin{equation}\label{estimate 3}
\lVert \mathcal{R} ((\theta_{q+1} - \theta_{q}) e^{3}) \rVert_{C_{t,x}} \leq C \lVert \theta_{q+1} - \theta_{q} \rVert_{C_{t}L_{x}^{4}} \leq C \lVert v_{q+1} - v_{q} \rVert_{C_{t}L_{x}^{4}} \int_{0}^{t}\lVert \theta_{q} \rVert_{\dot{W}_{x}^{1,\infty}} dr. 
\end{equation}  
For $\lVert v_{q+1} - v_{q} \rVert_{C_{t}L_{x}^{4}}$ within \eqref{estimate 3}, \eqref{[Equ. (45), Y20a]} offers a bound by $(2\pi)^{\frac{3}{4}} \lVert v_{q+1} - v_{q} \rVert_{C_{t,x}} \leq (2\pi)^{\frac{3}{4}} M_{0}(t) \delta_{q+1}$ (see \eqref{estimate 69} for definitions of $M_{0}(t)$ and $\delta_{q+1}$); however, we need to bound \eqref{estimate 3} by a constant multiple of $M_{0}(t) \delta_{q+2}$ (see \eqref{[Equ. (40c), Y20a]}) and this will not be small enough because $\delta_{q+2} \ll \delta_{q+1}$. Here, there is a room for improvement between $\lVert v_{q+1} - v_{q} \rVert_{C_{t}L_{x}^{4}}$ of \eqref{estimate 3} and $\lVert v_{q+1} - v_{q} \rVert_{C_{t,x}}$ in \eqref{[Equ. (45), Y20a]}. Indeed, in \cite[Equ. (133)]{Y21a}, this issue is overcome by first splitting $\lVert v_{q+1} - v_{q} \rVert_{C_{t}L_{x}^{4}}$ to $\lVert v_{q+1} - v_{l} \rVert_{C_{t}L_{x}^{4}}  + \lVert v_{l} - v_{q} \rVert_{C_{t}L_{x}^{4}}$ where the second term can be handled via standard mollifier estimates while the first term by careful $L^{p}(\mathbb{T}^{3})$-estimates. Unfortunately, the convex integration schemes within this manuscript are extensions of the approach on the Euler equations and will be completely at the level of $C(\mathbb{T}^{3})$ (see \eqref{[Equ. (40), Y20a]}), in contrast to the approach on the NS equations which can be on $L^{p}(\mathbb{T}^{3})$ for $p < \infty$ (e.g., \cite[Sec. 7]{BV19b}).  
\end{remark} 

\section{Preliminaries}\label{Preliminaries}
We denote by $\lceil{x}\rceil$ the smallest integer $j$ such that $j \geq x$. For vector $v$, we denote its $j$-th component by $v^{j}$. We write $A \overset{(\cdot)}{\lesssim}_{a,b}B$ and $A \overset{(\cdot)}{\approx}_{a,b}B$ to imply that there exists a constant $C(a,b) \geq 0$ such that $A \leq C(a,b) B$ and $A = C(a,b) B$ due to $(\cdot)$, respectively. We set $\mathbb{N}_{0} \triangleq \mathbb{N} \cup \{0\}$. For $j, m \in \mathbb{N}_{0}$ we denote supremum norm by 
\begin{equation}\label{estimate 140}
\lVert f \rVert_{C_{t,x}} \triangleq \lVert f \rVert_{C_{t,x}^{0}} \triangleq \sup_{s \in [0,t], x \in \mathbb{T}^{3}} \lvert f(t,x) \rvert \hspace{2mm} 
 \text{ and } \hspace{2mm} \lVert f \rVert_{C_{t,x}^{m}} \triangleq \sum_{0 \leq j + \lvert \beta \rvert \leq m} \lVert \partial_{t}^{j} D^{\beta} f \rVert_{C_{t,x}^{0}}.
\end{equation} 
Furthermore, given $\alpha \in (0,1)$, we define H$\ddot{\mathrm{o}}$lder semi-norms and norms respectively by
\begin{subequations}
\begin{align}
&  [f]_{C_{t}C_{x}^{m}} \triangleq \max_{\lvert \beta \rvert = m} \lVert D^{\beta} f \rVert_{C_{t,x}^{0}}, \hspace{5mm} [f]_{C_{t}C_{x}^{m+ \alpha}} \triangleq \max_{\lvert \beta \rvert =m} \sup_{s\in[0,t], x, y \in \mathbb{T}^{3}: x \neq y} \frac{ \lvert D^{\beta} f (s,x) - D^{\beta} f(s,x) \rvert}{\lvert x-y \rvert^{\alpha}}, \label{estimate 29}\\ 
&\lVert f \rVert_{C_{t}C_{x}^{m}} \triangleq \sum_{j=0}^{m} [f]_{C_{t} C_{x}^{j}}, \hspace{9mm}  \lVert f \rVert_{C_{t} C_{x}^{m+\alpha}} \triangleq \lVert f \rVert_{C_{t}C_{x}^{m}} + [f]_{C_{t} C_{x}^{m+\alpha}}; \label{estimate 30} 
\end{align}
\end{subequations} 
here, $\beta$ is a multi-index over $\mathbb{T}^{3}$. Let us recall from \cite[Equ. (128)]{BDIS15} that for $r \geq s \geq 0$, 
\begin{equation}\label{estimate 9}
[f]_{C_{t} C_{x}^{s}} \lesssim \lVert f \rVert_{C_{t,x}}^{1- \frac{s}{r}} [f]_{C_{t}C_{x}^{r}}^{\frac{s}{r}}.
\end{equation}
We define $L_{\sigma}^{2} \triangleq \{f \in L^{2}(\mathbb{T}^{3}): \nabla\cdot f = 0, \int_{\mathbb{T}^{3}} fdx = 0 \}$, $\mathbb{P}$ to be the Leray projection operator and denote by $\mathring{\otimes}$ the trace-free part of a tensor product. For any Polish space $H$, we define $\mathcal{B}(H)$ to be the $\sigma$-algebra of Borel sets in $H$. Given any probability measure $P$, $\mathbb{E}^{P}$ denotes a mathematical expectation with respect to (w.r.t.) $P$. We represent an $L^{2}(\mathbb{T}^{3})$-inner product of $A$ and $B$ and a quadratic variation of $A$ by $\langle A, B \rangle$ and $\langle \langle A \rangle \rangle \triangleq \langle \langle A, A \rangle \rangle$, respectively. We define $\mathcal{P} (\Omega_{0})$ to be the set of all probability measures on $(\Omega_{0}, \mathcal{B})$ where $\Omega_{0} \triangleq C([0,\infty); H^{-3} (\mathbb{T}^{3})) \cap L_{\text{loc}}^{\infty} ([0,\infty); L_{\sigma}^{2})$ and $\mathcal{B}$ is the Borel $\sigma$-field of $\Omega_{0}$ from the topology of locally uniform convergence on $\Omega_{0}$. We define the canonical process $\xi: \Omega_{0} \mapsto H^{-3} (\mathbb{T}^{3})$ by $\xi_{t} (\omega) \triangleq \omega(t)$. Similarly, for $t \geq 0$ we define $\Omega_{t} \triangleq C([t,\infty); H^{-3} (\mathbb{T}^{3})) \cap L_{\text{loc}}^{\infty} ([t, \infty); L_{\sigma}^{2})$ and the following Borel $\sigma$-algebras: $\mathcal{B}^{t} \triangleq \sigma ( \{ \xi(s): s \geq t \}); \mathcal{B}_{t}^{0} \triangleq \sigma ( \{ \xi(s): s \leq t \}); \mathcal{B}_{t} \triangleq \cap_{s > t} \mathcal{B}_{s}^{0}$. For any Hilbert space $U$ we denote by $L_{2}(U, L_{\sigma}^{2})$ the space of all Hilbert-Schmidt operators from $U$ to $L_{\sigma}^{2}$ with its norm $\lVert \cdot \rVert_{L_{2}(U, L_{\sigma}^{2})}$. We require $G: L_{\sigma}^{2} \mapsto L_{2}(U, L_{\sigma}^{2})$ to be $\mathcal{B}(L_{\sigma}^{2}) / \mathcal{B}(L_{2}(U, L_{\sigma}))$-measurable and satisfy for any $\psi \in C^{\infty} (\mathbb{T}^{3}) \cap L_{\sigma}^{2}$ 
\begin{equation}
\lVert G(\psi) \rVert_{L_{2}(U, L_{\sigma}^{2})} \leq C(1+ \lVert \psi \rVert_{L_{x}^{2}}) \hspace{2mm} \text{ and } \hspace{2mm} \lim_{j\to\infty} \lVert (\theta_{j})^{\ast} \psi - G(\theta)^{\ast} \psi \rVert_{U} = 0 
\end{equation} 
for some constant $C \geq 0$ if $\lim_{j\to\infty} \lVert \theta_{j} - \theta \rVert_{L_{x}^{2}} = 0$. Furthermore, we assume the existence of another Hilbert space $U_{1}$ such that the embedding $U \hookrightarrow U_{1}$ is Hilbert-Schmidt. We define $\bar{\Omega} \triangleq C([0,\infty); H^{-3} (\mathbb{T}^{3}) \times U_{1}) \cap L_{\text{loc}}^{\infty} ([0,\infty); L_{\sigma}^{2} \times U_{1})$ and $\mathcal{P} (\bar{\Omega})$ as the set of all probability measures on $(\bar{\Omega}, \bar{\mathcal{B}})$, where $\bar{\mathcal{B}}$ is the Borel $\sigma$-algebra on $\bar{\Omega}$. Analogously, we also define the canonical process on $\bar{\Omega}$ as $(\xi, \theta): \bar{\Omega} \mapsto H^{-3}(\mathbb{T}^{3}) \times U_{1}$ by $(\xi_{t}(\omega), \theta_{t}(\omega)) \triangleq \omega(t)$. We extend the previous definitions of $\mathcal{B}^{t}, \mathcal{B}_{t}^{0}$, and $\mathcal{B}_{t}$ to $\bar{\mathcal{B}}^{t} \triangleq \sigma ( \{ (\xi, \theta)(s): s \geq t \})$, $\bar{\mathcal{B}}_{t}^{0} \triangleq \sigma( \{ (\xi, \theta) (s): s \leq t \})$, and $\bar{\mathcal{B}}_{t} \triangleq \cap_{s > t} \bar{\mathcal{B}}_{s}^{0}$ for $t\geq 0$, respectively.

The convex integration scheme we will employ in this manuscript is different from those in \cite{CDD18, D19} (deterministic) or \cite{HZZ19, Y20a, Y20c, Y21a} (stochastic). We recall some setups from \cite{BV19b, DS13} which were actually applied to the 3D deterministic Euler equations rather than the GNS equations. First, given $\zeta \in \mathbb{S}^{2} \cap \mathbb{Q}^{3}$, let $A_{\zeta} \in \mathbb{S}^{2}$ satisfy 
\begin{equation}\label{[Equ. (5.11a), BV19b]}
A_{\zeta} \cdot \zeta = 0 \hspace{1mm} \text{ and } \hspace{1mm} A_{-\zeta} = A_{\zeta}. 
\end{equation} 
We define 
\begin{equation}\label{[Equ. (5.11b), BV19b]}
B_{\zeta} \triangleq 2^{-\frac{1}{2}} (A_{\zeta} + i \zeta \times A_{\zeta}) \in \mathbb{C}^{3}. 
\end{equation} 
It follows that 
\begin{equation}\label{[Equ. (5.11c), BV19b]}
\lvert B_{\zeta} \rvert = 1, \hspace{3mm} B_{\zeta} \cdot \zeta = 0, \hspace{3mm} i \zeta \times B_{\zeta} = B_{\zeta}, \hspace{2mm} \text{ and } \hspace{2mm} B_{-\zeta} = \bar{B}_{\zeta}. 
\end{equation} 
Next, for any $\lambda \in \mathbb{Z}$ such that $\lambda \zeta \in \mathbb{Z}^{3}$ we define 
\begin{equation}\label{[Equ. (5.12), BV19b]}
W_{(\zeta)} (x) \triangleq W_{\zeta, \lambda} (x) \triangleq B_{\zeta} e^{i\lambda \zeta \cdot x}
\end{equation} 
so that it is $\mathbb{T}^{3}$-periodic, divergence-free, and 
\begin{equation}\label{eigenvector}
\nabla \times W_{(\zeta)} = \lambda W_{(\zeta)}.
\end{equation} 
\begin{lemma}\label{[Pro. 5.5, BV19b]}
\rm{(\cite[Pro. 3.1]{DS13}, \cite[Pro. 5.5]{BV19b})}  Let $\Lambda$ be a given finite subset of $\mathbb{S}^{2} \cap \mathbb{Q}^{3}$ such that $-\Lambda = \Lambda$, and $\lambda \in \mathbb{Z}$ be such that $\lambda \Lambda \subset \mathbb{Z}^{3}$. Then for any choice of coefficients $a_{\zeta} \in \mathbb{C}$ such that $\bar{a}_{\zeta} = a_{-\zeta}$ and $B_{\zeta}$ defined by \eqref{[Equ. (5.11b), BV19b]}, the vector field 
\begin{equation}
W(x) \triangleq \sum_{\zeta \in \Lambda} a_{\zeta} B_{\zeta} e^{i\lambda \zeta \cdot x} 
\end{equation} 
is a $\mathbb{R}$-valued, divergence-free Beltrami vector field such that $\nabla\times W = \lambda W$, and thus it is a stationary solution of the Euler equations 
\begin{equation}
\text{div} (W \otimes W) = \nabla \frac{ \lvert W \rvert^{2}}{2}. 
\end{equation} 
Furthermore, the following identities hold: 
\begin{equation}\label{estimate 112}
B_{\zeta} \otimes B_{-\zeta} + B_{-\zeta} \otimes B_{\zeta}  = \text{Id} - \zeta \otimes \zeta \hspace{2mm} \text{ and } \hspace{2mm}  \fint_{\mathbb{T}^{3}} W \otimes W dx = \frac{1}{2} \sum_{\zeta \in \Lambda} \lvert a_{\zeta} \rvert^{2}(\text{Id} - \zeta \otimes \zeta). 
\end{equation} 
\end{lemma} 
\begin{lemma}\label{[Pro. 5.6, BV19b]}
\rm{(\cite[Lem. 3.2]{DS13}, \cite[Pro. 5.6]{BV19b})} There exists a sufficiently small constant $C_{\ast} > 0$ with the following properties. Let $B_{C_{\ast}}(\text{Id})$ denote the closed ball of symmetric $3 \times 3$ matrices, centered at $\text{Id}$ of radius $C_{\ast}$. Then there exist pair-wise disjoint subsets 
\begin{equation}\label{[Equ. (5.15a), BV19b]}
\Lambda_{\alpha} \subset \mathbb{S}^{2} \cap \mathbb{Q}^{3}, \hspace{5mm} \alpha \in \{0,1\}, 
\end{equation} 
and smooth positive functions 
\begin{equation}\label{[Equ. (5.15b), BV19b]}
\gamma_{\zeta}^{(\alpha)} \in C^{\infty} (B_{C_{\ast}} (\text{Id} )), \hspace{5mm} \alpha \in \{0,1\}, \zeta \in \Lambda_{\alpha}, 
\end{equation} 
such that for every $\zeta \in \Lambda_{\alpha}$ we have $- \zeta \in \Lambda_{\alpha}$ and $\gamma_{\zeta}^{(\alpha)} = \gamma_{-\zeta}^{(\alpha)}$, while for every $R \in B_{C_{\ast}}(\text{Id})$ we have the identity 
\begin{equation}\label{[Equ. (5.16), BV19b]}
R = \frac{1}{2} \sum_{\zeta \in \Lambda_{\alpha}} (\gamma_{\zeta}^{(\alpha)} (R))^{2} (\text{Id} - \zeta \otimes \zeta).
\end{equation} 
\end{lemma} 

It suffices to consider index sets $\Lambda_{0}$ and $\Lambda_{1}$ in Lemma \ref{[Pro. 5.6, BV19b]} to have 12 elements (cf. \cite[Rem. 3.3]{BV19a}). By abuse of notation, we hereafter denote $\Lambda_{j} = \Lambda_{j \text{ mod} 2}$ for $j \in \mathbb{N}_{0}$. For convenience, we denote a universal constant $M$ such that for both $j \in \{0,1\}$
\begin{equation}\label{[Equ. (5.17), BV19b]}
\sum_{\zeta \in \Lambda_{j}} \lVert \gamma_{\zeta}^{(j)} \rVert_{C^{(\lceil \frac{1}{2m} \rceil + 1) \vee 10} (B_{C_{\ast}} (\text{Id} ))} \leq M. 
\end{equation} 
We leave rest of the preliminaries in the Appendix \ref{Appendix}. 

\section{Proofs of Theorems \ref{Theorem 2.1}-\ref{Theorem 2.2}}\label{Proofs Theorems 2.1-2.2}
Without loss of generality we assume that $\sigma$ in the hypothesis of Theorems \ref{Theorem 2.1}-\ref{Theorem 2.2} satisfy $\sigma \in (0,1)$. 
\subsection{Proof of Theorem \ref{Theorem 2.2} assuming Theorem \ref{Theorem 2.1}}
We fix $\gamma \in (0,1)$ for the following definitions. 
\begin{define}\label{[Def. 4.1, Y20a]}
Let $s \geq 0$ and $\xi^{\text{in}} \in L_{\sigma}^{2}$. Then $P \in \mathcal{P}(\Omega_{0})$ is a martingale solution to \eqref{stochastic GNS} with initial condition $\xi^{\text{in}}$ at initial time $s$ if 
\begin{itemize}
\item [] (M1) $P( \{ \xi(t) = \xi^{\text{in}} \hspace{1mm} \forall \hspace{1mm} t \in [0,s] \}) = 1$ and for all $l \in \mathbb{N}$ 
\begin{equation}
P( \{ \xi \in \Omega_{0}: \int_{0}^{l} \lVert G(\xi(r)) \rVert_{L_{2}(U, L_{\sigma}^{2})}^{2} dr < \infty \}) = 1, 
\end{equation} 
\item [] (M2) for every $\psi_{i} \in C^{\infty} (\mathbb{T}^{3} ) \cap L_{\sigma}^{2}$ and $t \geq s$, the process 
\begin{equation}
M_{t,s}^{i} \triangleq \langle \xi(t) - \xi(s), \psi_{i} \rangle + \int_{s}^{t} \langle \text{div} (\xi(r) \otimes \xi(r)) + (-\Delta)^{m} \xi(r), \psi_{i} \rangle dr
\end{equation}
is a continuous, square-integrable $(\mathcal{B}_{t})_{t\geq s}$-martingale under $P$ with $\langle \langle M_{t,s}^{i} \rangle \rangle = \int_{s}^{t} \lVert G(\xi(r))^{\ast} \psi_{i} \rVert_{U}^{2} dr$,  
\item [] (M3) for any $q \in \mathbb{N}$, there exists a function $t \mapsto C_{t,q} \in \mathbb{R}_{+}$ such that for all $t \geq s$, 
\begin{equation}\label{estimate 4}
\mathbb{E}^{p} [ \sup_{r \in [0,t]} \lVert \xi(r) \rVert_{L_{x}^{2}}^{2q} + \int_{s}^{t} \lVert \xi(r) \rVert_{\dot{H}_{x}^{\gamma}}^{2} dr ] \leq C_{t,q} (1+ \lVert \xi^{\text{in}} \rVert_{L_{x}^{2}}^{2q}). 
\end{equation}
\end{itemize}
The set of all such martingale solutions with the same constant $C_{t,q}$ in \eqref{estimate 4} for every $q \in \mathbb{N}$ and $t \geq s$ will be denoted by $\mathcal{C} (s, \xi^{\text{in}}, \{C_{t,q}\}_{q\in\mathbb{N}, t \geq s})$. 
\end{define}
In the current case of additive noise, if $\{\psi_{i}\}_{i=1}^{\infty}$ is a complete orthonormal system that consists of eigenvectors of $GG^{\ast}$, then $M_{t,s} \triangleq \sum_{i=1}^{\infty} M_{t,s}^{i} \psi_{i}$ becomes a $GG^{\ast}$-Wiener process starting from initial time $s$ w.r.t. the filtration $(\mathcal{B}_{t})_{t\geq s}$ under $\textbf{P}$. 
\begin{define}\label{[Def. 4.2, Y20a]}
Let $s \geq 0$, $\xi^{\text{in}} \in L_{\sigma}^{2}$ and $\tau: \Omega_{0} \mapsto [s, \infty]$ be a stopping time of $(\mathcal{B}_{t})_{t\geq s}$. Define the space of trajectories stopped at $\tau$ by 
\begin{equation}\label{estimate 159}
\Omega_{0,\tau} \triangleq \{\omega( \cdot  \wedge \tau(\omega)) : \omega \in \Omega_{0} \} = \{ \omega \in \Omega_{0}:  \xi(t,\omega) = \xi(t\wedge \tau(\omega), \omega) \hspace{1mm} \forall \hspace{1mm} t \geq 0 \}. 
\end{equation} 
Then $P \in \mathcal{P} (\Omega_{0,\tau})$ is a martingale solution to \eqref{stochastic GNS} on $[s, \tau]$ with initial condition $\xi^{\text{in}}$ at initial time $s$ if 
\begin{enumerate}
\item [] (M1) $P(\{ \xi(t) = \xi^{\text{in}} \hspace{1mm} \forall \hspace{1mm} t \in [0, s]\}) = 1$ and for all $l \in \mathbb{N}$ 
\begin{equation}
P(\{ \xi \in \Omega_{0}: \int_{0}^{l \wedge \tau} \lVert G(\xi(r)) \rVert_{L_{2}(U, L_{\sigma}^{2})}^{2} dr < \infty \}) = 1, 
\end{equation}
\item [] (M2) for every $\psi_{i} \in C^{\infty} (\mathbb{T}^{3}) \cap L_{\sigma}^{2}$ and $t\geq s$, the process 
\begin{equation}
M_{t \wedge \tau, s}^{i} \triangleq \langle \xi(t\wedge \tau) - \xi^{\text{in}}, \psi_{i} \rangle + \int_{s}^{t \wedge \tau} \langle \text{div} (\xi(r) \otimes \xi(r)) + (-\Delta)^{m} \xi(r), \psi_{i} \rangle dr 
\end{equation}
is a continuous, square-integrable $(\mathcal{B}_{t})_{t\geq s}$-martingale under $P$ with $\langle \langle M_{t\wedge \tau, s}^{i} \rangle \rangle = \int_{s}^{t\wedge \tau} \lVert G(\xi(r))^{\ast} \psi_{i} \rVert_{U}^{2} dr$, 
\item [] (M3) for any $q \in \mathbb{N}$, there exists a function $t \mapsto C_{t,q} \in \mathbb{R}_{+}$ such that for all $t \geq s$, 
\begin{equation}
\mathbb{E}^{P} [ \sup_{r \in [0, t \wedge \tau]} \lVert \xi(r) \rVert_{L_{x}^{2}}^{2q} + \int_{s}^{t \wedge \tau} \lVert \xi(r) \rVert_{\dot{H}_{x}^{\gamma}}^{2} dr] \leq C_{t,q} (1+ \lVert \xi^{\text{in}} \rVert_{L_{x}^{2}}^{2q}).
\end{equation}
\end{enumerate}
\end{define}
The proof of the following proposition concerning existence and stability of martingale solutions to \eqref{stochastic GNS} is identical to that of \cite[Pro. 4.1]{Y20a}, which in turn follows \cite[The. 3.1]{HZZ19}, because it makes use of the range of $m$ only in a few parts of its proof, which are flexible, and hence can readily be extended to our current case $m \in (0, \frac{1}{2})$. 
\begin{proposition}\label{[Pro. 4.1, Y20a]}
For any $(s, \xi^{\text{in}}) \in [0,\infty) \times L_{\sigma}^{2}$, there exists $P \in \mathcal{P}(\Omega_{0})$ which is a martingale solution to \eqref{stochastic GNS} with initial condition $\xi^{\text{in}}$ at initial time $s$ according to Definition \ref{[Def. 4.1, Y20a]}. Moreover, if there exists a family $\{(s_{l}, \xi_{l})\}_{l\in\mathbb{N}} \subset [0,\infty) \times L_{\sigma}^{2}$ such that $\lim_{l\to\infty} \lVert (s_{l}, \xi_{l}) - (s, \xi^{\text{in}}) \rVert_{\mathbb{R} \times L_{x}^{2}} = 0 $ and $P_{l} \in \mathcal{C} (s_{l}, \xi_{l}, \{C_{t,q}\}_{q\in\mathbb{N}, t \geq s_{l}})$ is the martingale solution corresponding to $(s_{l}, \xi_{l})$, then there exists a subsequence $\{P_{l_{k}}\}_{k\in\mathbb{N}}$ that converges weakly to some $P \in \mathcal{C}(s, \xi^{\text{in}}, \{C_{t,q}\}_{q \in \mathbb{N}, t \geq s})$. 
\end{proposition} 
Proposition \ref{[Pro. 4.1, Y20a]} leads to the following two results from \cite{HZZ19} which apply to our case as their proofs do not rely on the specific form of the diffusive term. Let $\mathcal{B}_{\tau}$ represent the $\sigma$-algebra associated to any given stopping time $\tau$.
 \begin{lemma}\label{[Lem. 4.2, Y20a]}
(cf. \cite[Pro. 3.2]{HZZ19}) Let $\tau$ be a bounded stopping time of $(\mathcal{B}_{t})_{t\geq 0}$. Then for every $\omega \in \Omega_{0}$, there exists $Q_{\omega} \triangleq \delta_{\omega} \otimes_{\tau(\omega)} R_{\tau(\omega), \xi(\tau(\omega), \omega)} \in \mathcal{P} (\Omega_{0})$ where $\delta_{\omega}$ is a point-mass at $\omega$ and $R_{\tau(\omega), \xi(\tau(\omega), \omega)} \in \mathcal{P}(\Omega_{0})$ is a martingale solution to \eqref{stochastic GNS} with initial condition $\xi(\tau(\omega), \omega)$ at initial time $\tau(\omega)$ such that 
\begin{subequations}
\begin{align} 
& Q_{\omega} ( \{ \omega' \in \Omega_{0}: \xi(t, \omega') = \omega(t) \hspace{1mm} \forall \hspace{1mm} t \in [0, \tau(\omega) ] \}) = 1, \label{[Equ. (20a), Y20a]} \\
& Q_{\omega}(A) = R_{\tau (\omega), \xi(\tau(\omega), \omega)} (A) \hspace{1mm} \forall \hspace{1mm} A \in \mathcal{B}^{\tau(\omega)}, \label{[Equ. (20b), Y20a]}
\end{align}
\end{subequations} 
and the mapping $\omega \mapsto Q_{\omega}(B)$ is $\mathcal{B}_{\tau}$-measurable for every $B \in \mathcal{B}$. 
\end{lemma} 

\begin{lemma}\label{[Lem. 4.3, Y20a]}
(cf. \cite[Pro. 3.4]{HZZ19}) Let $\tau$ be a bounded stopping time of $(\mathcal{B}_{t})_{t\geq 0}$, $\xi^{\text{in}} \in L_{\sigma}^{2}$, and $P$ be a martingale solution to \eqref{stochastic GNS} on $[0,\tau]$ with initial condition $\xi^{\text{in}}$ at initial time 0 according to Definition \ref{[Def. 4.2, Y20a]}. Suppose that there exists a Borel set $\mathcal{N} \subset \Omega_{0,\tau}$ such that $P(\mathcal{N}) = 0$ and $Q_{\omega}$ from Lemma \ref{[Lem. 4.2, Y20a]} satisfies for every $\omega \in \Omega_{0} \setminus \mathcal{N}$ 
\begin{equation}\label{estimate 6}
Q_{\omega} (\{\omega' \in \Omega_{0}: \tau(\omega') = \tau(\omega) \}) = 1. 
\end{equation} 
Then the probability measure $P \otimes_{\tau}R \in \mathcal{P}(\Omega_{0})$ defined by 
\begin{equation}\label{[Equ. (23), Y20a]}
P\otimes_{\tau} R (\cdot) \triangleq \int_{\Omega_{0}} Q_{\omega} (\cdot) P(d\omega) 
\end{equation} 
satisfies $P \otimes_{\tau}R \rvert_{\Omega_{0,\tau}} = P \rvert_{\Omega_{0,\tau}}$ and it is a martingale solution to \eqref{stochastic GNS} on $[0,\infty)$ with initial condition $\xi^{\text{in}}$ at initial time 0 according to Definition \ref{[Def. 4.1, Y20a]}.  
\end{lemma} 
Now we see that if  
\begin{subequations}
\begin{align}
&dz + (-\Delta)^{m} z dt + \nabla \pi^{1} dt = dB, \hspace{2mm} \nabla\cdot z = 0 \hspace{2mm} \text{ for } t > 0, \hspace{5mm} z (0,x) = 0, \label{[Equ. (24), Y20a]}\\
&\partial_{t} v + (-\Delta)^{m} v + \text{div} ((v+z) \otimes (v+z)) + \nabla \pi^{2} = 0, \nonumber\\
& \hspace{44mm}  \nabla\cdot v =0 \hspace{2mm}  \text{ for } t > 0, \hspace{5mm} v(0,x) = u^{\text{in}}(x) \label{[Equ. (25), Y20a]}
\end{align} 
\end{subequations}
so that $z(t) = \int_{0}^{t} \mathbb{P} e^{- (-\Delta)^{m} (t-s)} dB(s)$, then $u = v+z$ solves \eqref{stochastic GNS} with $\pi = \pi^{1} + \pi^{2}$. Let us formally fix a $GG^{\ast}$-Wiener process $B$ on $(\Omega, \mathcal{F}, \textbf{P})$ with $(\mathcal{F}_{t})_{t\geq 0}$ as the canonical filtration of $B$ augmented by all the $\textbf{P}$-negligible sets. We have the following results concerning regularity of $z$.  
\begin{proposition}\label{[Pro. 4.4, Y20a]}
For all $\delta \in (0, \frac{1}{2})$, $T > 0$, and $l \in \mathbb{N}$, 
\begin{equation}
\mathbb{E}^{\textbf{P}}[ \lVert z \rVert_{C_{T}\dot{H}_{x}^{\frac{5+ \sigma}{2}}}^{l} + \lVert z \rVert_{C_{T}^{\frac{1}{2} - \delta} \dot{H}_{x}^{\frac{3+ \sigma}{2}}}^{l} ] < \infty. 
\end{equation} 
\end{proposition}
\begin{proof}[Proof of Proposition \ref{[Pro. 4.4, Y20a]}]
This is an immediate consequence of \cite[Pro. 4.4]{Y21a} and the hypothesis of Theorems \ref{Theorem 2.1}-\ref{Theorem 2.2} that $\text{Tr} ((-\Delta)^{\frac{5}{2} - m + 2 \sigma} GG^{\ast}) < \infty$. 
\end{proof} 
Next, for every $\omega \in \Omega_{0}$ we define 
\begin{subequations}
\begin{align}
M_{t,0}^{\omega} \triangleq& \omega(t) - \omega(0) + \int_{0}^{t} \mathbb{P} \text{div} (\omega(r) \otimes \omega(r)) + (-\Delta)^{m} \omega(r) dr, \label{[Equation (3.12), HZZ19]}\\
Z^{\omega} (t) \triangleq& M_{t,0}^{\omega} - \int_{0}^{t} \mathbb{P} (-\Delta)^{m} e^{-(t-r) (-\Delta)^{m}} M_{r,0}^{\omega} dr. \label{[Equation (3.13), HZZ19]}
\end{align} 
\end{subequations} 
If $P$ is a martingale solution to \eqref{stochastic GNS}, then the mapping $\omega \mapsto M_{t,0}^{\omega}$ is a $GG^{\ast}$-Wiener process under $P$ and it follows from \eqref{[Equation (3.12), HZZ19]}-\eqref{[Equation (3.13), HZZ19]} that 
\begin{equation}\label{[Equation (3.13a), HZZ19]}
Z(t) =  \int_{0}^{t} \mathbb{P} e^{-(t-r) (-\Delta)^{m}} dM_{r,0}. 
\end{equation} 
It follows from Proposition \ref{[Pro. 4.4, Y20a]} that for any $\delta \in (0, \frac{1}{2})$, $Z \in C_{T} \dot{H}_{x}^{\frac{5+\sigma}{2}} \cap C_{T}^{\frac{1}{2} - \delta} \dot{H}_{x}^{\frac{3+ \sigma}{2}}$ $P$-almost surely. For $\omega \in \Omega_{0}$, $l \in \mathbb{N}$, and $\delta \in (0, \frac{1}{24})$,  we define 
\begin{align} 
\tau_{L}^{l} (\omega) \triangleq& \inf \{t \geq 0: C_{S}\lVert Z^{\omega}(t) \rVert_{\dot{H}_{x}^{\frac{5+\sigma}{2}}} >  (L - \frac{1}{l})^{\frac{1}{4}} \} \nonumber\\ 
& \wedge \inf \{ t \geq 0: C_{S}\lVert Z^{\omega} \rVert_{C_{t}^{\frac{1}{2} - 2 \delta} \dot{H}_{x}^{\frac{3 + \sigma}{2}}} > (L - \frac{1}{l})^{\frac{1}{2}} \} \wedge L \hspace{2mm} \text{ and } \hspace{2mm} \tau_{L} \triangleq \lim_{l\to\infty} \tau_{L}^{l}  \label{[Equation (3.13b), HZZ19]}
\end{align} 
where $C_{S} > 0$ is the Sobolev constant such that $\lVert f \rVert_{L_{x}^{\infty}} \leq C_{s} \lVert f \rVert_{\dot{H}_{x}^{\frac{3 + \sigma}{2}}}$ for all $f \in \dot{H}_{x}^{\frac{3+ \sigma}{2}}$ that is mean-zero. We note that the condition of $\delta \in (0, \frac{1}{24})$ is more restrictive than $\delta \in (0, \frac{1}{12})$ in previous works such as \cite{HZZ19, Y20a}, and this is needed in \eqref{estimate 135}. By \cite[Lem. 3.5]{HZZ19} it follows that $\tau_{L}$ is a stopping time of $(\mathcal{B}_{t})_{t\geq 0}$. We define for $C_{S} > 0$ in \eqref{[Equation (3.13b), HZZ19]}, $L > 1$, and $\delta \in (0, \frac{1}{24})$, 
\begin{equation}
T_{L} \triangleq \inf \{t \geq 0: C_{S}\lVert z(t) \rVert_{\dot{H}_{x}^{\frac{5+ \sigma}{2}}} \geq L^{\frac{1}{4}} \}  \wedge \inf \{ t \geq 0: C_{S}\lVert z \rVert_{C_{t}^{\frac{1}{2} - 2 \delta} \dot{H}_{x}^{\frac{3+ \sigma}{2}}} \geq L^{\frac{1}{2}} \} \wedge L, \label{[Equ. (33), Y20a]}
\end{equation} 
and realize that $T_{L} > 0$ and $\lim_{L\to\infty} T_{L} = \infty$ $\textbf{P}$-a.s. due to Proposition \ref{[Pro. 4.4, Y20a]}. The stopping time $\mathfrak{t}$ in the statement of Theorem \ref{Theorem 2.1} is actually $T_{L}$ for $L > 1$ sufficiently large. Next, we assume Theorem \ref{Theorem 2.1} on $(\Omega, \mathcal{F}, (\mathcal{F}_{t})_{t\geq 0}, \textbf{P})$ and denote the solution constructed therein by $u$ and $P = \mathcal{L} (u)$ the law of $u$. Then the following propositions can be proven identically to \cite[Pro. 4.5 and 4.6]{Y20a} as the proofs therein do not rely on the range of $m$. We only mention that a consequence from the proof of Proposition \ref{[Proposition 3.7, HZZ19]} is that $\tau_{L}$ from \eqref{[Equation (3.13b), HZZ19]} satisfies 
\begin{equation}\label{[Equ. (184), Y20a]}
\tau_{L} (u) = T_{L} \hspace{2mm} \textbf{P}\text{-almost surely}.
\end{equation} 
\begin{proposition}\label{[Proposition 3.7, HZZ19]}
Let $\tau_{L}$ be defined by \eqref{[Equation (3.13b), HZZ19]}. Then $P= \mathcal{L}(u)$ where $u$ is constructed by Theorem \ref{Theorem 2.1} is a martingale solution to \eqref{stochastic GNS} on $[0, \tau_{L}]$ according to Definition \ref{[Def. 4.2, Y20a]}. 
\end{proposition} 
\begin{proposition}\label{[Proposition 3.8, HZZ19]}
Let $\tau_{L}$ be defined by \eqref{[Equation (3.13b), HZZ19]} and $P = \mathcal{L} (u)$ constructed from Theorem \ref{Theorem 2.1}. Then $P\otimes_{\tau_{L}} R$ in \eqref{[Equ. (23), Y20a]} is a martingale solution to \eqref{stochastic GNS} on $[0,\infty)$ according to Definition \ref{[Def. 4.1, Y20a]}. 
\end{proposition}
At this point we are ready to prove Theorem \ref{Theorem 2.2}; due to its similarity to previous works \cite{HZZ19, Y20a}, we leave this in the Appendix. 

\subsection{Proof of Theorem \ref{Theorem 2.1} assuming Proposition \ref{[Pro. 4.8, Y20a]}}
Considering \eqref{[Equ. (25), Y20a]}, for $q \in \mathbb{N}_{0}$ we aim to construct a solution $(v_{q}, \mathring{R}_{q})$ to 
\begin{equation}\label{[Equ. (36), Y20a]}
\partial_{t} v_{q} + (-\Delta)^{m} v_{q} + \text{div} ((v_{q} + z) \otimes (v_{q} + z)) + \nabla \pi_{q} = \text{div} \mathring{R}_{q}, \hspace{3mm} \nabla\cdot v_{q} = 0 \hspace{3mm} \text{ for } t > 0, 
\end{equation} 
where $\mathring{R}_{q}$ is a trace-free symmetric matrix. For any $a \in \mathbb{N}$, $\beta \in (0,\frac{1}{2})$, and $L > 1$, we set 
\begin{equation}\label{estimate 69}
\lambda_{q} \triangleq  a^{2^{q}}, \hspace{3mm} \delta_{q} \triangleq \lambda_{q}^{-2\beta}, \hspace{2mm} \text{ and } \hspace{2mm} M_{0}(t) \triangleq L^{4} e^{4Lt}
\end{equation}
so that $\delta_{q}^{\frac{1}{2}} \lambda_{q} < \delta_{q+1}^{\frac{1}{2}} \lambda_{q+1}$. We note that one can also set $\lambda_{q} = a^{b^{q}}$ for $b \in \mathbb{N}$ similarly to some previous works (e.g., \cite{BV19a, HZZ19}); we chose $a^{2^{q}}$ for simplicity because choosing $b \neq 2$ will not improve our results. We see from \eqref{[Equ. (33), Y20a]} that for any $\delta \in (0, \frac{1}{24})$ and $t \in [0, T_{L}]$, 
\begin{equation}\label{[Equ. (38), Y20a]}
\lVert z(t) \rVert_{L_{x}^{\infty}} \leq L^{\frac{1}{4}}, \hspace{2mm} \lVert z(t) \rVert_{\dot{W}_{x}^{1,\infty}} \leq L^{\frac{1}{4}} \hspace{2mm} \text{ and } \hspace{2mm} \lVert z \rVert_{C_{t}^{\frac{1}{2} - 2 \delta} L_{x}^{\infty}} \leq L^{\frac{1}{2}} 
\end{equation}
by definition of $C_{S}$ from \eqref{[Equation (3.13b), HZZ19]}. Now we see that if 
\begin{equation}\label{[Equ. (4.5), HZZ19]}
a^{2\beta} > 1 + 2 (2\pi)^{\frac{3}{2}},  
\end{equation}
which we will formally state in \eqref{estimate 8}, then $\sum_{1\leq \iota \leq q} \delta_{\iota}^{\frac{1}{2}}  < \frac{1}{2(2\pi)^{\frac{3}{2}}} < \frac{1}{2}$ for any $q \in \mathbb{N}$. We set the convention that $\sum_{1\leq \iota \leq 0} \delta_{\iota}^{\frac{1}{2}} \triangleq 0$, denote by $c_{R} > 0$ a universal constant to be described subsequently (see \eqref{[Equ. (5.22a), BV19b]}, \eqref{[Equ. (5.26), BV19b]}, \eqref{[Equ. (5.33), BV19b]}) and assume the following inductive bounds: for $q \in \mathbb{N}_{0}$ and $t \in [0, T_{L}]$, 
\begin{subequations}\label{[Equ. (40), Y20a]}
\begin{align}
& \lVert v_{q} \rVert_{C_{t,x}} \leq M_{0}(t)^{\frac{1}{2}} (1+ \sum_{1\leq \iota \leq q} \delta_{\iota}^{\frac{1}{2}}) \leq 2 M_{0}(t)^{\frac{1}{2}}, \label{[Equ. (40a), Y20a]}\\
& \lVert v_{q} \rVert_{C_{t,x}^{1}} \leq M_{0}(t) \delta_{q}^{\frac{1}{2}} \lambda_{q}, \label{[Equ. (40b), Y20a]} \\
& \lVert \mathring{R}_{q}\rVert_{C_{t,x}} \leq c_{R} M_{0}(t) \delta_{q+1}.  \label{[Equ. (40c), Y20a]}
\end{align} 
\end{subequations}
\begin{proposition}\label{[Pro. 4.7, Y20a]}
For $L > 1$, define  
\begin{equation}\label{estimate 141}
v_{0}(t,x) \triangleq (2\pi)^{-\frac{3}{2}} L^{2} e^{2Lt} 
\begin{pmatrix}
\sin(x^{3}) & 0 & 0 
\end{pmatrix}^{T}. 
\end{equation} 
Then together with 
\begin{equation}
\mathring{R}_{0}(t,x) \triangleq \frac{2L^{3} e^{2Lt}}{(2\pi)^{\frac{3}{2}}} 
\begin{pmatrix}
0 & 0 & - \cos(x^{3}) \\
0 & 0 & 0 \\
-\cos(x^{3}) & 0 & 0
\end{pmatrix} 
+ (\mathcal{R} (-\Delta)^{m} v_{0} + v_{0} \mathring{\otimes} z + z \mathring{\otimes} v_{0} + z \mathring{\otimes} z) (t,x), 
\end{equation} 
it satisfies \eqref{[Equ. (36), Y20a]} at level $q = 0$. Moreover, \eqref{[Equ. (40), Y20a]} at level $q = 0$ is satisfied provided 
\begin{subequations}\label{estimate 136}
\begin{align}
&\frac{2C_{S}}{\sqrt{2}L} + \frac{20}{(2\pi)^{\frac{3}{2}} L^{\frac{3}{4}}} + \frac{10}{L^{\frac{5}{2}}} \leq 1 - \frac{4}{(2\pi)^{\frac{3}{2}}},\label{[Equ. (43), Y20a]}\\
&(1+ 2 (2\pi)^{\frac{3}{2}})^{2} <  a^{4\beta} \leq c_{R} L, \label{estimate 8}
\end{align}  
\end{subequations} 
where the first inequality of \eqref{estimate 8} guarantees \eqref{[Equ. (4.5), HZZ19]}. Furthermore, $v_{0}(0,x)$ and $\mathring{R}_{0}(0,x)$ are both deterministic.
\end{proposition} 
\begin{proof}[Proof of Proposition \ref{[Pro. 4.7, Y20a]}] 
The facts that $v_{0}$ is incompressible, mean-zero, $\mathring{R}_{0}$ is trace-free and symmetric, \eqref{[Equ. (36), Y20a]} at level $q = 0$ holds with $\pi_{0} \triangleq - \frac{1}{3} (2 v_{0} \cdot z + \lvert z \rvert^{2})$, as well as $v_{0}(0,x)$ and $\mathring{R}_{0}(0,x)$ both being deterministic can be readily verified (see \cite[Pro. 4.7]{Y20a}). Concerning the three estimates of \eqref{[Equ. (40a), Y20a]}-\eqref{[Equ. (40c), Y20a]}, first we can directly compute from \eqref{estimate 141} 
\begin{equation}\label{estimate 143}
\lVert v_{0} \rVert_{C_{t,x}} = (2\pi)^{-\frac{3}{2}}M_{0}(t)^{\frac{1}{2}} \leq M_{0}(t)^{\frac{1}{2}}, \hspace{3mm} \lVert v_{0} \rVert_{C_{t,x}^{1}} = (2\pi)^{-\frac{3}{2}}L^{2} e^{2Lt} 2 (L+1) \leq M_{0}(t) \delta_{0}^{\frac{1}{2}} \lambda_{0},
\end{equation} 
and 
\begin{equation}\label{estimate 12}
\lVert v_{0}(t) \rVert_{L_{x}^{2}} \overset{\eqref{estimate 141}}{=} \frac{M_{0}(t)^{\frac{1}{2}}}{\sqrt{2}}.
\end{equation}
Moreover, we can estimate 
\begin{equation}\label{estimate 142}
 \lVert \mathring{R}_{0} \rVert_{C_{t,x}}\leq (2\pi)^{-\frac{3}{2}} 4 L^{3} e^{2L t} + \lVert \mathcal{R} (-\Delta)^{m} v_{0} \rVert_{C_{t,x}} + 20 \lVert v_{0} \rVert_{C_{t,x}} \lVert z \rVert_{C_{t,x}} + 10 \lVert z \rVert_{C_{t,x}}^{2}. 
\end{equation} 
Next, for $C_{S} > 0$ from \eqref{[Equation (3.13b), HZZ19]} we can estimate by the Sobolev embeddings $\dot{H}^{3-2m} (\mathbb{T}^{3})  \hookrightarrow \dot{H}^{\frac{3+ \sigma}{2}}(\mathbb{T}^{3}) \hookrightarrow C(\mathbb{T}^{3})$ for functions that are mean-zero, and the fact that $\Delta v_{0} = - v_{0}$, 
\begin{equation}\label{estimate 71}
\lVert \mathcal{R} (-\Delta)^{m} v_{0} \rVert_{C_{t,x}}  \leq C_{S} \lVert \mathcal{R} (-\Delta)^{m} v_{0} \rVert_{C_{t} \dot{H}_{x}^{3-2m}}  \overset{\eqref{estimate 5}}{\leq} C_{S} 2 \lVert v_{0} \rVert_{C_{t}L_{x}^{2}} \overset{\eqref{estimate 12}}{=} C_{S} 2 \frac{M_{0}(t)^{\frac{1}{2}}}{\sqrt{2}}. 
\end{equation}  
Therefore, applying \eqref{estimate 71} to \eqref{estimate 142} gives us 
\begin{equation}
\lVert \mathring{R}_{0}\rVert_{C_{t,x}} 
\overset{\eqref{estimate 71}\eqref{estimate 142} \eqref{estimate 143} \eqref{[Equ. (38), Y20a]} }{\leq} \frac{M_{0}(t)}{L} [ \frac{4}{(2\pi)^{\frac{3}{2}}} + \frac{2C_{S}}{\sqrt{2} L} + \frac{20}{(2\pi)^{\frac{3}{2}} L^{\frac{3}{4}}} + \frac{10}{L^{\frac{5}{2}}}] 
\overset{\eqref{estimate 136}}{\leq}c_{R} M_{0}(t) \delta_{1}. 
\end{equation}
\end{proof}
\begin{proposition}\label{[Pro. 4.8, Y20a]}
Let $L$ satisfy 
\begin{equation}\label{estimate 16}
L > c_{R}^{-1} (1 + 2 (2\pi)^{\frac{3}{2}})^{2}
\end{equation} 
and \eqref{[Equ. (43), Y20a]}. Suppose that $(v_{q}, \mathring{R}_{q})$ is an $(\mathcal{F}_{t})_{t\geq 0}$-adapted process that solves \eqref{[Equ. (36), Y20a]} and satisfies \eqref{[Equ. (40a), Y20a]}-\eqref{[Equ. (40c), Y20a]}. Then there exist a choice of parameters $a$ and $\beta$ such that \eqref{estimate 8} is fulfilled and an $(\mathcal{F}_{t})_{t\geq 0}$-adapted process $(v_{q+1}, \mathring{R}_{q+1})$ that solves \eqref{[Equ. (36), Y20a]}, satisfies \eqref{[Equ. (40a), Y20a]}-\eqref{[Equ. (40c), Y20a]} at level $q+1$ and for all $t \in [0, T_{L}]$ 
\begin{equation}\label{[Equ. (45), Y20a]}
\lVert v_{q+1} - v_{q} \rVert_{C_{t,x}} \leq M_{0}(t)^{\frac{1}{2}} \delta_{q+1}^{\frac{1}{2}}. 
\end{equation} 
Finally, if $v_{q}(0,x)$ and $\mathring{R}_{q}(0,x)$ are deterministic, then so are $v_{q+1}(0,x)$ and $\mathring{R}_{q+1}(0,x)$. 
\end{proposition} 
Taking Proposition \ref{[Pro. 4.8, Y20a]} for granted, we are able to prove Theorem \ref{Theorem 2.1} now. 
\begin{proof}[Proof of Theorem \ref{Theorem 2.1} assuming Proposition \ref{[Pro. 4.8, Y20a]}]
Given any $T > 0, K > 1$, and $\kappa \in (0,1)$, starting from $(v_{0}, \mathring{R}_{0})$ in Proposition \ref{[Pro. 4.7, Y20a]}, Proposition \ref{[Pro. 4.8, Y20a]} gives us $(v_{q}, \mathring{R}_{q})$ for all $q \geq 1$ that are $(\mathcal{F}_{t})_{t\geq 0}$-adapted and satisfy \eqref{[Equ. (36), Y20a]}, \eqref{[Equ. (40a), Y20a]}-\eqref{[Equ. (40c), Y20a]}, and \eqref{[Equ. (45), Y20a]}, as well as $a$ and $\beta$ such that \eqref{estimate 8} is fulfilled. Then for all $t \in [0, T_{L}]$, $\gamma \in (0, \beta)$, using the fact that $2^{q+1} \geq 2(q+1)$ for all $q \in \mathbb{N}_{0}$, 
\begin{align}
\sum_{q\geq 0} \lVert v_{q+1} - v_{q} \rVert_{C_{t}C_{x}^{\gamma}} \lesssim& \sum_{q\geq 0} \lVert v_{q+1} - v_{q} \rVert_{C_{t,x}}^{1-\gamma} \lVert v_{q+1} - v_{q} \rVert_{C_{t,x}^{1}}^{\gamma}  \nonumber\\
&\overset{\eqref{[Equ. (45), Y20a]} \eqref{[Equ. (40b), Y20a]}}{\lesssim}  M_{0}(t)^{\frac{1+\gamma}{2}} \sum_{q\geq 0} a^{2(q+1) (\gamma - \beta)} \lesssim M_{0}(t)^{\frac{1+\gamma}{2}}. \label{estimate 62}
\end{align}
Therefore, $\{v_{q}\}_{q=1}^{\infty}$ is Cauchy in $C([0, T_{L}]; C^{\gamma}(\mathbb{T}^{3}))$ and hence we can deduce a limiting solution $\lim_{q\to\infty} v_{q} \triangleq v \in C([0, T_{L}]; C^{\gamma}(\mathbb{T}^{3}))$. It follows that there exists a deterministic constant $C_{L} > 0$ such that 
\begin{equation}\label{[Equ. (47), Y20a]}
\lVert v \rVert_{C_{T_{L}} C_{x}^{\gamma}} \leq C_{L}.
\end{equation} 
Because each $v_{q}$ is $(\mathcal{F}_{t})_{t\geq 0}$-adapted, $v$ is also $(\mathcal{F}_{t})_{t\geq 0}$-adapted. Because $\lim_{q\to \infty} \mathring{R}_{q} = 0$ in $C_{t,x}$ by \eqref{[Equ. (40c), Y20a]}, we see that $v$ is a weak solution to \eqref{[Equ. (25), Y20a]} and considering \eqref{[Equ. (24), Y20a]} we see that $u = v+ z$ solves \eqref{stochastic GNS} weakly. Now for the universal constant $c_{R} > 0$ determined from the proof of Proposition \ref{[Pro. 4.8, Y20a]} (see \eqref{[Equ. (5.22a), BV19b]}, \eqref{[Equ. (5.26), BV19b]}, \eqref{[Equ. (5.33), BV19b]}), we choose $L > 1$ sufficiently large so that it satisfies \eqref{estimate 16}, \eqref{[Equ. (43), Y20a]}, and additionally 
\begin{subequations}
\begin{align}
& \frac{3}{2} + \frac{1}{L} < (\frac{1}{\sqrt{2}} - \frac{1}{2}) e^{LT}, \label{estimate 13}\\
& L^{\frac{1}{4}} (2\pi)^{\frac{3}{2}} + K (T \text{Tr} (GG^{\ast}))^{\frac{1}{2}} \leq (e^{LT} - K) \lVert u^{\text{in}} \rVert_{L_{x}^{2}} + Le^{LT}. \label{estimate 14}
\end{align}
\end{subequations} 
As $\lim_{L\to\infty} T_{L} = + \infty$ $\textbf{P}$-a.s. due to \eqref{[Equ. (33), Y20a]}, for the $T>  0$ and $\kappa > 0$ already fixed, increasing $L$ larger if necessary gives us \eqref{estimate 1}. Because $z$ is $(\mathcal{F}_{t})_{t\geq 0}$-adapted, and we already verified that $v$ is $(\mathcal{F}_{t})_{t\geq 0}$-adapted, we deduce that $u$ is also $(\mathcal{F}_{t})_{t\geq 0}$-adapted. Moreover, \eqref{[Equ. (38), Y20a]} and \eqref{[Equ. (47), Y20a]} give us \eqref{estimate 2}. Next, for all $t \in [0, T_{L}]$, using the fact that $2^{q+1} \geq 2(q+1)$ for $q \in \mathbb{N}_{0}$, 
\begin{equation}\label{estimate 10}
\lVert v - v_{0} \rVert_{C_{t,x}} 
\overset{\eqref{[Equ. (45), Y20a]}}{\leq} M_{0}(t)^{\frac{1}{2}} \sum_{q\geq 0} \delta_{q+1}^{\frac{1}{2}} \leq M_{0}(t)^{\frac{1}{2}} \sum_{q\geq 0} a^{-2(q+1) \beta} \overset{\eqref{estimate 8}}{<} M_{0}(t)^{\frac{1}{2}} \frac{1}{2(2\pi)^{\frac{3}{2}}}. 
\end{equation} 
This implies that 
\begin{equation}\label{estimate 11}
\lVert v- v_{0} \rVert_{C_{t}L_{x}^{2}} \leq (2\pi)^{\frac{3}{2}} \lVert v - v_{0} \rVert_{C_{t,x}} \overset{\eqref{estimate 10}}{\leq} \frac{M_{0}(t)^{\frac{1}{2}}}{2} 
\end{equation} 
and therefore 
\begin{align}
( \lVert v(0) \rVert_{L_{x}^{2}} +L ) e^{LT} \overset{\eqref{estimate 11} \eqref{estimate 12}}{\leq}&  ( \frac{3}{2} M_{0}(0)^{\frac{1}{2}} + L) e^{LT} 
\overset{\eqref{estimate 13}}{<} (\frac{1}{\sqrt{2}} - \frac{1}{2}) M_{0}(T)^{\frac{1}{2}} \nonumber\\
\overset{\eqref{estimate 11} \eqref{estimate 12}}{\leq}& \lVert v_{0}(T) \rVert_{L_{x}^{2}} - \lVert v(T) - v_{0}(T) \rVert_{L_{x}^{2}} \leq \lVert v(T) \rVert_{L_{x}^{2}}. \label{estimate 15}
\end{align}
Therefore, on $\{T_{L} \geq T \}$ we obtain 
\begin{equation}
\lVert u(T) \rVert_{L_{x}^{2}} \overset{\eqref{estimate 15}}{>} (\lVert v(0) \rVert_{L_{x}^{2}} + L) e^{LT} - \lVert z(T) \rVert_{L_{x}^{\infty}} (2\pi)^{\frac{3}{2}} 
\overset{\eqref{[Equ. (24), Y20a]} \eqref{[Equ. (38), Y20a]} \eqref{estimate 14}}{\geq} K \lVert u^{\text{in}} \rVert_{L_{x}^{2}} + K (T \text{Tr} (GG^{\ast}))^{\frac{1}{2}}. 
\end{equation} 
This verifies \eqref{estimate 7}. Finally, because $v_{0}(0,x)$ is deterministic by Proposition \ref{[Pro. 4.7, Y20a]}, Proposition \ref{[Pro. 4.8, Y20a]} implies that $v(0,x)$ remains deterministic; by \eqref{[Equ. (24), Y20a]} this implies that $u^{\text{in}}$ is deterministic. 
\end{proof} 

\subsection{Proof of Proposition \ref{[Pro. 4.8, Y20a]}}
\subsubsection{Mollification}
We fix $L > 0$ that satisfies \eqref{estimate 16} and \eqref{[Equ. (43), Y20a]} and see that taking $a \in \mathbb{N}$ sufficiently large and then $\beta \in (0, \frac{1}{2})$ sufficiently small can give us \eqref{estimate 8}. Now we define 
\begin{equation}\label{[Equ. (56), Y20a]}
l \triangleq \lambda_{q}^{-\frac{3}{2}}.
\end{equation} 
We let $\{ \phi_{\epsilon} \}_{\epsilon > 0 }$ and $\{ \varphi_{\epsilon} \}_{\epsilon > 0}$ be families of standard mollifiers with mass one on $\mathbb{R}^{3}$ with compact support and $\mathbb{R}$ with compact support on $\mathbb{R}_{+}$, respectively. Then we mollify $v_{q}, \mathring{R}_{q}$, and $z$ in space and time to obtain 
\begin{equation}\label{[Equ. (58), Y20a]}
v_{l} \triangleq (v_{q} \ast_{x} \phi_{l}) \ast_{t} \varphi_{l}, \hspace{3mm} \mathring{R}_{l} \triangleq (\mathring{R}_{q} \ast_{x} \phi_{l})\ast_{t} \varphi_{l}, \hspace{3mm} z_{l} \triangleq (z\ast_{x} \phi_{l}) \ast_{t} \varphi_{l}. 
\end{equation} 
It follows from \eqref{[Equ. (36), Y20a]} that $(v_{l}, \mathring{R}_{l})$ satisfies 
\begin{equation}\label{[Equ. (59), Y20a]}
\partial_{t}v_{l} + (-\Delta)^{m} v_{l} + \text{div} ((v_{l} + z_{l}) \otimes (v_{l} + z_{l} )) + \nabla \pi_{l} = \text{div} (\mathring{R}_{l} + R_{\text{com1}}), \hspace{3mm} \nabla\cdot v_{l} = 0
\end{equation} 
for $t > 0$ where 
\begin{subequations}\label{estimate 164}
\begin{align}
R_{\text{com1}} \triangleq R_{\text{commutator1}} \triangleq& (v_{l} + z_{l}) \mathring{\otimes} (v_{l} + z_{l}) - (((v_{q} + z ) \mathring{\otimes} (v_{q} + z))\ast_{x}\phi_{l})\ast_{t} \varphi_{l}, \label{[Equ. (60a), Y20a]}\\
\pi_{l} \triangleq& (\pi_{q} \ast_{x} \phi_{l}) \ast_{t} \varphi_{l} - \frac{1}{3} (\lvert v_{l} + z_{l} \rvert^{2} - (\lvert v_{q} + z \rvert^{2} \ast_{x} \phi_{l}) \ast_{t} \varphi_{l}). \label{[Equ. (60b), Y20a]}
\end{align}
\end{subequations} 
Let us observe that because $\beta \in (0, \frac{1}{2})$ and mollifiers have mass one, for any $N \in \mathbb{N}$, by taking $a \in \mathbb{N}$ sufficiently large, 
\begin{subequations}
\begin{align}
& \lVert v_{q} - v_{l} \rVert_{C_{t,x}} \overset{\eqref{[Equ. (40b), Y20a]}}{\lesssim} l M_{0}(t) \delta_{q}^{\frac{1}{2}} \lambda_{q} \ll M_{0}(t)^{\frac{1}{2}} \delta_{q+1}^{\frac{1}{2}},\label{[Equ. (5.9), BV19b]}\\
& \lVert v_{l} \rVert_{C_{t,x}^{N}} \overset{\eqref{[Equ. (40b), Y20a]}}{\lesssim} l^{-N+1} M_{0}(t) \delta_{q}^{\frac{1}{2}} \lambda_{q}  \ll l^{-N} M_{0}(t)^{\frac{1}{2}}, \label{[Equ. (5.10), BV19b]}\\
& \lVert v_{l} \rVert_{C_{t,x}} \leq \lVert v_{q} \rVert_{C_{t,x}}  \overset{\eqref{[Equ. (40a), Y20a]}}{\leq} M_{0}(t)^{\frac{1}{2}} (1+ \sum_{1 \leq \iota \leq q} \delta_{\iota}^{\frac{1}{2}}).\label{[Equ. (5.11), BV19b]}
\end{align}
\end{subequations} 

\subsubsection{Perturbation}
Next, in order to attain acceptable estimates for transport and corrector errors subsequently, we split $[0, T_{L}]$ into an interval of size $l$, define $\Phi_{j}: [0, T_{L}] \times \mathbb{R}^{3} \mapsto \mathbb{R}^{3}$ for $j \in \{0, \hdots, \lceil l^{-1} T_{L} \rceil \}$ a $\mathbb{T}^{3}$-periodic solution to 
\begin{subequations}\label{estimate 147}
\begin{align}
(\partial_{t} + (v_{l} + z_{l}) \cdot \nabla) \Phi_{j} =0, \label{[Equ. (5.18a), BV19b]}\\
\Phi_{j} (jl, x) = x.  \label{[Equ. (5.18b), BV19b]}
\end{align}
\end{subequations}
Let us comment in Remark \ref{transport error} on the importance of including $z_{l}$ in \eqref{[Equ. (5.18a), BV19b]}.  We now collect suitable estimates on $\Phi_{j}$.
\begin{proposition}\label{Proposition 4.9}
For all $j \in \{0, \hdots, \lceil l^{-1} T_{L} \rceil \}$ and $t \in [l(j-1), l(j+1)]$ with appropriate modification in case $j = 0$ and $\lceil l^{-1} T_{L} \rceil$,  
\begin{subequations}\label{estimate 25}
\begin{align}
& \lVert \nabla \Phi_{j}(t) - \mathrm{Id} \rVert_{C_{x}} \lesssim l \delta_{q}^{\frac{1}{2}} \lambda_{q} M_{0}(t) \ll 1, \label{[Equ. (5.19a), BV19b]}\\ 
&\frac{1}{2} \leq\lvert \nabla \Phi_{j}(t,x) \rvert \leq 2 \hspace{1mm} \forall \hspace{1mm} x \in \mathbb{T}^{3} \hspace{2mm}  \text{ and } \hspace{2mm} \lVert \Phi_{j}(t) \rVert_{C_{x}^{1}} \lesssim 1, \label{estimate 17} \\
& \lVert \partial_{t} \Phi_{j}(t) \rVert_{C_{x}} \lesssim M_{0}(t)^{\frac{1}{2}}, \label{estimate 18} \\
&\lVert \nabla \Phi_{j}(t) \rVert_{C_{x}^{N}} \lesssim l^{-N+1} M_{0}(t) \delta_{q}^{\frac{1}{2}} \lambda_{q} \hspace{5mm} \forall \hspace{1mm} N \in \mathbb{N}, \label{[Equ. (5.19c), BV19b]}\\
& \lVert \partial_{t} \nabla \Phi_{j}(t) \rVert_{C_{x}^{N}} \lesssim l^{-N} M_{0}(t)^{\frac{3}{2}} \delta_{q}^{\frac{1}{2}} \lambda_{q} \hspace{3mm} \forall \hspace{1mm} N \in \mathbb{N}_{0} \label{estimate 19}
\end{align} 
\end{subequations} 
(cf. \cite[Equ. (5.19a) and (5.19c) on p. 206]{BV19b}, \cite[Lem. 3.1]{BDIS15}). 
\end{proposition} 
\begin{proof}[Proof of Proposition \ref{Proposition 4.9}]
These are just direct consequences of \cite[Pro. D.1]{BDIS15}. Specifically, first, \eqref{[Equ. (5.19a), BV19b]} follows from \cite[Equ. (135)]{BV19b} as 
\begin{equation*}
\lVert \nabla \Phi_{j}(t) - \text{Id} \rVert_{C_{x}} \overset{\eqref{[Equ. (5.10), BV19b]} \eqref{[Equ. (38), Y20a]}}{\lesssim}e^{C l M_{0}(t) \delta_{q}^{\frac{1}{2}} \lambda_{q}} - 1 \lesssim l \delta_{q}^{\frac{1}{2}} \lambda_{q}M_{0}(t) \ll 1. 
\end{equation*} 
Second, the first estimate of \eqref{estimate 17} follows from \eqref{[Equ. (5.19a), BV19b]} and the second estimate of \eqref{estimate 17} follows from \cite[Equ. (132)-(133)]{BDIS15}, \eqref{[Equ. (5.11), BV19b]}, and \eqref{[Equ. (38), Y20a]}. Third, \eqref{estimate 18} follows directly from \eqref{[Equ. (5.18a), BV19b]}, \eqref{estimate 17}, \eqref{[Equ. (5.11), BV19b]}, and \eqref{[Equ. (38), Y20a]}. Fourth, \eqref{[Equ. (5.19c), BV19b]} follows from \cite[Equ. (136)]{BDIS15} as follows: 
\begin{equation*}
\lVert \nabla \Phi_{j}(t) \rVert_{C_{x}^{N}} \overset{\eqref{[Equ. (5.10), BV19b]} \eqref{[Equ. (38), Y20a]}}{\lesssim} l [ l^{-N} M_{0}(t) \delta_{q}^{\frac{1}{2}} \lambda_{q} + l^{-N} L^{\frac{1}{4}}] e^{Cl M_{0}(t) \delta_{q}^{\frac{1}{2}} \lambda_{q}} \lesssim l^{-N+1} M_{0}(t) \delta_{q}^{\frac{1}{2}} \lambda_{q}.
\end{equation*} 
Finally, we can take $\nabla$ on \eqref{[Equ. (5.18a), BV19b]} and estimate for all $N \in \mathbb{N}$, 
\begin{align*}
& \lVert \partial_{t} \nabla \Phi_{j}(t) \rVert_{C_{x}^{N}} \leq \lVert (v_{l} + z_{l})(t) \cdot \nabla \nabla \Phi_{j}(t) \rVert_{C_{x}^{N}} + \lVert \nabla (v_{l} + z_{l})(t) \cdot \nabla \Phi_{j}(t) \rVert_{C_{x}^{N}} \\
&\hspace{6mm} \overset{\eqref{[Equ. (5.10), BV19b]} \eqref{[Equ. (5.11), BV19b]}  \eqref{[Equ. (38), Y20a]} \eqref{estimate 18} \eqref{[Equ. (5.19c), BV19b]} }{\lesssim} M_{0}(t) \delta_{q}^{\frac{1}{2}} \lambda_{q} [ l^{-N+1} M_{0}(t) \delta_{q}^{\frac{1}{2}} \lambda_{q} + l^{-N} M_{0}(t)^{\frac{1}{2}} + l^{-N}] 
\lesssim l^{-N} M_{0}(t)^{\frac{3}{2}} \delta_{q}^{\frac{1}{2}} \lambda_{q} 
\end{align*} 
while the case $N = 0$ can be proven similarly and more easily. 
\end{proof}

Next, we introduce a non-negative bump function $\chi$  that is supported in $(-1, 1)$ such that $\chi \rvert_{(-\frac{1}{4}, \frac{1}{4})} \equiv 1$ and shifted bump functions for $j \in \{0, 1, \hdots, \lceil l^{-1} T_{L} \rceil \}$ 
\begin{equation}\label{[Equ. (5.20), BV19b]}
\chi_{j}(t) \triangleq \chi (l^{-1} t - j), 
\end{equation}
which satisfy for all $t \in [0, T_{L}]$, 
\begin{equation}\label{[Equ. (5.21), BV19b]}
\sum_{j} \chi_{j}^{2}(t) = 1 \text{ and supp }\chi_{j} \subset (l(j-1), l(j+1)); 
\end{equation} 
consequently, for all $t \in [0, T_{L}]$, at most two cutoffs are non-trivial. Next, we recall Lemma \ref{[Pro. 5.6, BV19b]} and introduce an amplitude function 
\begin{equation}\label{[Equ. (5.22), BV19b]}
a_{(\zeta)}(t,x) \triangleq a_{q+1, j, \zeta} (t,x) \triangleq c_{R}^{\frac{1}{4}} \delta_{q+1}^{\frac{1}{2}} M_{0}(t)^{\frac{1}{2}} \chi_{j}(t) \gamma_{\zeta} \left(\text{Id} - \frac{\mathring{R}_{l}(t,x)}{c_{R}^{\frac{1}{2}} \delta_{q+1} M_{0}(t)} \right). 
\end{equation} 
Thus, for all $(t,x) \in [0, T_{L}] \times \mathbb{T}^{3}$, by applying Young's inequality for convolution, taking $c_{R}^{\frac{1}{2}} \leq C_{\ast}$ where $C_{\ast}$ is the constant from Lemma \ref{[Pro. 5.6, BV19b]}, and relying on the fact that mollifiers have mass one, we obtain 
\begin{equation}\label{[Equ. (5.22a), BV19b]}
\big\lvert \frac{ \mathring{R}_{l}(t,x)}{c_{R}^{\frac{1}{2}} \delta_{q+1} M_{0}(t)} \big\rvert \leq \frac{ \lVert \mathring{R}_{l} \rVert_{C_{t,x}}}{c_{R}^{\frac{1}{2}} \delta_{q+1} M_{0}(t)} 
\leq \frac{ \lVert \mathring{R}_{q} \rVert_{C_{t,x}}}{c_{R}^{\frac{1}{2}} \delta_{q+1} M_{0}(t)} \overset{\eqref{[Equ. (40c), Y20a]}}{\leq}  c_{R}^{\frac{1}{2}} \leq C_{\ast},
\end{equation}
and hence $\big\lVert \frac{ \mathring{R}_{l}}{c_{R}^{\frac{1}{2}} \delta_{q+1} M_{0}} \big\rVert_{C_{t,x}} \leq C_{\ast}$, from which it follows that
\begin{equation}\label{estimate 20}
\text{Id} - \frac{ \mathring{R}_{l}(t,x)}{c_{R}^{\frac{1}{2}} \delta_{q+1} M_{0}(t)} \in B_{C_{\ast}}(\text{Id}).
\end{equation} 
We also obtain the following crucial point-wise identity:
\begin{equation}\label{[Equ. (5.23), BV19b]}
 \frac{1}{2} \sum_{j} \sum_{\zeta \in \Lambda_{j}} a_{(\zeta)}^{2}(t,x) (\text{Id} - \zeta \otimes \zeta) 
\overset{\eqref{[Equ. (5.22), BV19b]} \eqref{[Equ. (5.16), BV19b]}\eqref{[Equ. (5.21), BV19b]}}{=} c_{R}^{\frac{1}{2}} \delta_{q+1} M_{0}(t) - \mathring{R}_{l}(t,x). 
\end{equation}
For convenience, let us record suitable estimates of the amplitude function $a_{(\zeta)}$. 
\begin{proposition}\label{Proposition 4.10}
The amplitude function $a_{(\zeta)}$ in \eqref{[Equ. (5.22), BV19b]} satisfies the following bounds on $[0, T_{L}]$: 
\begin{subequations}\label{estimate 26}
\begin{align} 
\lVert a_{(\zeta)} \rVert_{C_{t}C_{x}^{N}} \lesssim& c_{R}^{\frac{1}{4}} \delta_{q+1}^{\frac{1}{2}} M_{0}(t)^{\frac{1}{2}} \lVert \gamma_{\zeta} \rVert_{C^{N} (B_{C_{\ast}} (\mathrm{Id} ))}l^{-N} \hspace{8mm} \forall \hspace{1mm} N \in \mathbb{N}_{0},  \label{estimate 22}\\
\lVert a_{(\zeta)} \rVert_{C_{t}^{1}C_{x}^{N}} \lesssim& c_{R}^{\frac{1}{4}} \delta_{q+1}^{\frac{1}{2}} M_{0}(t)^{\frac{1}{2}} \lVert \gamma_{\zeta} \rVert_{C^{N+1} (B_{C_{\ast}} (\mathrm{Id} ))} l^{-N-1} \hspace{3mm} \forall \hspace{1mm} N \in \mathbb{N}_{0}. \label{estimate 23}
\end{align}
\end{subequations} 
\end{proposition} 
\begin{proof}[Proof of Proposition \ref{Proposition 4.10}]
The first estimate \eqref{estimate 22} in case $N = 0$ follows immediately from \eqref{[Equ. (5.21), BV19b]} and \eqref{estimate 20}. In case $N \in \mathbb{N}$ we see that 
\begin{equation}\label{estimate 92}
\lVert a_{(\zeta)} \rVert_{C_{t}C_{x}^{N}} \overset{\eqref{[Equ. (5.21), BV19b]}}{\lesssim} c_{R}^{\frac{1}{4}} \delta_{q+1}^{\frac{1}{2}} M_{0}(t)^{\frac{1}{2}} \left\lVert \gamma_{\zeta} \left(\text{Id} - \frac{ \mathring{R}_{l} (s,x)}{c_{R}^{\frac{1}{2}} \delta_{q+1} M_{0}(s)} \right) \right\rVert_{C_{t}C_{x}^{N}}
\end{equation} 
where we can rely on \cite[Equ. (129)]{BDIS15} to deduce 
\begin{equation}\label{estimate 91}
\left\lVert \gamma_{\zeta} \left(\text{Id} - \frac{\mathring{R}_{l}(s,x)}{c_{R}^{\frac{1}{2}} \delta_{q+1} M_{0}(s)}\right) \right\rVert_{C_{t}C_{x}^{N}} \overset{ \eqref{[Equ. (5.22a), BV19b]} \eqref{estimate 20} \eqref{[Equ. (40c), Y20a]}}{\lesssim} \lVert \gamma_{\zeta} \rVert_{C^{N}(B_{C_{\ast}}(\text{Id}))} l^{-N} c_{R}^{\frac{1}{2}}
\end{equation} 
so that applying \eqref{estimate 91} to \eqref{estimate 92} verifies \eqref{estimate 22} in this case as well. For estimate \eqref{estimate 23} we can directly differentiate \eqref{[Equ. (5.22), BV19b]} w.r.t. $t$ so that relying on \eqref{[Equ. (40c), Y20a]} in case $N = 0$ while additionally applying \cite[Equ. (129)]{BDIS15} in case $N \in \mathbb{N}$ can give us the desired results. 
\end{proof}

Next, we define 
\begin{equation}\label{[Equ. (5.24), BV19b]}
w_{(\zeta)}^{(p)} (t,x) \triangleq w_{q+1, j, \zeta}^{(p)} (t,x) \triangleq a_{q+1, j, \zeta} (t,x) W_{\zeta, \lambda_{q+1}} (\Phi_{j} (t,x)) = a_{q+1, j, \zeta}(t,x) B_{\zeta} e^{i\lambda_{q+1} \zeta \cdot \Phi_{j}(t,x)}
\end{equation} 
where $a_{q+1, j, \zeta}, W_{\zeta, \lambda_{q+1}}$, and $B_{\zeta}$ are defined in \eqref{[Equ. (5.22), BV19b]}, \eqref{[Equ. (5.12), BV19b]}, and \eqref{[Equ. (5.11b), BV19b]}, respectively. Then we define the principal part $w_{q+1}^{(p)}$ of a perturbation $w_{q+1}$, to be defined in \eqref{[Equ. (5.29) and (5.29a), BV19b]}, as 
\begin{equation}\label{[Equ. (5.25), BV19b]}
w_{q+1}^{(p)} (t,x) \triangleq \sum_{j} \sum_{\zeta \in \Lambda_{j}} w_{(\zeta)}^{(p)} (t,x). 
\end{equation} 
It follows by choosing $c_{R} \leq (2 \sqrt{2} M)^{-4}$ and taking advantage of the fact that for any $s \in [0,t]$ fixed, there exist at most two non-trivial cutoffs that  
\begin{equation}\label{[Equ. (5.26), BV19b]}
\lVert w_{q+1}^{(p)} \rVert_{C_{t,x}} \overset{\eqref{[Equ. (5.25), BV19b]} \eqref{[Equ. (5.24), BV19b]}\eqref{[Equ. (5.11c), BV19b]} \eqref{[Equ. (5.22), BV19b]}}{\leq} c_{R}^{\frac{1}{4}} \delta_{q+1}^{\frac{1}{2}} M_{0}(t)^{\frac{1}{2}} \sup_{s \in [0,t]} \sum_{j} \chi_{j}(s) \sum_{\zeta \in \Lambda_{j}} \lVert \gamma_{\zeta} \rVert_{C(B_{C_{\ast}} (\text{Id}))} \overset{\eqref{[Equ. (5.21), BV19b]}}{\leq} \frac{\delta_{q+1}^{\frac{1}{2}} M_{0}(t)^{\frac{1}{2}}}{2}.  
\end{equation} 
Next, we define the scalar phase function for $\zeta \in \Lambda_{j}$ 
\begin{equation}\label{[Equ. (5.27), BV19b]}
\phi_{(\zeta)} (t,x) \triangleq \phi_{q+1, j, \zeta} (t,x) \triangleq e^{ i \lambda_{q+1} \zeta \cdot (\Phi_{j} (t,x) - x)}
\end{equation}
so that we can rewrite 
\begin{equation}\label{[Equ. (5.27a), BV19b]}
w_{(\zeta)}^{(p)}(t,x)\overset{\eqref{[Equ. (5.24), BV19b]} \eqref{[Equ. (5.27), BV19b]}}{=} a_{(\zeta)} (t,x) B_{\zeta} \phi_{(\zeta)} (t,x) e^{i \lambda_{q+1} \zeta \cdot x} 
\overset{\eqref{[Equ. (5.12), BV19b]}}{=} a_{(\zeta)} (t,x) \phi_{(\zeta)} (t,x) W_{(\zeta)} (x).
\end{equation}
Due to \eqref{eigenvector} we can obtain 
\begin{equation}\label{[Equ. (5.27b), BV19b]}
a_{(\zeta)} \phi_{(\zeta)} W_{(\zeta)} = \lambda_{q+1}^{-1} \nabla \times (a_{(\zeta)} \phi_{(\zeta)} W_{(\zeta)}) - \lambda_{q+1}^{-1} \nabla (a_{(\zeta)} \phi_{(\zeta)}) \times W_{(\zeta)}. 
\end{equation} 
Therefore, if we define 
\begin{equation}\label{estimate 24}
w_{(\zeta)}^{(c)}(t,x) \triangleq \lambda_{q+1}^{-1} \nabla (a_{(\zeta)} \phi_{(\zeta)})(t,x) \times B_{\zeta} e^{i\lambda_{q+1} \zeta \cdot x}, 
\end{equation}  
then 
\begin{align}
w_{(\zeta)}^{(c)} (t,x) \overset{\eqref{estimate 24} \eqref{[Equ. (5.27), BV19b]}}{=}& \lambda_{q+1}^{-1} (\nabla a_{(\zeta)} + a_{(\zeta)} i\lambda_{q+1} \zeta \cdot (\nabla \Phi_{j} - \text{Id} ))(t,x) \times B_{\zeta} e^{i\lambda_{q+1} \zeta \cdot \Phi_{j}(t,x)} \nonumber\\
\overset{\eqref{[Equ. (5.12), BV19b]}}{=}& (\lambda_{q+1}^{-1} \nabla a_{(\zeta)} + i a_{(\zeta)} \zeta \cdot (\nabla \Phi_{j} - \text{Id} ))(t,x) \times W_{(\zeta)} (\Phi_{j}(t,x)). \label{[Equ. (5.28), BV19b]}
\end{align} 
Now we can define the incompressibility corrector $w_{q+1}^{(c)}$ and then the perturbation $w_{q+1}$ as 
\begin{equation}\label{[Equ. (5.29) and (5.29a), BV19b]}
w_{q+1}^{(c)} (t,x) \triangleq \sum_{j} \sum_{\zeta \in \Lambda_{j}} w_{(\zeta)}^{(c)} (t,x) \hspace{2mm} \text{ and } \hspace{2mm} w_{q+1} \triangleq w_{q+1}^{(p)} + w_{q+1}^{(c)}
\end{equation} 
so that 
\begin{equation}
w_{q+1} \overset{\eqref{[Equ. (5.29) and (5.29a), BV19b]} \eqref{[Equ. (5.25), BV19b]}}{=}  \sum_{j} \sum_{\zeta \in \Lambda_{j}} w_{(\zeta)}^{(p)} + w_{(\zeta)}^{(c)} \overset{\eqref{[Equ. (5.27), BV19b]}-\eqref{estimate 24} \eqref{[Equ. (5.12), BV19b]} }{=} \sum_{j}\sum_{\zeta \in \Lambda_{j}} \lambda_{q+1}^{-1} \nabla \times (a_{(\zeta)} W_{(\zeta)} \circ \Phi_{j})
\end{equation}
from which we clearly see that $w_{q+1}$ is mean-zero and divergence-free as desired. Next, for $a \in \mathbb{N}$ sufficiently large 
\begin{align}
 \lVert w_{q+1}^{(c)} \rVert_{C_{t,x}} \overset{\eqref{[Equ. (5.29) and (5.29a), BV19b]} \eqref{[Equ. (5.28), BV19b]} }{\leq}& 2 \sup_{j} \sum_{\zeta \in \Lambda_{j}} \lambda_{q+1}^{-1} \lVert \nabla a_{(\zeta)} \rVert_{C_{t,x}} + \lVert a_{(\zeta)} \rVert_{C_{t,x}} \sup_{s \in [0,t]} \lVert (\nabla \Phi_{j}(s) - \text{Id})1_{(l(j-1), l(j+1))} (s) \rVert_{C_{x}} \nonumber \\
\overset{\eqref{[Equ. (5.19a), BV19b]} \eqref{estimate 22} \eqref{[Equ. (5.17), BV19b]}\eqref{[Equ. (56), Y20a]} }{\lesssim}& M_{0}(t)^{\frac{1}{2}} \delta_{q+1}^{\frac{1}{2}}[ \lambda_{q+1}^{-1} \lambda_{q}^{\frac{3}{2}} + M_{0}(t) \lambda_{q}^{-\frac{1}{2}} \delta_{q}^{\frac{1}{2}}] \ll  \delta_{q+1}^{\frac{1}{2}}M_{0}(t)^{\frac{1}{2}}. \label{[Equ. (5.31), BV19b]}
\end{align} 
It follows now that for $a \in \mathbb{N}$ sufficiently large 
\begin{equation}\label{[Equ. (5.32a), BV19b]} 
\lVert w_{q+1} \rVert_{C_{t,x}} \overset{\eqref{[Equ. (5.29) and (5.29a), BV19b]}}{\leq} \lVert w_{q+1}^{(p)}\rVert_{C_{t,x}} + \lVert w_{q+1}^{(c)} \rVert_{C_{t,x}} 
\overset{\eqref{[Equ. (5.26), BV19b]} \eqref{[Equ. (5.31), BV19b]}}{\leq}  \frac{3 \delta_{q+1}^{\frac{1}{2}} M_{0}(t)^{\frac{1}{2}}}{4}.
\end{equation} 
Thus, if we define 
\begin{equation}\label{[Equ. (5.32), BV19b]}
v_{q+1} \triangleq v_{l} + w_{q+1}, 
\end{equation}  
then we may verify \eqref{[Equ. (45), Y20a]} as follows:
\begin{equation}
\lVert v_{q+1} - v_{q} \rVert_{C_{t,x}} \overset{\eqref{[Equ. (5.32), BV19b]}}{\leq} \lVert w_{q+1} \rVert_{C_{t,x}} + \lVert v_{l} - v_{q} \rVert_{C_{t,x}} 
\overset{\eqref{[Equ. (5.32a), BV19b]} \eqref{[Equ. (5.9), BV19b]}}{\leq} \delta_{q+1}^{\frac{1}{2}} M_{0}(t)^{\frac{1}{2}}.
\end{equation}
Next, we can verify \eqref{[Equ. (40a), Y20a]} at level $q+1$ as follows: 
\begin{equation}\label{estimate 66}
\lVert v_{q+1} \rVert_{C_{t,x}} \overset{\eqref{[Equ. (5.32), BV19b]} \eqref{[Equ. (5.32a), BV19b]}}{\leq} \lVert v_{q} \rVert_{C_{t,x}} + \frac{3}{4} \delta_{q+1}^{\frac{1}{2}} M_{0}(t)^{\frac{1}{2}} 
\overset{ \eqref{[Equ. (5.11), BV19b]}}{\leq} M_{0}(t)^{\frac{1}{2}} ( 1+ \sum_{1 \leq \iota \leq q+1} \delta_{\iota}^{\frac{1}{2}}). 
\end{equation}
Next, in order to verify \eqref{[Equ. (40b), Y20a]} at level $q+1$, we compute using the fact that for any fixed time $s \in [0,t]$, there are at most two non-trivial cutoffs 
\begin{subequations}\label{estimate 138}
\begin{align}
 \lVert \partial_{t} w_{q+1}^{(p)} \rVert_{C_{t,x}} +  \lVert \nabla w_{q+1}^{(p)} \rVert_{C_{t,x}} \overset{\eqref{estimate 25} \eqref{estimate 26} \eqref{[Equ. (5.17), BV19b]} \eqref{[Equ. (5.11c), BV19b]} }{\lesssim}& Mc_{R}^{\frac{1}{4}} \delta_{q+1}^{\frac{1}{2}} [M_{0}(t)^{\frac{1}{2}} l^{-1} + \lambda_{q+1} M_{0}(t)], \\
\lVert \partial_{t} w_{q+1}^{(c)} \rVert_{C_{t,x}} + \lVert \nabla w_{q+1}^{(c)} \rVert_{C_{t,x}} \overset{\eqref{estimate 25} \eqref{estimate 26} \eqref{[Equ. (5.17), BV19b]} \eqref{[Equ. (5.11c), BV19b]} }{\lesssim}& Mc_{R}^{\frac{1}{4}}\delta_{q+1}^{\frac{1}{2}} [\lambda_{q+1}^{-1}  M_{0}(t)^{\frac{1}{2}} l^{-2} + M_{0}(t) l^{-1} \nonumber\\
& \hspace{11mm} + M_{0}(t)^{2} \delta_{q}^{\frac{1}{2}} \lambda_{q} + \lambda_{q+1}  M_{0}(t)^{2} l \delta_{q}^{\frac{1}{2}} \lambda_{q}].
\end{align}
\end{subequations} 
Thus, taking $c_{R} \ll M^{-4}$ and $a \in \mathbb{N}$ sufficiently large gives us  
\begin{align}
\lVert w_{q+1} \rVert_{C_{t,x}^{1}} \overset{\eqref{estimate 138} \eqref{[Equ. (5.29) and (5.29a), BV19b]}\eqref{[Equ. (5.32a), BV19b]} }{\leq}& \frac{3}{4} \delta_{q+1}^{\frac{1}{2}} M_{0}(t)^{\frac{1}{2}} + \lVert \partial_{t} w_{q+1}^{(p)}\rVert_{C_{t,x}} + \lVert \nabla w_{q+1}^{(p)} \rVert_{C_{t,x}} + \lVert \partial_{t} w_{q+1}^{(c)} \rVert_{C_{t,x}} + \lVert \nabla w_{q+1}^{(c)} \rVert_{C_{t,x}} \nonumber\\
\leq& \frac{3}{4} \delta_{q+1}^{\frac{1}{2}} M_{0}(t)^{\frac{1}{2}} + C \lambda_{q+1} \delta_{q+1}^{\frac{1}{2}} M_{0}(t) M c_{R}^{\frac{1}{4}} \leq \frac{\lambda_{q+1} \delta_{q+1}^{\frac{1}{2}} M_{0}(t)}{2}. \label{[Equ. (5.33), BV19b]} 
\end{align} 
We are now ready to verify \eqref{[Equ. (40b), Y20a]} at level $q+1$ as follows. Because mollifiers have mass one, for $\beta \in (0,\frac{1}{2})$, we can take $a\in\mathbb{N}$ sufficiently large to attain due to \eqref{[Equ. (5.32), BV19b]}, \eqref{[Equ. (5.33), BV19b]}, and \eqref{[Equ. (40b), Y20a]}
\begin{equation}\label{estimate 139}
\lVert v_{q+1} \rVert_{C_{t,x}^{1}} 
\leq \lVert v_{q} \rVert_{C_{t,x}^{1}} + \frac{ \lambda_{q+1} \delta_{q+1}^{\frac{1}{2}} M_{0}(t)}{2}\leq  \lambda_{q+1} \delta_{q+1}^{\frac{1}{2}} M_{0}(t) [ a^{2^{q} [-1 + \beta]} + \frac{1}{2}] 
\leq \lambda_{q+1} \delta_{q+1}^{\frac{1}{2}} M_{0}(t). 
\end{equation} 
Subsequently, we will rely on Lemma \ref{[Lem. 5.7, BV19b]} and estimate Reynolds stress. We observe that due to \eqref{estimate 17}, if we choose $a\in\mathbb{N}$ sufficiently large, then $\frac{1}{2} \leq \lvert \nabla \Phi_{j}(t,x) \rvert \leq 2$ for all $t \in [l(j-1), l(j+1)]$ and $x \in \mathbb{T}^{3}$ so that \eqref{[Equ. (5.34a), BV19b]} is satisfied with $C = 2$. Thus, as discussed on \cite[p. 210]{BV19b}, for any $\alpha \in (0,1), p \in \mathbb{N}$, and $a$ that is smooth, periodic, $a = a 1_{(l(j-1), l(j+1))}$ such that  
\begin{equation}\label{[Equ. (5.35), BV19b]}
\lVert a \rVert_{C_{t}C_{x}^{N}} \lesssim C_{a} l^{-N}, \hspace{1mm} \forall \hspace{1mm} N \in \mathbb{N}_{0} \cap [0, p+1], p + 1 \geq \max\{\frac{1}{\alpha}, 8\},
\end{equation} 
we can estimate
\begin{equation}\label{[Equ. (5.36), BV19b]}
\lVert \mathcal{R} (a W_{(\zeta)} \circ \Phi_{j} ) \rVert_{C_{t}C_{x}^{\alpha}}  \overset{\eqref{[Equ. (5.34b), BV19b]} \eqref{[Equ. (5.35), BV19b]} }{\lesssim}_{\alpha, p} \frac{C_{a}}{\lambda_{q+1}^{1-\alpha}}.
\end{equation} 
Additionally, because for $\zeta \in \Lambda_{j}$ and $\zeta' \in \Lambda_{j'}$ such that $\lvert j-j'\rvert \leq 1$ and $\zeta + \zeta' \neq 0$, there exists $C_{\Lambda} \in (0,1)$ such that $\lvert \zeta + \zeta'\rvert \geq C_{\Lambda}$ (cf. \cite[p. 110]{BV19a} and \cite[Equ. (9)]{LT20}), it follows from \eqref{estimate 25} again that for a smooth, periodic function $a(x)$ that satisfies $a = a1_{(l(j-1), l(j+1))}$ and  \eqref{[Equ. (5.35), BV19b]}, from \cite[Equ. (5.37)]{BV19b} we have an estimate of 
\begin{equation}\label{[Equ. (5.37), BV19b]}
\lVert \mathcal{R} ( a(W_{(\zeta)} \circ \Phi_{j} \otimes W_{\zeta'} \circ \Phi_{j'})) \rVert_{C_{t}C_{x}^{\alpha}} \lesssim_{\alpha, p} \frac{C_{a}}{\lambda_{q+1}^{1-\alpha}}. 
\end{equation} 

\subsubsection{Reynolds stress}
The following decomposition of the Reynolds stress at level $q+1$ is crucial to attain the necessary estimates. First, 
\begin{align}
& \text{div} \mathring{R}_{q+1} - \nabla \pi_{q+1} \nonumber \\
\overset{\eqref{[Equ. (36), Y20a]} \eqref{[Equ. (5.32), BV19b]} \eqref{[Equ. (59), Y20a]} \eqref{[Equ. (5.29) and (5.29a), BV19b]}}{=}& - \text{div} (v_{l} \otimes z_{l} + z_{l} \otimes v_{l} + z_{l} \otimes z_{l}) - \nabla \pi_{l} + \text{div} \mathring{R}_{l} + \text{div} R_{\text{com1}} \nonumber \\
&+ \partial_{t} w_{q+1}^{(p)} + \partial_{t} w_{q+1}^{(c)} + (-\Delta)^{m} w_{q+1}  \nonumber\\
&+ \text{div} (v_{l} \otimes w_{q+1}^{(p)} + v_{l} \otimes w_{q+1}^{(c)} + v_{l} \otimes z + w_{q+1} \otimes v_{l} + w_{q+1}^{(p)} \otimes w_{q+1}^{(p)} \nonumber\\
&+ w_{q+1}^{(p)} \otimes w_{q+1}^{(c)} + w_{q+1}^{(c)} \otimes w_{q+1} + w_{q+1} \otimes z + z \otimes v_{l} + z \otimes w_{q+1} + z \otimes z). \label{estimate 28}
\end{align} 
To take advantage of mollifier estimates, we make the following arrangements by \eqref{[Equ. (5.32), BV19b]}:
\begin{align}
& - \text{div} (v_{l} \otimes z_{l} + z_{l} \otimes v_{l} + z_{l} \otimes z_{l}) \label{estimate 27}\\
&+ \text{div} (v_{l} \otimes z + w_{q+1} \otimes z + z \otimes v_{l} + z \otimes w_{q+1} + z \otimes z) \nonumber\\
=& \text{div} (v_{q+1} \otimes (z-z_{l}) + w_{q+1} \otimes z_{l} + (z-z_{l}) \otimes v_{q+1} + z_{l} \otimes w_{q+1} + z \otimes (z-z_{l}) + (z-z_{l}) \otimes z_{l}). \nonumber
\end{align} 

\begin{remark}\label{transport error}
We point out that within \eqref{estimate 27}, the most difficult term is $\text{div} (z_{l} \otimes w_{q+1}) = (z_{l} \cdot \nabla) w_{q+1}$, which is absent in the deterministic case. First, a naive attempt of rewriting 
\begin{equation*}
\text{div} (z_{l} \otimes w_{q+1}) = \text{div}(z_{l} \mathring{\otimes} w_{q+1}) + \nabla (\frac{1}{3} z_{l} \cdot w_{q+1})
\end{equation*} 
and estimating on $\lVert z_{l} \mathring{\otimes} w_{q+1} \rVert_{C_{t,x}}$ fails as 
\begin{equation}
\lVert z_{l} \mathring{\otimes} w_{q+1} \rVert_{C_{t,x}} \leq \lVert z_{l} \rVert_{C_{t,x}} \lVert w_{q+1} \rVert_{C_{t,x}}  \overset{\eqref{[Equ. (38), Y20a]} \eqref{[Equ. (5.32a), BV19b]}}{\lesssim} L^{\frac{1}{4}} \delta_{q+1}^{\frac{1}{2}} M_{0}(t)^{\frac{1}{2}} \approx \delta_{q+2} a^{2^{q+1} 3\beta} L^{\frac{1}{4}}M_{0}(t)^{\frac{1}{2}}
\end{equation}
which clearly cannot be bounded by $c_{R} M_{0}(t) \delta_{q+2}$ that is needed to attain \eqref{[Equ. (40c), Y20a]} at level $q+1$. Second, the approach of writing $\text{div} (z_{l} \otimes w_{q+1}) = (z_{l} \cdot \nabla) w_{q+1}$ and relying on \eqref{[Equ. (5.36), BV19b]} also fails, because $\nabla$ is applied on 
\begin{equation*}
w_{q+1} \overset{\eqref{[Equ. (5.29) and (5.29a), BV19b]} \eqref{[Equ. (5.24), BV19b]} \eqref{[Equ. (5.28), BV19b]}}{=}\sum_{j} \sum_{\zeta \in \Lambda_{j}} a_{(\zeta)} B_{\zeta} e^{i \lambda_{q+1} \zeta \cdot \Phi_{j}} + [ \lambda_{q+1}^{-1} \nabla a_{(\zeta)} + i a_{(\zeta)} \zeta \cdot (\nabla \Phi_{j} - \text{Id} )] \times B_{\zeta} e^{i \lambda_{q+1} \zeta \cdot \Phi_{j}},
\end{equation*} 
and thus particularly on $e^{i\lambda_{q+1} \zeta \cdot \Phi_{j}}$, and the $\lambda_{q+1}$ from its chain rule becomes too large to handle. Our new idea to overcome this difficulty is to include $(z_{l} \cdot \nabla) \Phi_{j}$ in \eqref{[Equ. (5.18a), BV19b]}, and include this problematic term $(z_{l} \cdot \nabla)w_{q+1}$ within the transport and corrector errors in $R_{\text{tran}}$ and $R_{\text{corr}}$, to be defined respectively in \eqref{estimate 45} and \eqref{estimate 49}, so that not only the term when $\nabla$ is applied on $e^{i\lambda_{q+1} \zeta \cdot \Phi_{j}}$ in $(v_{l} \cdot \nabla) w_{q+1}$ vanishes, but the term when $\nabla$ is applied on  $e^{i\lambda_{q+1} \zeta \cdot \Phi_{j}}$ in $(z_{l} \cdot \nabla) w_{q+1}$ also vanishes, as we will see in \eqref{estimate 54} and \eqref{estimate 61}. Let us make this precise. 
\end{remark}
Let us write $\text{div} (z_{l} \otimes w_{q+1})$ in \eqref{estimate 27} as
\begin{equation*}
\text{div} (z_{l} \otimes w_{q+1}) \overset{\eqref{[Equ. (5.29) and (5.29a), BV19b]}}{=} (z_{l} \cdot \nabla) w_{q+1}^{(p)} + (z_{l} \cdot \nabla) w_{q+1}^{(c)} 
\end{equation*} 
and apply \eqref{estimate 27} to \eqref{estimate 28} to write 
\begin{align}
&\text{div} \mathring{R}_{q+1} - \nabla \pi_{q+1} \label{[Equ. (5.38), BV19b]}\\
=&  \underbrace{(-\Delta)^{m} w_{q+1} + (w_{q+1} \cdot \nabla) z_{l}}_{\text{div} R_{\text{line}} } +  \underbrace{(\partial_{t} + (v_{l} + z_{l}) \cdot \nabla)w_{q+1}^{(p)}}_{\text{div} R_{\text{tran}}} + \underbrace{\text{div} (w_{q+1}^{(p)} \otimes w_{q+1}^{(p)} + \mathring{R}_{l})}_{\text{div}R_{\text{osc}} + \nabla \pi_{\text{osc}}}  \nonumber\\
&+ \underbrace{(w_{q+1} \cdot \nabla) v_{l}}_{\text{div} R_{\text{Nash}}} + \underbrace{( \partial_{t} + (v_{l} + z_{l}) \cdot \nabla) w_{q+1}^{(c)} + \text{div} (w_{q+1}^{(c)} \otimes w_{q+1} + w_{q+1}^{(p)} \otimes w_{q+1}^{(c)})}_{\text{div}R_{\text{corr}} + \nabla \pi_{\text{corr}}} \nonumber\\
& + \text{div} R_{\text{com1}} - \nabla \pi_{l} + \underbrace{\text{div} (v_{q+1} \otimes (z-z_{l}) + (z-z_{l}) \otimes v_{q+1} + z \otimes (z-z_{l}) + (z-z_{l}) \otimes z_{l})}_{\text{div} R_{\text{com2}} + \nabla \pi_{\text{com2}}}. \nonumber 
\end{align} 
Concerning $R_{\text{osc}}$ and $\pi_{\text{osc}}$ in \eqref{[Equ. (5.38), BV19b]}, first we see that $\chi_{j}(t)\chi_{j'}(t) = 0$ if $\lvert j-j'\rvert \geq 2$ because $\chi_{j}$ has support in $(l(j-1), l(j+1))$. Second, by Lemma \ref{[Pro. 5.6, BV19b]} we know that $\Lambda_{j} \cap \Lambda_{j'} = \emptyset$ if $\lvert j- j' \rvert = 1$. Third, using an identity of 
\begin{equation}\label{estimate 163}
(A \cdot \nabla B) + (B\cdot \nabla)A = \nabla (A \cdot B) - A \times \nabla \times B - B \times \nabla \times A
\end{equation} 
and \eqref{eigenvector}, we can compute 
\begin{equation}\label{estimate 113}
\text{div} (W_{(\zeta)} \otimes W_{(\zeta')} + W_{(\zeta')} \otimes W_{(\zeta)}) = \nabla (W_{(\zeta)} \cdot W_{(\zeta')}).
\end{equation}
Taking into account of these observations allows us to rewrite  
\begin{align}
& \text{div} (w_{q+1}^{(p)} \otimes w_{q+1}^{(p)} + \mathring{R}_{l}) \label{estimate 144} \\
\overset{\eqref{[Equ. (5.25), BV19b]}}{=}& \text{div} ( \sum_{j} \sum_{\zeta \in \Lambda_{j}} w_{(\zeta)}^{(p)} \otimes w_{(-\zeta)}^{(p)} + \mathring{R}_{l}) + \sum_{j, j'} \sum_{\zeta \in \Lambda_{j}, \zeta' \in \Lambda_{j'}: \zeta + \zeta' \neq 0} \text{div} (w_{(\zeta)}^{(p)} \otimes w_{(\zeta')}^{(p)}) \nonumber\\
\overset{\eqref{[Equ. (5.12), BV19b]} \eqref{[Equ. (5.24), BV19b]} \eqref{[Equ. (5.27a), BV19b]}}{=}& \text{div} ( \sum_{j} \sum_{\zeta \in \Lambda_{j}} a_{q+1, j, \zeta} B_{\zeta} e^{i \lambda_{q+1} \zeta \cdot \Phi_{j}} \otimes a_{q+1, j', -\zeta} B_{-\zeta} e^{-i\lambda_{q+1} \zeta \cdot \Phi_{j}} + \mathring{R}_{l}) \nonumber\\
& + \sum_{j,j'} \sum_{\zeta \in \Lambda_{j}, \zeta' \in \Lambda_{j'}: \zeta + \zeta' \neq 0} \text{div} (a_{(\zeta)} \phi_{(\zeta)} W_{(\zeta)} \otimes a_{(\zeta')} \phi_{(\zeta')} W_{(\zeta')}) \nonumber\\
\overset{\eqref{estimate 112} \eqref{[Equ. (5.23), BV19b]}}{=}& \text{div} (c_{R}^{\frac{1}{2}} \delta_{q+1} M_{0}(t)) + \sum_{j,j'} \sum_{\zeta \in \Lambda_{j}, \zeta' \in \Lambda_{j'}: \zeta + \zeta' \neq 0} \text{div} (a_{(\zeta)} a_{(\zeta')} \phi_{(\zeta)} \phi_{(\zeta')} W_{(\zeta)} \otimes W_{(\zeta')}) \nonumber \\
\overset{\eqref{estimate 113}}{=}& \nabla ( \frac{1}{2} \sum_{j,j'} \sum_{\zeta \in \Lambda_{j}, \zeta' \in \Lambda_{j'}: \zeta + \zeta' \neq 0} a_{(\zeta)} a_{(\zeta')} \phi_{(\zeta)} \phi_{(\zeta')} (W_{(\zeta)} \cdot W_{(\zeta')})) \nonumber\\
&+ \text{div} \mathcal{R} ( \sum_{j,j'} \sum_{\zeta \in \Lambda_{j}, \zeta' \in \Lambda_{j''}: \zeta + \zeta' \neq 0} (W_{(\zeta)} \otimes W_{(\zeta')} - \frac{W_{(\zeta)} \cdot W_{(\zeta')}}{2} \text{Id}) \nabla (a_{(\zeta)} a_{(\zeta')} \phi_{(\zeta)} \phi_{(\zeta')})). \nonumber 
\end{align} 
Thus, \eqref{[Equ. (5.38), BV19b]} and \eqref{estimate 144} motivate us to define in addition to $R_{\text{com1}}$ and $\pi_{l}$ defined in \eqref{estimate 164}, 
\begin{subequations}\label{estimate 148}
\begin{align}
R_{\text{line}} \triangleq& R_{\text{linear}} \triangleq \mathcal{R} ( (-\Delta)^{m} w_{q+1} + (w_{q+1} \cdot \nabla) z_{l}), \label{estimate 44}\\
R_{\text{tran}} \triangleq& R_{\text{transport}} \triangleq \mathcal{R} ( ( \partial_{t} + (v_{l} + z_{l}) \cdot \nabla) w_{q+1}^{(p)}), \label{estimate 45}  \\
R_{\text{osc}} \triangleq& R_{\text{oscillation}} \nonumber\\
\triangleq& \mathcal{R} ( \sum_{j,j'} \sum_{\zeta \in \Lambda_{j}, \zeta' \in \Lambda_{j'}: \zeta + \zeta' \neq 0} (W_{(\zeta)} \otimes W_{(\zeta')} - \frac{W_{(\zeta)} \cdot W_{(\zeta')}}{2} \text{Id}) \nabla (a_{(\zeta)} a_{(\zeta')} \phi_{(\zeta)} \phi_{(\zeta')})), \label{estimate 46} \\
\pi_{\text{osc}} \triangleq& \pi_{\text{oscillation}} \triangleq \frac{1}{2} \sum_{j,j'} \sum_{\zeta \in \Lambda_{j}, \zeta' \in \Lambda_{j'}: \zeta + \zeta' \neq 0} a_{(\zeta)} a_{(\zeta')} \phi_{(\zeta)} \phi_{(\zeta')} (W_{(\zeta)} \cdot W_{(\zeta')}), \label{estimate 47} \\
R_{\text{Nash}} \triangleq& \mathcal{R} ( (w_{q+1} \cdot\nabla) v_{l}), \label{estimate 48} \\
R_{\text{corr}} \triangleq& R_{\text{corrector}} \triangleq \mathcal{R}  (( \partial_{t} + (v_{l} + z_{l}) \cdot \nabla) w_{q+1}^{(c)}) + w_{q+1}^{(c)} \mathring{\otimes} w_{q+1} + w_{q+1}^{(p)} \mathring{\otimes} w_{q+1}^{(c)}, \label{estimate 49} \\
\pi_{\text{corr}} \triangleq& \pi_{\text{corrector}} \triangleq \frac{1}{3} \lvert w_{q+1}^{(c)} \rvert^{2} + \frac{2}{3} w_{q+1}^{(p)} \cdot w_{q+1}^{(c)}, \label{estimate 50} \\
R_{\text{com2}} \triangleq& R_{\text{commutator2}} \triangleq v_{q+1} \mathring{\otimes} (z-z_{l}) + (z-z_{l}) \mathring{\otimes} v_{q+1} + z \mathring{\otimes} (z-z_{l}) + (z-z_{l}) \mathring{\otimes} z_{l},  \label{estimate 51}\\
\pi_{\text{com2}} \triangleq& \pi_{\text{commutator2}} \triangleq \frac{1}{3} (2v_{q+1} \cdot (z-z_{l}) + \lvert z \rvert^{2} - \lvert z_{l} \rvert^{2}). \label{estimate 52}
\end{align}
\end{subequations}  
We define from \eqref{[Equ. (5.38), BV19b]}
\begin{subequations} 
\begin{align}
\mathring{R}_{q+1} \triangleq& R_{\text{line}} + R_{\text{tran}} + R_{\text{osc}} + R_{\text{Nash}} + R_{\text{corr}} +  R_{\text{com1}} + R_{\text{com2}}, \label{estimate 115}\\
\pi_{q+1} \triangleq& \pi_{l} - \pi_{\text{osc}} - \pi_{\text{corr}} - \pi_{\text{com2}}. \label{estimate 116}
\end{align}
\end{subequations} 
First, we work on $R_{\text{line}}$ from \eqref{estimate 44}. As $m \in (0, \frac{1}{2})$ by hypothesis, for any $\epsilon \in (0, 1-2m)$, by relying on Lemma \ref{[The. 1.4, RS16]} we obtain 
\begin{equation}\label{estimate 35}
\lVert \mathcal{R} (-\Delta)^{m} w_{q+1} \rVert_{C_{t,x}} \overset{\eqref{estimate 145} \eqref{[Equ. (5.29) and (5.29a), BV19b]}}{\lesssim_{\epsilon}} \lVert \mathcal{R} w_{q+1}^{(p)} \rVert_{C_{t}C_{x}^{2m+\epsilon}} + \lVert \mathcal{R} w_{q+1}^{(c)} \rVert_{C_{t}C_{x}^{2m+\epsilon}}. 
\end{equation}
First, using the fact that for all $s \in [0,t]$ fixed, there exist at most two non-trivial cutoffs
\begin{equation}\label{estimate 31}
\lVert \mathcal{R} w_{q+1}^{(p)} \rVert_{C_{t}C_{x}^{2m+\epsilon}} \overset{\eqref{[Equ. (5.25), BV19b]} \eqref{[Equ. (5.24), BV19b]}}{\leq} 2\sup_{j} \sum_{\zeta \in \Lambda_{j}}  \lVert \mathcal{R}(a_{(\zeta)}  W_{(\zeta)} (\Phi_{j})) \rVert_{C_{t}C_{x}^{2m+\epsilon}}. 
\end{equation}
From \eqref{estimate 22} we see that \eqref{[Equ. (5.35), BV19b]} is satisfied by ``$C_{a}$'' = $\delta_{q+1}^{\frac{1}{2}} M_{0}(t)^{\frac{1}{2}} \lVert \gamma_{\zeta} \rVert_{C^{\lceil \frac{1}{2m} \rceil \vee 8}(B_{C_{\ast}} (\text{Id}))}$ for all $0 \leq N \leq \lceil \frac{1}{2m} \rceil \vee 8$. Therefore, by taking $\beta < \frac{1}{3} (1-2m-\epsilon)$ and $a \in \mathbb{N}$ sufficiently large, continuing from \eqref{estimate 31}
\begin{align}
\lVert \mathcal{R} w_{q+1}^{(p)} \rVert_{C_{t}C_{x}^{2m+\epsilon}} \overset{\eqref{[Equ. (5.36), BV19b]} \eqref{[Equ. (5.17), BV19b]}}{\lesssim}& M \delta_{q+1}^{\frac{1}{2}} M_{0}(t)^{\frac{1}{2}} \lambda_{q+1}^{2m+\epsilon -1} \nonumber\\
\approx& c_{R} M_{0}(t) \delta_{q+2}  a^{2^{q+1} [3 \beta + 2m + \epsilon -1]} \ll c_{R} M_{0}(t) \delta_{q+2}. \label{estimate 32}
\end{align} 
Next, because for all $s \in [0,t]$ there exist at most two non-trivial cutoffs, we have 
\begin{equation}
\lVert \mathcal{R} w_{q+1}^{(c)} \rVert_{C_{t}C_{x}^{2m+\epsilon}} \overset{\eqref{[Equ. (5.29) and (5.29a), BV19b]}\eqref{[Equ. (5.28), BV19b]}}{\leq} 2\sup_{j} \sum_{\zeta \in \Lambda_{j}} \lVert \mathcal{R} (( \lambda_{q+1}^{-1} \nabla a_{(\zeta)} + i a_{(\zeta)} \zeta \cdot (\nabla \Phi_{j} - \text{Id} )) \times W_{(\zeta)} (\Phi_{j})) \rVert_{C_{t}C_{x}^{2m+\epsilon}}. \label{estimate 33}
\end{equation} 
Now, for all $N \in\mathbb{N}_{0}$ we can estimate 
\begin{equation}\label{estimate 39}
\lVert a_{(\zeta)} (\nabla \Phi_{j} - \text{Id} ) \rVert_{C_{t}C_{x}^{N}} 
\overset{\eqref{estimate 22} \eqref{[Equ. (5.19a), BV19b]} \eqref{[Equ. (5.19c), BV19b]}}{\lesssim} \delta_{q+1}^{\frac{1}{2}} M_{0}(t)^{\frac{3}{2}} \delta_{q}^{\frac{1}{2}} \lambda_{q} \lVert \gamma_{\zeta} \rVert_{C^{N} (B_{C_{\ast}} (\text{Id} ))} l^{-N+1}
\end{equation} 
and hence deduce for all $N = 0, \hdots, \lceil \frac{1}{2m} \rceil \vee 8$, by taking $a \in \mathbb{N}$ sufficiently large 
\begin{equation}\label{estimate 146}
\lVert \lambda_{q+1}^{-1} \nabla a_{(\zeta)} + i a_{(\zeta)} \zeta \cdot (\nabla \Phi_{j} - \text{Id} ) \rVert_{C_{t}C_{x}^{N}} 
\overset{\eqref{estimate 26} \eqref{estimate 25} }{\lesssim} \lambda_{q+1}^{-1} \delta_{q+1}^{\frac{1}{2}} M_{0}(t)^{\frac{1}{2}} \lVert \gamma_{\zeta} \rVert_{C^{N+1}(B_{C_{\ast}} (\text{Id}))} l^{-N-1}.
\end{equation} 
Therefore, \eqref{estimate 146} shows that \eqref{[Equ. (5.35), BV19b]} holds with ``$C_{a}$''$= \lambda_{q+1}^{-1} \delta_{q+1}^{\frac{1}{2}} M_{0}(t)^{\frac{1}{2}} \lVert \gamma_{\zeta} \rVert_{C^{(\lceil \frac{1}{2m} \rceil + 1)\vee 9}(B_{C_{\ast}} (\text{Id}))} l^{-1}$ for $N= 0, \hdots, (\lceil \frac{1}{2m} \rceil + 1)\vee 8$, so that by \eqref{[Equ. (5.36), BV19b]} we can continue from \eqref{estimate 33} by taking $\beta < \frac{1}{6} ( \frac{5}{2} - 4m - 2 \epsilon)$ and taking $a \in \mathbb{N}$ sufficiently large 
\begin{align}
\lVert \mathcal{R} w_{q+1}^{(c)} \rVert_{C_{t}C_{x}^{2m+\epsilon}} 
\overset{\eqref{[Equ. (5.36), BV19b]}}{\lesssim}& \sup_{j} \sum_{\zeta \in \Lambda_{j}}  \frac{\lambda_{q+1}^{-1} \delta_{q+1}^{\frac{1}{2}} M_{0}(t)^{\frac{1}{2}} \lVert \gamma_{\zeta} \rVert_{C^{(\lceil \frac{1}{2m} \rceil + 1)\vee 9}(B_{C_{\ast}} (\text{Id}))} l^{-1}}{\lambda_{q+1}^{1- (2m+\epsilon)}} \label{estimate 34}\\
\overset{\eqref{[Equ. (5.17), BV19b]}}{\lesssim}& \delta_{q+1}^{\frac{1}{2}} M_{0}(t)^{\frac{1}{2}} l^{-1}\lambda_{q+1}^{-2 + 2m + \epsilon}  \approx c_{R} M_{0}(t) \delta_{q+2} [ a^{2^{q} [ 6 \beta + 4m + 2 \epsilon - \frac{5}{2} ]}] \ll c_{R} M_{0}(t) \delta_{q+2}. \nonumber
\end{align}
Applying \eqref{estimate 32} and \eqref{estimate 34} to \eqref{estimate 35} gives us 
\begin{equation}\label{estimate 36}
\lVert \mathcal{R} (-\Delta)^{m} w_{q+1} \rVert_{C_{t,x}} \ll c_{R} M_{0}(t) \delta_{q+2}.
\end{equation} 
Next, within $R_{\text{line}}$ from \eqref{estimate 44} we first split 
\begin{equation}\label{estimate 37}
\lVert \mathcal{R} ( (w_{q+1} \cdot \nabla) z_{l}) \rVert_{C_{t,x}} \overset{\eqref{[Equ. (5.29) and (5.29a), BV19b]}}{\leq} \lVert \mathcal{R} ( (w_{q+1}^{(p)} \cdot \nabla) z_{l}) \rVert_{C_{t,x}} + \lVert \mathcal{R} ( (w_{q+1}^{(c)} \cdot \nabla) z_{l}) \rVert_{C_{t,x}}. 
\end{equation} 
First, we compute 
\begin{equation}\label{estimate 38}
\lVert \mathcal{R} ( (w_{q+1}^{(p)} \cdot \nabla) z_{l}) \rVert_{C_{t,x}} \overset{\eqref{[Equ. (5.25), BV19b]} \eqref{[Equ. (5.24), BV19b]}}{=} \lVert \mathcal{R} (\sum_{j} \sum_{\zeta \in \Lambda_{j}} a_{(\zeta)} B_{\zeta} e^{i \lambda_{q+1} \zeta \cdot \Phi_{j} (t,x)} \cdot \nabla z_{l}) \rVert_{C_{t,x}}.
\end{equation} 
For any $\epsilon \in (\frac{1}{8}, 1)$, for all $N = 0, \hdots, \lceil \frac{1}{\epsilon} \rceil \vee 8 = 8$, we can estimate 
\begin{equation}
\lVert a_{(\zeta)} \nabla z_{l} \rVert_{C_{t}C_{x}^{N}} \overset{\eqref{estimate 22} \eqref{[Equ. (38), Y20a]} }{\lesssim}    \delta_{q+1}^{\frac{1}{2}} M_{0}(t)^{\frac{1}{2}} \lVert \gamma_{\zeta} \rVert_{C^{N} (B_{C_{\ast}} (\text{Id} ))}l^{-N}L^{\frac{1}{4}}.
\end{equation}
Thus, \eqref{[Equ. (5.35), BV19b]} holds with ``$C_{a}$'' $= \delta_{q+1}^{\frac{1}{2}} M_{0}(t)^{\frac{1}{2}} \lVert \gamma_{\zeta} \rVert_{C^{8} (B_{C_{\ast}} (\text{Id} ))}L^{\frac{1}{4}}$ so that, as for all time $s\in [0,t]$ fixed, there exist at most two non-trivial cutoffs, continuing from \eqref{estimate 38}, choosing $\beta < \frac{1}{3} (1-\epsilon)$ and $a \in \mathbb{N}$ sufficiently large, 
\begin{align}
\lVert \mathcal{R} ((w_{q+1}^{(p)} \cdot \nabla) z_{l}) \rVert_{C_{t,x}} \overset{ \eqref{[Equ. (5.36), BV19b]}}{\lesssim}&  \sup_{j} \sum_{\zeta \in \Lambda_{j}} \frac{ \delta_{q+1}^{\frac{1}{2}} M_{0}(t)^{\frac{1}{2}} \lVert \gamma_{\zeta} \rVert_{C^{8} (B_{C_{\ast}} (\text{Id} ))}L^{\frac{1}{4}}}{\lambda_{q+1}^{1-\epsilon}}  \nonumber\\
\overset{\eqref{[Equ. (5.17), BV19b]}}{\lesssim}& c_{R} M_{0}(t) \delta_{q+2} a^{2^{q+1} (3 \beta + \epsilon -1)} \ll c_{R} M_{0}(t) \delta_{q+2}.  \label{estimate 42}
\end{align}
Second, we use that for all $s \in [0,t]$ there exist at most two non-trivial cutoffs to write 
\begin{align}
&\lVert \mathcal{R} ((w_{q+1}^{(c)} \cdot \nabla) z_{l}) \rVert_{C_{t,x}} \nonumber\\
\overset{\eqref{[Equ. (5.29) and (5.29a), BV19b]} \eqref{[Equ. (5.28), BV19b]}}{\leq}& 2\sup_{j} \sum_{\zeta \in \Lambda_{j}} \lVert \mathcal{R} ( (\lambda_{q+1}^{-1} \nabla a_{(\zeta)} + i a_{(\zeta)} \zeta \cdot (\nabla \Phi_{j}  - \text{Id} ) ) \times W_{(\zeta)} (\Phi_{j}) \cdot \nabla z_{l} )  \rVert_{C_{t,x}}. \label{estimate 40}
\end{align}
For any $\epsilon \in (\frac{1}{8}, 1)$, for all $N = 0, \hdots, \lceil \frac{1}{\epsilon} \rceil \vee 8 = 8$, we can estimate by taking $a \in \mathbb{N}$ sufficiently large  
\begin{align}
& \lVert ( \lambda_{q+1}^{-1} \nabla a_{(\zeta)} + i a_{(\zeta)} \zeta \cdot (\nabla \Phi_{j} - \text{Id} )) \cdot \nabla z_{l} \rVert_{C_{t}C_{x}^{N}} \nonumber \\
\overset{\eqref{estimate 22} \eqref{[Equ. (38), Y20a]} \eqref{estimate 39}}{\lesssim}&  \lambda_{q+1}^{-1} \delta_{q+1}^{\frac{1}{2}} M_{0}(t)^{\frac{1}{2}} \lVert \gamma_{\zeta} \rVert_{C^{N+1} (B_{C_{\ast}}(\text{Id} ))} l^{-N-1} L^{\frac{1}{4}} + \delta_{q+1}^{\frac{1}{2}} M_{0}(t)^{\frac{3}{2}} \delta_{q}^{\frac{1}{2}} \lambda_{q} \lVert \gamma_{\zeta} \rVert_{C^{N} (B_{C_{\ast}} (\text{Id} ))} l^{-N+1}L^{\frac{1}{4}} \nonumber\\
\lesssim& \delta_{q+1}^{\frac{1}{2}} \lVert \gamma_{\zeta} \rVert_{C^{N+1} (B_{C_{\ast}} (\text{Id} ))} L^{\frac{1}{4}} M_{0}(t)^{\frac{1}{2}} \lambda_{q+1}^{-1} l^{-N-1}. 
\end{align} 
Therefore, \eqref{[Equ. (5.35), BV19b]} holds with ``$C_{a}$''= $\delta_{q+1}^{\frac{1}{2}} \lVert \gamma_{\zeta} \rVert_{C^{9} (B_{C_{\ast}} (\text{Id} ))} L^{\frac{1}{4}} M_{0}(t)^{\frac{1}{2}} \lambda_{q+1}^{-1} l^{-1}$ so that choosing $\beta < \frac{1}{6} (\frac{5}{2} - 2 \epsilon)$ and $a \in \mathbb{N}$ sufficiently large, we can continue from \eqref{estimate 40} as 
\begin{align}
\lVert \mathcal{R} ((w_{q+1}^{(c)} \cdot \nabla) z_{l}) \rVert_{C_{t,x}} \overset{\eqref{[Equ. (5.36), BV19b]} \eqref{[Equ. (5.17), BV19b]}}{\lesssim}& M\delta_{q+1}^{\frac{1}{2}} M_{0}(t) \lambda_{q+1}^{\epsilon -2} \lambda_{q}^{\frac{3}{2}} \nonumber \\
\approx& c_{R} M_{0}(t) \delta_{q+2} a^{2^{q} ( 6 \beta + 2 \epsilon - \frac{5}{2})} \ll c_{R} M_{0}(t) \delta_{q+2}. \label{estimate 41}
\end{align} 
Applying \eqref{estimate 42} and \eqref{estimate 41} to  \eqref{estimate 37} 
gives us 
\begin{equation}\label{estimate 43}
\lVert \mathcal{R} ((w_{q+1} \cdot \nabla) z_{l}) \rVert_{C_{t,x}} \ll c_{R}M_{0}(t) \delta_{q+2}. 
\end{equation}
Together with \eqref{estimate 36} and \eqref{estimate 44}, \eqref{estimate 43} allows us to conclude that
\begin{equation}\label{estimate 165}
\lVert R_{\text{line}} \rVert_{C_{t,x}} \overset{\eqref{estimate 44}}{\leq} \lVert \mathcal{R} (-\Delta)^{m} w_{q+1} \rVert_{C_{t,x}} +\lVert \mathcal{R} ( (w_{q+1} \cdot \nabla) z_{l}) \rVert_{C_{t,x}} \overset{\eqref{estimate 36} \eqref{estimate 43}}{\ll} c_{R} M_{0}(t)\delta_{q+2}. 
\end{equation} 
Next, we look at $R_{\text{tran}} = \mathcal{R} ((\partial_{t}  + (v_{l} + z_{l}) \cdot \nabla) w_{q+1}^{(p)})$ in \eqref{estimate 45}. We make the following key observation that the worst term when $\nabla$ falls on $W_{(\zeta)} \circ \Phi_{j}$ vanishes: 
\begin{align}
(\partial_{t} + (v_{l} + z_{l}) \cdot \nabla) w_{q+1}^{(p)} \overset{\eqref{[Equ. (5.24), BV19b]} \eqref{[Equ. (5.25), BV19b]}}{=}& \sum_{j} \sum_{\zeta \in \Lambda_{j}} [ \partial_{t} a_{(\zeta)} + (v_{l} + z_{l}) \cdot \nabla a_{(\zeta)} ] W_{(\zeta)} (\Phi_{j}) \nonumber  \\
&+ a_{(\zeta)} \nabla W_{(\zeta)}(\Phi_{j}) \cdot [ \partial_{t} \Phi_{j} + (v_{l} + z_{l}) \cdot \nabla \Phi_{j} ] \nonumber\\
\overset{\eqref{[Equ. (5.18a), BV19b]}}{=}& \sum_{j} \sum_{\zeta \in \Lambda_{j}} [ \partial_{t} a_{(\zeta)} + (v_{l} + z_{l}) \cdot \nabla a_{(\zeta)} ] W_{(\zeta)} \circ \Phi_{j}.  \label{estimate 54}
\end{align} 
For any $\epsilon \in (\frac{1}{8}, \frac{1}{4})$, for $N = 0, \hdots, \lceil \frac{1}{\epsilon} \rceil \vee 8 = 8$, we estimate
\begin{align}
\lVert (v_{l} + z_{l}) \cdot \nabla a_{(\zeta)} \rVert_{C_{t}C_{x}^{N}} 
\overset{\eqref{estimate 22}}{\lesssim}& l^{-(N-1)} (\lVert v_{q} \rVert_{C_{t}C_{x}^{1}} + \lVert z \rVert_{C_{t}C_{x}^{1}}) \delta_{q+1}^{\frac{1}{2}} M_{0}(t)^{\frac{1}{2}} \lVert \gamma_{\zeta} \rVert_{C^{1} (B_{C_{\ast}} (\text{Id} ))}l^{-1} \nonumber \\
&+ (\lVert v_{q} \rVert_{C_{t,x}} + \lVert z \rVert_{C_{t,x}}) \delta_{q+1}^{\frac{1}{2}} M_{0}(t)^{\frac{1}{2}} \lVert \gamma_{\zeta} \rVert_{C^{N+1} (B_{C_{\ast}} (\text{Id} ))}l^{-N-1} \nonumber \\
\overset{\eqref{[Equ. (40a), Y20a]} \eqref{[Equ. (40b), Y20a]} \eqref{[Equ. (38), Y20a]}}{\lesssim}& \delta_{q+1}^{\frac{1}{2}} M_{0}(t) l^{-N-1} \lVert \gamma_{\zeta} \rVert_{C^{N+1} (B_{C_{\ast}} (\text{Id} ))}. \label{estimate 53}
\end{align} 
Hence, together with \eqref{estimate 23}, for all $N = 0, \hdots, \lceil \frac{1}{\epsilon} \rceil \vee 8 = 8$,  we have 
\begin{equation}
\lVert \partial_{t} a_{(\zeta)} + (v_{l} + z_{l}) \cdot \nabla a_{(\zeta)} \rVert_{C_{t}C_{x}^{N}} \overset{\eqref{estimate 23} \eqref{estimate 53}}{\lesssim}  \delta_{q+1}^{\frac{1}{2}} M_{0}(t) l^{-N-1} \lVert \gamma_{\zeta} \rVert_{C^{N+1} (B_{C_{\ast}} (\text{Id} ))}. 
\end{equation} 
Therefore, \eqref{[Equ. (5.35), BV19b]} is satisfied with ``$C_{a}$'' $= \delta_{q+1}^{\frac{1}{2}} M_{0}(t) l^{-1} \lVert \gamma_{\zeta} \rVert_{C^{9} (B_{C_{\ast}} (\text{Id} ))}$ so that we can take $\beta  < \frac{1}{6} (\frac{1}{2} - 2 \epsilon)$ and $a \in \mathbb{N}$ sufficiently large to compute by \eqref{[Equ. (5.36), BV19b]} 
\begin{align}
\lVert R_{\text{tran}} \rVert_{C_{t,x}} \overset{\eqref{estimate 45}\eqref{estimate 54}}{\lesssim}& \sup_{j} \sum_{\zeta \in \Lambda_{j}} \lVert \mathcal{R} ( [ \partial_{t} a_{(\zeta)} + (v_{l} + z_{l}) \cdot \nabla a_{(\zeta)}] W_{(\zeta)} \circ \Phi_{j}) \rVert_{C_{t}C_{x}^{\epsilon}} \nonumber\\
\overset{\eqref{[Equ. (5.36), BV19b]}\eqref{[Equ. (5.17), BV19b]}}{\lesssim}& \delta_{q+1}^{\frac{1}{2}} M_{0}(t) \lambda_{q}^{\frac{3}{2}} \lambda_{q+1}^{\epsilon -1}  \approx c_{R} M_{0}(t) \delta_{q+2} a^{2^{q} (6 \beta - \frac{1}{2} + 2 \epsilon)} \ll c_{R} M_{0}(t) \delta_{q+2}. \label{estimate 166}
\end{align}
Next, we work on $R_{\text{osc}}$ from \eqref{estimate 46}: by relying on the identities of 
\begin{equation*}
 \nabla (\phi_{(\zeta)} \phi_{(\zeta')}) 
\overset{\eqref{[Equ. (5.27), BV19b]}}{=} i \lambda_{q+1} \zeta \cdot (\nabla \Phi_{j} - \text{Id} ) \phi_{(\zeta)} \phi_{(\zeta')} + i \lambda_{q+1} \zeta' \cdot (\nabla \Phi_{j'} - \text{Id} ) \phi_{(\zeta)} \phi_{(\zeta')} , W_{(\zeta)} \phi_{(\zeta)} \overset{\eqref{[Equ. (5.12), BV19b]} \eqref{[Equ. (5.27), BV19b]}}{=} W_{(\zeta)} \circ \Phi_{j} 
\end{equation*} 
for $\zeta \in \Lambda_{j}, \zeta' \in \Lambda_{j'}$, we can rewrite 
\begin{align}
R_{\text{osc}} \overset{\eqref{estimate 46}}{=}& \sum_{j, j'} \sum_{\zeta \in \Lambda_{j}, \zeta ' \in \Lambda_{j'}: \zeta + \zeta' \neq 0} \mathcal{R} (( W_{(\zeta)} \circ \Phi_{j} \otimes W_{(\zeta')} \circ \Phi_{j'} - \frac{ W_{(\zeta)} \circ \Phi_{j} \cdot W_{(\zeta')} \circ \Phi_{j'}}{2} \text{Id}) \nonumber\\
& \times [ \nabla (a_{(\zeta)} a_{(\zeta')}) + a_{(\zeta)} a_{(\zeta')} [ i \lambda_{q+1} \zeta \cdot (\nabla \Phi_{j} - \text{Id}) + i \lambda_{q+1} \zeta' \cdot (\nabla \Phi_{j'} - \text{Id} )]]).  \label{estimate 55}
\end{align} 
Now for any $\epsilon \in (\frac{1}{8}, \frac{1}{4})$, for all $N = 0, \hdots, \lceil \frac{1}{\epsilon} \rceil \vee 8 = 8$, by taking $a \in \mathbb{N}$ sufficiently large we obtain 
\begin{align}
& \lVert \nabla (a_{(\zeta)} a_{(\zeta')} ) + a_{(\zeta)} a_{(\zeta')} [i\lambda_{q+1} \zeta \cdot (\nabla \Phi_{j} - \text{Id} ) + i \lambda_{q+1} \zeta' \cdot (\nabla \Phi_{j'} - \text{Id} )] \rVert_{C_{t}C_{x}^{N}} \nonumber \\
\lesssim&  \lVert \nabla (a_{(\zeta)} a_{(\zeta')}) \rVert_{C_{t}C_{x}^{N}}  +  \lambda_{q+1} \lVert a_{(\zeta)} a_{(\zeta')} [ \zeta \cdot (\nabla \Phi_{j} - \text{Id} ) + \zeta' \cdot (\nabla \Phi_{j'} - \text{Id} )] \rVert_{C_{t}C_{x}^{N}} \nonumber\\
\overset{\eqref{estimate 22}\eqref{[Equ. (5.19a), BV19b]} \eqref{[Equ. (5.19c), BV19b]} }{\lesssim}& \delta_{q+1} M_{0}(t) \lVert \gamma_{\zeta} \rVert_{C^{N+1} (B_{C_{\ast}} (\text{Id}))} \lVert \gamma_{\zeta'} \rVert_{C^{N+1} (B_{C_{\ast}}(\text{Id}))} l^{-N-1} \nonumber \\
&+ \lambda_{q+1} \delta_{q+1} M_{0}(t)^{2} \lVert \gamma_{\zeta} \rVert_{C^{N} (B_{C_{\ast}} (\text{Id}))} \lVert \gamma_{\zeta'} \rVert_{C^{N} (B_{C_{\ast}} (\text{Id} ))} l^{-N+1} \delta_{q}^{\frac{1}{2}} \lambda_{q} \nonumber \\
\lesssim& \delta_{q+1} M_{0}(t) \lVert \gamma_{\zeta} \rVert_{C^{N+1} (B_{C_{\ast}} (\text{Id}))} \lVert \gamma_{\zeta'} \rVert_{C^{N+1} (B_{C_{\ast}} (\text{Id} ))} l^{-N-1}.  
\end{align} 
Therefore, \eqref{[Equ. (5.35), BV19b]} is satisfied with ``$C_{a}$''$=  \delta_{q+1} M_{0}(t) \lVert \gamma_{\zeta} \rVert_{C^{9} (B_{C_{\ast}} (\text{Id}))} \lVert \gamma_{\zeta'} \rVert_{C^{9} (B_{C_{\ast}} (\text{Id} ))} l^{-1}$. Hence,   we can choose $\beta < \frac{1}{4}( \frac{1}{2} - 2 \epsilon)$, as well as $a \in \mathbb{N}$ sufficiently large, continue  from \eqref{estimate 55}, use the fact that for any $s \in [0,t]$ fixed there exist at most two non-trivial cutoffs, and compute 
\begin{align}
\lVert R_{\text{osc}} \rVert_{C_{t,x}} \overset{\eqref{estimate 55} \eqref{[Equ. (5.37), BV19b]}}{\lesssim}& \sup_{j,j'} \sum_{\zeta \in \Lambda_{j}, \zeta' \in \Lambda_{j'}: \zeta + \zeta' \neq 0} \frac{ \delta_{q+1} M_{0}(t) \lVert \gamma_{\zeta} \rVert_{C^{9} (B_{C_{\ast}} (\text{Id}))} \lVert \gamma_{\zeta'} \rVert_{C^{9} (B_{C_{\ast}} (\text{Id} ))} l^{-1}}{\lambda_{q+1}^{1-\epsilon}} \nonumber \\
\overset{\eqref{[Equ. (5.17), BV19b]}}{\lesssim}&  \delta_{q+1} M_{0}(t) l^{-1} \lambda_{q+1}^{\epsilon -1}  \approx c_{R} M_{0}(t) \delta_{q+2} a^{2^{q} (4 \beta - \frac{1}{2} + 2 \epsilon)} \ll c_{R} M_{0}(t) \delta_{q+2}. \label{estimate 114}
\end{align} 
Next, we rewrite $R_{\text{Nash}}$ from \eqref{estimate 48} as follows:
\begin{align}
R_{\text{Nash}} \overset{\eqref{estimate 48} \eqref{[Equ. (5.29) and (5.29a), BV19b]}\eqref{[Equ. (5.24), BV19b]} \eqref{[Equ. (5.28), BV19b]}}{=}& \sum_{j}\sum_{\zeta \in \Lambda_{j}} \mathcal{R} ( ( a_{(\zeta)} W_{(\zeta)} \circ \Phi_{j} \nonumber\\
& \hspace{12mm} + ( \lambda_{q+1}^{-1} \nabla a_{(\zeta)} + ia_{(\zeta)} \zeta \cdot (\nabla \Phi_{j} - \text{Id} )) \times W_{(\zeta)} (\Phi_{j})) \cdot \nabla v_{l} ). \label{estimate 58}
\end{align}
Now for any $\epsilon \in (\frac{1}{8}, \frac{1}{2})$, for all $N = 0, \hdots, \lceil \frac{1}{\epsilon} \rceil \vee 8 = 8$ we can estimate 
\begin{equation}\label{estimate 56}
\lVert a_{(\zeta)} \cdot \nabla v_{l} \rVert_{C_{t}C_{x}^{N}} \overset{\eqref{estimate 22}\eqref{[Equ. (40b), Y20a]}}{\lesssim} \delta_{q+1}^{\frac{1}{2}} M_{0}(t)^{\frac{3}{2}} \lVert \gamma_{\zeta} \rVert_{C^{N} (B_{C_{\ast}} (\text{Id} ))}l^{-N}  \delta_{q}^{\frac{1}{2}} \lambda_{q}.
\end{equation} 
On the other hand, for all $N = 0, \hdots, \lceil \frac{1}{\epsilon} \rceil \vee 8 = 8$, we can estimate by taking $a \in \mathbb{N}$ sufficiently large, 
\begin{align}
& \lVert (\lambda_{q+1}^{-1} \nabla a_{(\zeta)} + i a_{(\zeta)} \zeta \cdot (\nabla \Phi_{j} - \text{Id} )) \cdot \nabla v_{l} \rVert_{C_{t}C_{x}^{N}} \nonumber\\
\overset{\eqref{[Equ. (5.10), BV19b]} \eqref{estimate 22} \eqref{[Equ. (5.19a), BV19b]}\eqref{[Equ. (5.19c), BV19b]}}{\lesssim}& \lVert \gamma_{\zeta} \rVert_{C^{N+1} (B_{C_{\ast}} (\text{Id} ))} \lambda_{q+1}^{-1} \delta_{q+1}^{\frac{1}{2}} M_{0}(t)^{\frac{3}{2}} l^{-N-1} \delta_{q}^{\frac{1}{2}} \lambda_{q}.  \label{estimate 57}
\end{align} 
Therefore, for all $N = 0, \hdots, \lceil \frac{1}{\epsilon} \rceil \vee 8 = 8$, 
\begin{align}
& \lVert a_{(\zeta)} \cdot \nabla v_{l} \rVert_{C_{t}C_{x}^{N}} + \lVert ( \lambda_{q+1}^{-1} \nabla a_{(\zeta)} + ia_{(\zeta)} \zeta \cdot (\nabla \Phi_{j} - \text{Id} )) \cdot \nabla v_{l} \rVert_{C_{t}C_{x}^{N}} \nonumber\\
\overset{\eqref{estimate 56}\eqref{estimate 57}}{\lesssim}& \lVert \gamma_{\zeta} \rVert_{C^{9} (B_{C_{\ast}} ( \text{Id}))} \delta_{q+1}^{\frac{1}{2}} M_{0}(t)^{\frac{3}{2}} \delta_{q}^{\frac{1}{2}} \lambda_{q} l^{-N}.
\end{align} 
Hence, \eqref{[Equ. (5.35), BV19b]} is satisfied with ``$C_{a}$'' =$\lVert \gamma_{\zeta} \rVert_{C^{9} (B_{C_{\ast}} ( \text{Id}))} \delta_{q+1}^{\frac{1}{2}} M_{0}(t)^{\frac{3}{2}} \delta_{q}^{\frac{1}{2}} \lambda_{q}$ so that by \eqref{[Equ. (5.36), BV19b]}, choosing $\beta < \frac{1}{5} (1- 2 \epsilon)$ and $a \in \mathbb{N}$ sufficiently large, continuing from \eqref{estimate 58} and taking advantage of the fact that for any $s \in [0,t]$ fixed there exist at most two non-trivial cutoffs give 
\begin{align}
& \lVert R_{\text{Nash}} \rVert_{C_{t,x}} \nonumber \\
\overset{ \eqref{estimate 58}}{\lesssim}& \sup_{j}  \sum_{\zeta \in \Lambda_{j}} \lVert \mathcal{R} (( a_{(\zeta)} W_{(\zeta)} \circ \Phi_{j} + ( \lambda_{q+1}^{-1} \nabla a_{(\zeta)} + ia_{(\zeta)}  \zeta \cdot (\nabla \Phi_{j} - \text{Id} )) \times W_{(\zeta)} (\Phi_{j} )) \cdot \nabla v_{l} ) \rVert_{C_{t}C_{x}^{\epsilon}} \nonumber \\
\overset{\eqref{[Equ. (5.36), BV19b]} \eqref{[Equ. (5.17), BV19b]}}{\lesssim}& M \delta_{q+1}^{\frac{1}{2}} M_{0}(t)^{\frac{3}{2}} \delta_{q}^{\frac{1}{2}} \lambda_{q} \lambda_{q+1}^{\epsilon -1}\approx c_{R} M_{0}(t) \delta_{q+2} [M_{0}(t)^{\frac{1}{2}} a^{2^{q} (5 \beta - 1 + 2 \epsilon) } ] \ll c_{R} M_{0}(t) \delta_{q+2}.  \label{estimate 121}
\end{align} 
Next, we work on $R_{\text{corr}}$ from \eqref{estimate 49}. First, we again make the important observation that 
\begin{align}
& ( \partial_{t} + (v_{l} + z_{l}) \cdot \nabla) w_{q+1}^{(c)} \nonumber\\
\overset{\eqref{[Equ. (5.29) and (5.29a), BV19b]}\eqref{[Equ. (5.28), BV19b]}}{=}& ( \partial_{t} + (v_{l} + z_{l}) \cdot \nabla) \sum_{j} \sum_{\zeta \in \Lambda_{j}} [ ( \lambda_{q+1}^{-1} \nabla a_{(\zeta)} + ia_{(\zeta)} \zeta \cdot  (\nabla \Phi_{j} - \text{Id} )) \times W_{(\zeta)} (\Phi_{j})] \nonumber\\
\overset{\eqref{[Equ. (5.18a), BV19b]}}{=}& \sum_{j} \sum_{\zeta \in \Lambda_{j}} (\partial_{t} + (v_{l} + z_{l}) \cdot \nabla) ( \lambda_{q+1}^{-1} \nabla a_{(\zeta)} + i a_{(\zeta)} \zeta \cdot (\nabla \Phi_{j} - \text{Id} )) \times W_{(\zeta)} (\Phi_{j}). \label{estimate 61}
\end{align} 
For any $\epsilon \in (\frac{1}{8}, \frac{1}{2})$, for all $N = 0, \hdots, \lceil \frac{1}{\epsilon} \rceil \vee 8 = 8$, we can estimate 
\begin{align}
& \lVert ( \partial_{t} + (v_{l} + z_{l}) \cdot \nabla) ( \lambda_{q+1}^{-1} \nabla a_{(\zeta)} + i a_{(\zeta)} \zeta \cdot ( \nabla \Phi_{j} - \text{Id} )) \rVert_{C_{t}C_{x}^{N}} \label{estimate 60}\\
\lesssim& \lambda_{q+1}^{-1} \lVert \nabla \partial_{t} a_{(\zeta)} \rVert_{C_{t}C_{x}^{N}} + \lVert \partial_{t} a_{(\zeta)} (\nabla \Phi_{j} - \text{Id} ) \rVert_{C_{t}C_{x}^{N}} + \lVert a_{(\zeta)} \partial_{t} \nabla \Phi_{j} \rVert_{C_{t}C_{x}^{N}} + \lambda_{q+1}^{-1} \lVert (v_{l} + z_{l}) \cdot \nabla^{2} a_{(\zeta)} \rVert_{C_{t}C_{x}^{N}} \nonumber\\
& \hspace{37mm} + \lVert (v_{l} + z_{l}) \cdot \nabla a_{(\zeta)} (\nabla \Phi_{j} - \text{Id} ) \rVert_{C_{t}C_{x}^{N}} + \lVert (v_{l} + z_{l}) a_{(\zeta)} \cdot \nabla^{2} \Phi_{j} \rVert_{C_{t}C_{x}^{N}}. \nonumber
\end{align} 
We can estimate separately for all $N = 0, \hdots, \lceil \frac{1}{\epsilon} \rceil \vee 8 = 8$, by taking $a \in \mathbb{N}$ sufficiently large 
\begin{subequations}\label{estimate 59}
\begin{align}
& \lambda_{q+1}^{-1} \lVert \nabla \partial_{t} a_{(\zeta)} \rVert_{C_{t}C_{x}^{N}} \overset{\eqref{estimate 23}}{\lesssim} \lambda_{q+1}^{-1}  \delta_{q+1}^{\frac{1}{2}} M_{0}(t)^{\frac{1}{2}} \lVert \gamma_{\zeta} \rVert_{C^{N+2} (B_{C_{\ast}} (\text{Id} ))} l^{-N-2}, \\
&  \lVert \partial_{t} a_{(\zeta)} ( \nabla \Phi_{j} - \text{Id} ) \rVert_{C_{t}C_{x}^{N}} \overset{\eqref{estimate 23} \eqref{[Equ. (5.19a), BV19b]}\eqref{[Equ. (5.19c), BV19b]} }{\lesssim} \delta_{q+1}^{\frac{1}{2}} M_{0}(t)^{\frac{3}{2}} \lVert \gamma_{\zeta} \rVert_{C^{N+1} (B_{C_{\ast}} (\text{Id} ))}l^{-N} \delta_{q}^{\frac{1}{2}} \lambda_{q}, \\
& \lVert a_{(\zeta)} \partial_{t} \nabla \Phi_{j} \rVert_{C_{t}C_{x}^{N}} \overset{\eqref{estimate 22} \eqref{estimate 19}}{\lesssim} \delta_{q+1}^{\frac{1}{2}} M_{0}(t)^{2} \lVert \gamma_{\zeta} \rVert_{C^{N} (B_{C_{\ast}} (\text{Id} ))} l^{-N} \delta_{q}^{\frac{1}{2}} \lambda_{q}, \\
& \lambda_{q+1}^{-1} \lVert (v_{l} + z_{l} ) \cdot \nabla^{2} a_{(\zeta)} \rVert_{C_{t}C_{x}^{N}} \overset{\eqref{[Equ. (40a), Y20a]} \eqref{[Equ. (40b), Y20a]} \eqref{[Equ. (38), Y20a]} \eqref{estimate 22} }{\lesssim} \lambda_{q+1}^{-1} \delta_{q+1}^{\frac{1}{2}} M_{0}(t) \lVert \gamma_{\zeta} \rVert_{C^{N+2} (B_{C_{\ast}} (\text{Id} ))} l^{-N-2},  \\
& \lVert (v_{l} + z_{l}) \cdot \nabla a_{(\zeta)} (\nabla \Phi_{j} - \text{Id} ) \rVert_{C_{t}C_{x}^{N}}   \overset{\eqref{estimate 22} \eqref{estimate 25}  \eqref{[Equ. (40), Y20a]} \eqref{[Equ. (38), Y20a]}}{\lesssim} \lVert \gamma_{\zeta} \rVert_{C^{N+1} (B_{C_{\ast}} (\text{Id} ))} \delta_{q+1}^{\frac{1}{2}} \delta_{q}^{\frac{1}{2}} \lambda_{q} l^{-N} M_{0}(t)^{2},  \\
& \lVert (v_{l}+  z_{l}) a_{(\zeta)} \cdot \nabla^{2} \Phi_{j} \rVert_{C_{t}C_{x}^{N}} \overset{\eqref{estimate 22} \eqref{[Equ. (5.19c), BV19b]}\eqref{[Equ. (40a), Y20a]} \eqref{[Equ. (40b), Y20a]} \eqref{[Equ. (38), Y20a]}}{\lesssim} \lVert \gamma_{\zeta} \rVert_{C^{N} (B_{C_{\ast}} (\text{Id} ))} \delta_{q+1}^{\frac{1}{2}} \delta_{q}^{\frac{1}{2}} \lambda_{q} l^{-N} M_{0}(t)^{2}.
\end{align} 
\end{subequations} 
Applying \eqref{estimate 59} to \eqref{estimate 60} gives us by taking $a \in \mathbb{N}$ sufficiently large, for all $N = 0, \hdots, \lceil \frac{1}{\epsilon} \rceil \vee 8 = 8$, 
\begin{align}
& \lVert ( \partial_{t} + (v_{l} + z_{l}) \cdot \nabla) ( \lambda_{q+1}^{-1} \nabla a_{(\zeta)} + i a_{(\zeta)} \zeta \cdot ( \nabla \Phi_{j} - \text{Id} )) \rVert_{C_{t}C_{x}^{N}}  \\
\lesssim& \lVert \gamma_{\zeta} \rVert_{C^{N+2} (B_{C_{\ast}} (\text{Id} ))} \delta_{q+1}^{\frac{1}{2}}l^{-N} [ \lambda_{q+1}^{-1} M_{0}(t) l^{-2} +  M_{0}(t)^{2} \delta_{q}^{\frac{1}{2}} \lambda_{q} ] 
\lesssim \lVert \gamma_{\zeta} \rVert_{C^{N+2} (B_{C_{\ast}} (\text{Id} ))} \delta_{q+1}^{\frac{1}{2}} l^{-N-2} \lambda_{q+1}^{-1} M_{0}(t). \nonumber
\end{align} 
This implies that \eqref{[Equ. (5.35), BV19b]} holds with ``$C_{a}$'' = $\lVert \gamma_{\zeta} \rVert_{C^{10} (B_{C_{\ast}} (\text{Id} ))} \delta_{q+1}^{\frac{1}{2}} l^{-2} \lambda_{q+1}^{-1} M_{0}(t)$ and hence via \eqref{[Equ. (5.36), BV19b]},  continuing from \eqref{estimate 61}, taking $\beta < \frac{1}{6} (1-2\epsilon)$ and $a \in \mathbb{N}$ sufficiently large, we obtain 
\begin{align}
& \lVert \mathcal{R} ( ( \partial_{t} + (v_{l} + z_{l}) \cdot \nabla) w_{q+1}^{(c)} ) \rVert_{C_{t,x}} \nonumber\\
\overset{\eqref{estimate 61}}{\lesssim}& \sup_{j} \sum_{\zeta \in \Lambda_{j}} \lVert \mathcal{R} ( ( \partial_{t} + (v_{l} + z_{l}) \cdot \nabla) ( \lambda_{q+1}^{-1} \nabla a_{(\zeta)} + ia_{(\zeta)} \zeta \cdot (\nabla \Phi_{j} - \text{Id} ) ) \times W_{(\zeta)} (\Phi_{j}) ) \rVert_{C_{t}C_{x}^{\epsilon}} \nonumber \\
\overset{\eqref{[Equ. (5.36), BV19b]}\eqref{[Equ. (5.17), BV19b]}}{\lesssim}& \delta_{q+1}^{\frac{1}{2}} l^{-2} \lambda_{q+1}^{\epsilon -2} M_{0}(t) \approx c_{R} M_{0}(t) \delta_{q+2}  a^{2^{q } (6\beta -1 + 2 \epsilon)} \ll c_{R} M_{0}(t) \delta_{q+2}. \label{estimate 64}
\end{align} 
Next, within $R_{\text{corr}}$ from \eqref{estimate 49}, we can directly estimate by taking $\beta < \frac{1}{8}$ and $a \in \mathbb{N}$ sufficiently large, 
\begin{align}
& \lVert w_{q+1}^{(c)} \mathring{\otimes} w_{q+1} + w_{q+1}^{(p)} \mathring{\otimes} w_{q+1}^{(c)} \rVert_{C_{t,x}} 
\overset{\eqref{[Equ. (5.29) and (5.29a), BV19b]}}{\lesssim} \lVert w_{q+1}^{(c)} \rVert_{C_{t,x}} ( \lVert w_{q+1}^{(c)} \rVert_{C_{t,x}} + \lVert w_{q+1}^{(p)} \rVert_{C_{t,x}}) \label{estimate 63}\\
& \hspace{18mm} \overset{\eqref{[Equ. (5.31), BV19b]} \eqref{[Equ. (5.26), BV19b]} }{\lesssim} \delta_{q+1} M_{0}(t) \lambda_{q+1}^{-1} \lambda_{q}^{\frac{3}{2}} ( \lambda_{q+1}^{-1} \lambda_{q}^{\frac{3}{2}} + 1) \approx c_{R} M_{0}(t) \delta_{q+2} a^{2^{q} (4 \beta - \frac{1}{2})} \ll c_{R} M_{0}(t) \delta_{q+2}. \nonumber
\end{align} 
Applying \eqref{estimate 64} and \eqref{estimate 63} to \eqref{estimate 49} gives us 
\begin{equation}\label{estimate 65}
\lVert R_{\text{corr}} \rVert_{C_{t,x}} 
\overset{\eqref{estimate 49} \eqref{estimate 64}\eqref{estimate 63}}{\ll} c_{R} M_{0}(t) \delta_{q+2}.
\end{equation}
Next, the estimate of $R_{\text{com1}}$ in \eqref{[Equ. (60a), Y20a]} can be achieved by standard commutator estimates (e.g., \cite[Pro. 6.5]{BV19b} or \cite[Equ. (5)]{CDS12}) as follows: for $\delta \in (0, \frac{1}{24})$ so that $\frac{1}{2} - 2 \delta < \frac{5}{12}$, and $\beta < \frac{5}{182}$, by taking $a \in \mathbb{N}$ sufficiently large we can compute 
\begin{align}
\lVert R_{\text{com1}} \rVert_{C_{t,x}} \lesssim& l( \lVert v_{q} \rVert_{C_{t,x}^{1}} + \lVert z \rVert_{C_{t}C_{x}^{1}}) ( \lVert v_{q} \rVert_{C_{t,x}} + \lVert z \rVert_{C_{t,x}}) \nonumber \\
&+ l^{\frac{1}{2} - 2 \delta} ( \lVert v_{q} \rVert_{C_{t,x}}^{\frac{1}{2} + 2 \delta} \lVert v_{q}\rVert_{C_{t}^{1}C_{x}}^{\frac{1}{2} - 2\delta} + \lVert z \rVert_{C_{t}^{\frac{1}{2} - 2 \delta} C_{x}}) (\lVert v_{q} \rVert_{C_{t,x}} + \lVert z \rVert_{C_{t,x}}) \nonumber\\
&\overset{\eqref{[Equ. (38), Y20a]} \eqref{[Equ. (40a), Y20a]} \eqref{[Equ. (40b), Y20a]}}{\lesssim} c_{R} \delta_{q+2}  M_{0}(t)^{\frac{5}{4} - \delta} a^{2^{q} [ 8 \beta - (\frac{1}{2} + \beta) (\frac{5}{12})]} \ll c_{R} \delta_{q+2} M_{0}(t). 
\end{align}
Finally, for $\beta < \frac{5}{64}$, taking $a \in \mathbb{N}$ sufficiently large we can directly estimate $R_{\text{com2}}$ from \eqref{estimate 51} as follows: as $\frac{1}{2} - 2 \delta > \frac{5}{12}$ for $\delta \in (0, \frac{1}{24})$, 
\begin{align}
 \lVert R_{\text{com2}} \rVert_{C_{t,x}} \overset{\eqref{estimate 51}}{\lesssim}& ( \lVert v_{q+1} \rVert_{C_{t,x}} + \lVert z \rVert_{C_{t,x}} + \lVert z_{l} \rVert_{C_{t,x}}) \lVert z - z_{l} \rVert_{C_{t,x}} \overset{\eqref{[Equ. (38), Y20a]}\eqref{estimate 66} }{\lesssim} M_{0}(t)^{\frac{1}{2}} (l L^{\frac{1}{4}} + l^{\frac{1}{2} - 2 \delta} L^{\frac{1}{2}})  \nonumber\\
\lesssim& M_{0}(t)^{\frac{1}{2}} l^{\frac{1}{3}} L^{\frac{1}{2}} \lesssim  c_{R} M_{0}(t) \delta_{q+2} a^{2^{q} (8\beta - \frac{5}{8} )} \ll c_{R} M_{0}(t) \delta_{q+2}. \label{estimate 167}
\end{align} 
Applying \eqref{estimate 165},  \eqref{estimate 166}, \eqref{estimate 114}, \eqref{estimate 121}, \eqref{estimate 65}-\eqref{estimate 167} to \eqref{estimate 115} shows that \eqref{[Equ. (40c), Y20a]} at level $q+1$ holds.  

At last, following similar arguments in \cite{HZZ19} we comment on how $(v_{q+1}, \mathring{R}_{q+1})$ is $(\mathcal{F}_{t})_{t\geq 0}$-adapted and that $v_{q+1}(0,x)$ and $\mathring{R}_{q+1}(0,x)$ are deterministic if $v_{q}(0,x)$ and $\mathring{R}_{q}(0,x)$ are deterministic.  First, $z(t)$ from \eqref{[Equ. (24), Y20a]} is $(\mathcal{F}_{t})_{t\geq 0}$-adapted. Due to the compact support of $\varphi_{l}$ in $\mathbb{R}_{+}$, it follows that $z_{l}$ is $(\mathcal{F}_{t})_{t\geq 0}$-adapted. Similarly, because $(v_{q}, \mathring{R}_{q})$ are both $(\mathcal{F}_{t})_{t\geq 0}$-adapted by hypothesis, so are $(v_{l}, \mathring{R}_{l})$. Because $M_{0}(t)$ from \eqref{estimate 69}, $\chi_{j}(t)$ from \eqref{[Equ. (5.20), BV19b]}, and $\gamma_{\zeta}$ from Lemma \ref{[Pro. 5.6, BV19b]} are deterministic, $a_{(\zeta)}$ from \eqref{[Equ. (5.22), BV19b]} is also $(\mathcal{F}_{t})_{t\geq 0}$-adapted. It follows that $w_{(\zeta)}^{(p)}$ from \eqref{[Equ. (5.24), BV19b]} is $(\mathcal{F}_{t})_{t\geq 0}$-adapted and consequently so are $w_{q+1}^{(p)}$ and $\partial_{t} w_{q+1}^{(p)}$ from \eqref{[Equ. (5.25), BV19b]}. Similarly, $w_{(\zeta)}^{(c)}$ from \eqref{estimate 24} and hence in turn $w_{q+1}^{(c)}$ and $\partial_{t} w_{q+1}^{(c)}$ from \eqref{[Equ. (5.29) and (5.29a), BV19b]} are also $(\mathcal{F}_{t})_{t\geq 0}$-adapted. Therefore, $w_{q+1}$ from \eqref{[Equ. (5.29) and (5.29a), BV19b]}  is $(\mathcal{F}_{t})_{t\geq 0}$-adapted, indicating that $v_{q+1}$ from \eqref{[Equ. (5.32), BV19b]} is $(\mathcal{F}_{t})_{t\geq 0}$-adapted. It follows that all of $R_{\text{line}}$, $R_{\text{tran}}$, $R_{\text{osc}}$, $R_{\text{Nash}}$, $R_{\text{corr}}$, and $R_{\text{com2}}$ from \eqref{estimate 148}, and $R_{\text{com1}}$ from \eqref{[Equ. (60a), Y20a]} are $(\mathcal{F}_{t})_{t\geq 0}$-adapted and consequently so is $\mathring{R}_{q+1}$ from \eqref{estimate 115}. 

Similarly, due to the compact support of $\varphi_{l}$ in $\mathbb{R}_{+}$, if $v_{q}(0,x)$ and $\mathring{R}_{q}(0,x)$ are deterministic, then so are $v_{l}(0,x), \mathring{R}_{l}(0,x)$, and $\partial_{t} \mathring{R}_{l} (0,x)$. Similarly,  $z_{l} (0,x)$ is also deterministic because $z(0,x) \equiv 0$ by \eqref{[Equ. (24), Y20a]}. Because $M_{0}(t)$, $\gamma_{\zeta}$, and $\chi_{j}$ are deterministic, it follows that $a_{(\zeta)} (0,x)$ and $\partial_{t} a_{(\zeta)}(0,x)$ from \eqref{[Equ. (5.22), BV19b]} are both deterministic. It follows that $w_{(\zeta)}^{(p)} (0,x)$ from \eqref{[Equ. (5.24), BV19b]} is deterministic; therefore, $w_{q+1}^{(p)} (0,x)$ and $\partial_{t} w_{q+1}^{(p)} (0,x)$ from \eqref{[Equ. (5.25), BV19b]} are deterministic. Similarly, $w_{(\zeta)}^{(c)} (0,x)$ from \eqref{estimate 24} is deterministic and hence so is $w_{q+1}^{(c)} (0,x)$, as well as $\partial_{t} w_{q+1}^{(c)} (0,x)$ from \eqref{[Equ. (5.29) and (5.29a), BV19b]}. This implies that $w_{q+1}(0,x)$ from \eqref{[Equ. (5.29) and (5.29a), BV19b]} is also deterministic and thus so is $v_{q+1}(0,x)$ from \eqref{[Equ. (5.32), BV19b]}. Finally, all of $R_{\text{line}}(0,x)$, $R_{\text{tran}}(0,x)$, $R_{\text{osc}} (0,x)$, $R_{\text{Nash}}(0,x)$, $R_{\text{corr}}(0,x)$, and  $R_{\text{com2}}(0,x)$ from \eqref{estimate 148}, and $R_{\text{com1}}(0,x)$ from \eqref{[Equ. (60a), Y20a]} are all deterministic and hence so is $\mathring{R}_{q+1}(0,x)$ from \eqref{estimate 115}. This completes the proof of Proposition \ref{[Pro. 4.8, Y20a]}

\section{Proofs of Theorems \ref{Theorem 2.3}-\ref{Theorem 2.4}}\label{Proofs Theorems 2.3-2.4}

\subsection{Proof of Theorem \ref{Theorem 2.4} assuming Theorem \ref{Theorem 2.3}}
We recall $U_{1}, \bar{\Omega},$ and $\bar{\mathcal{B}}_{t}$ from Section \ref{Preliminaries}, fix any $\gamma \in (0,1)$, and state the definition of a probabilistically weak solution:
\begin{define}\label{[Def. 5.1, Y20a]}
Let $s \geq 0$, $\xi^{\text{in}} \in L_{\sigma}^{2}$, and $\theta^{\text{in}} \in U_{1}$. Then $P \in \mathcal{P} (\bar{\Omega})$ is a probabilistically weak solution to \eqref{stochastic GNS} with initial condition $(\xi^{\text{in}}, \theta^{\text{in}})$ at initial time $s$ if  
\begin{itemize}
\item [] (M1) $P(\{ \xi(t) = \xi^{\text{in}}, \theta(t) = \theta^{\text{in}} \hspace{1mm} \forall \hspace{1mm} t \in [0,s]\}) = 1$ and for all $l \in \mathbb{N}$ 
\begin{equation}\label{[Equation (4.48r), HZZ19]}
P ( \{ (\xi, \theta) \in \bar{\Omega}: \int_{0}^{l} \lVert G(\xi(r)) \rVert_{L_{2} (U, L_{\sigma}^{2})}^{2} dr < \infty \} ) = 1, 
\end{equation} 
\item [] (M2) under $P$, $\theta$ is a cylindrical $(\bar{\mathcal{B}}_{t})_{t\geq s}$-Wiener process on $U$ starting from initial condition $\theta^{\text{in}}$ at initial time $s$ and for every $\psi_{i} \in C^{\infty} (\mathbb{T}^{3}) \cap L_{\sigma}^{2}$ and $t\geq s$, 
\begin{equation}\label{[Equation (4.48s), HZZ19]}
\langle \xi(t) - \xi(s), \psi_{i} \rangle + \int_{s}^{t} \langle \text{div} (\xi(r) \otimes \xi(r)) + (-\Delta)^{m} \xi(r), \psi_{i} \rangle dr = \int_{s}^{t} \langle \psi_{i}, G(\xi(r)) d\theta(r) \rangle, 
\end{equation} 
\item [] (M3) for any $q \in \mathbb{N}$, there exists a function $t \mapsto C_{t,q} \in \mathbb{R}_{+}$ such that for all $t \geq s$, 
\begin{equation}\label{[Equation (4.48t), HZZ19]}
\mathbb{E}^{P} [ \sup_{r \in [0,t]} \lVert \xi(r) \rVert_{L_{x}^{2}}^{2q} + \int_{s}^{t} \lVert \xi(r) \rVert_{\dot{H}_{x}^{\gamma}}^{2} dr] \leq C_{t,q} (1+ \lVert \xi^{\text{in}} \rVert_{L_{x}^{2}}^{2q}). 
\end{equation}
\end{itemize}
The set of all such probabilistically weak solutions with the same constant $C_{t,q}$ in \eqref{[Equation (4.48t), HZZ19]} for every $q \in \mathbb{N}$ and $t \geq s$ is denoted by $\mathcal{W}(s, \xi^{\text{in}}, \theta^{\text{in}}, \{C_{t,q}\}_{q\in\mathbb{N}, t \geq s})$. 
\end{define} 

\begin{define}\label{[Def. 5.2, Y20a]}
Let $s \geq 0$, $\xi^{\text{in}} \in L_{\sigma}^{2}$, and $\theta^{\text{in}} \in U_{1}$. Let $\tau \geq s$ be a stopping time of $(\bar{\mathcal{B}}_{t})_{t \geq s}$ and set 
\begin{equation}
\bar{\Omega}_{\tau} \triangleq \{ \omega ( \cdot \wedge \tau(\omega)): \hspace{0.5mm} \omega \in \bar{\Omega} \} = \{ \omega \in \bar{\Omega}: \hspace{0.5mm}  (\xi, \theta)(t, \omega) = (\xi, \theta)(t \wedge \tau(\omega), \omega) \}. 
\end{equation} 
Then $P \in \mathcal{P} (\bar{\Omega}_{\tau})$ is a probabilistically weak solution to \eqref{stochastic GNS} on $[s, \tau]$ with initial condition $(\xi^{\text{in}}, \theta^{\text{in}})$ at initial time $s$ if 
\begin{itemize}
\item [] (M1) $P(\{ \xi(t) = \xi^{\text{in}}, \theta(t) = \theta^{\text{in}} \hspace{1mm} \forall \hspace{1mm} t \in [0,s] \}) = 1$ and for all $l \in \mathbb{N}$
\begin{equation}\label{[Equation (4.48v), HZZ19]}
P ( \{ ( \xi, \theta) \in \bar{\Omega}: \int_{0}^{l \wedge \tau} \lVert G(\xi(r)) \rVert_{L_{2}(U, L_{\sigma}^{2})}^{2} dr < \infty \}) = 1, 
\end{equation}
\item  [] (M2) under $P$, $\langle \theta(\cdot \wedge \tau), l_{i} \rangle_{U}$, where $\{l_{i} \}_{i\in \mathbb{N}}$ is an orthonormal basis of $U$,  is a continuous, square-integrable $(\bar{\mathcal{B}}_{t})_{t \geq s}$-martingale with initial condition $\langle \theta^{\text{in}}, l_{i} \rangle$ at initial time $s$ with a quadratic variation process given by $(t \wedge \tau - s) \lVert l_{i} \rVert_{U}^{2}$ and for every $\psi_{i} \in C^{\infty} (\mathbb{T}^{3} ) \cap L_{\sigma}^{2}$ and $t \geq s$, 
\begin{align}
\langle \xi(t\wedge \tau) - \xi(s), \psi_{i} \rangle + \int_{s}^{t \wedge \tau} \langle \text{div} (\xi(r) \otimes \xi(r)) +& (-\Delta)^{m} \xi(r), \psi_{i} \rangle dr\nonumber\\
=& \int_{s}^{t \wedge \tau} \langle \psi_{i}, G(\xi(r)) d\theta(r) \rangle, \label{[Equation (4.48w), HZZ19]}
\end{align}
\item [] (M3) for any $q \in \mathbb{N}$, there exists a function $t\mapsto C_{t,q} \in \mathbb{R}_{+}$ such that for all $t \geq s$, 
\begin{equation}\label{[Equation (4.48x), HZZ19]}
\mathbb{E}^{P} [ \sup_{r \in [0, t \wedge \tau]} \lVert \xi(r) \rVert_{L_{x}^{2}}^{2q} + \int_{s}^{t \wedge \tau} \lVert \xi(r) \rVert_{\dot{H}_{x}^{\gamma}}^{2} dr] \leq C_{t,q} (1+ \lVert \xi^{\text{in}} \rVert_{L_{x}^{2}}^{2q}). 
\end{equation} 
\end{itemize} 
\end{define}  
The following three results immediately follow from previous works \cite{HZZ19, Y20a} because the diffusivity strength made little differences in their proofs. Let $\bar{\mathcal{B}}_{\tau}$ denote the $\sigma$-algebra associated to any given stopping time $\tau$. 
\begin{proposition}\label{[Pro. 5.1, Y20a]}
\rm{(\cite[The. 5.1]{HZZ19}, \cite[Pro. 5.1]{Y20a})} For every $(s, \xi^{\text{in}}, \theta^{\text{in}}) \in [0,\infty) \times L_{\sigma}^{2} \times U_{1}$, there exists a probabilistically weak solution $P \in \mathcal{P} (\bar{\Omega})$ to \eqref{stochastic GNS} with initial condition $(\xi^{\text{in}}, \theta^{\text{in}})$ at initial time $s$ according to Definition \ref{[Def. 5.1, Y20a]}. Moreover, if there exists a family $\{(s_{l}, \xi_{l}, \theta_{l})\}_{l\in\mathbb{N}} \subset [0,\infty) \times L_{\sigma}^{2} \times U_{1}$ such that $\lim_{l\to\infty} \lVert (s_{l}, \xi_{l}, \theta_{l}) - (s, \xi^{\text{in}}, \theta^{\text{in}}) \rVert_{\mathbb{R} \times L_{\sigma}^{2} \times U_{1}} = 0$ and $P_{l} \in \mathcal{W} (s_{l}, \xi_{l}, \theta_{l}, \{C_{t,q} \}_{q \in \mathbb{N}, t \geq s_{l}})$, then there exists a subsequence $\{P_{l_{k}} \}_{k\in\mathbb{N}}$ that converges weakly to some $P \in \mathcal{W} (s, \xi^{\text{in}}, \theta^{\text{in}}, \{C_{t,q}\}_{q \in \mathbb{N}, t \geq s})$. 
\end{proposition} 

\begin{lemma}\label{[Lem. 5.2, Y20a]}
\rm{(\cite[Pro. 5.2]{HZZ19})} Let $\tau$ be a bounded $(\bar{\mathcal{B}}_{t})_{t \geq 0}$-stopping time. Then for every $\omega \in \bar{\Omega}$, there exists $Q_{\omega} \in \mathcal{P}(\bar{\Omega})$ such that 
\begin{subequations}
\begin{align}
& Q_{\omega} ( \{ \omega' \in \bar{\Omega}: \hspace{0.5mm} \hspace{1mm}  ( \xi, \theta) (t, \omega') = (\xi, \theta) (t,\omega) \hspace{1mm} \forall \hspace{1mm} t \in [0, \tau(\omega)] \}) = 1, \\
& Q_{\omega} (A) = R_{\tau(\omega), \xi(\tau(\omega), \omega), \theta(\tau(\omega), \omega)} (A) \hspace{1mm} \forall \hspace{1mm} A \in \bar{\mathcal{B}}^{\tau(\omega)}, 
\end{align} 
\end{subequations}
where $R_{\tau(\omega), \xi(\tau(\omega), \omega), \theta(\tau(\omega), \omega)} \in \mathcal{P} (\bar{\Omega})$ is a probabilistically weak solution to \eqref{stochastic GNS} with initial condition $(\xi(\tau(\omega), \omega), \theta(\tau(\omega), \omega))$ at initial time $\tau(\omega)$. Moreover, for every $A \in \bar{\mathcal{B}}$, the mapping $\omega \mapsto Q_{\omega}(A)$ is $\bar{\mathcal{B}}_{\tau}$-measurable, where $\bar{\mathcal{B}}$ is the $\sigma$-algebra on $\bar{\Omega}$ from Section \ref{Preliminaries}.   
\end{lemma} 

\begin{lemma}\label{[Lem. 5.3, Y20a]}
\rm{(\cite[Pro. 5.3]{HZZ19})} Let $\tau$ be a bounded $(\bar{\mathcal{B}}_{t})_{t\geq 0}$-stopping time, $\xi^{\text{in}} \in L_{\sigma}^{2}$, and $P \in \mathcal{P} (\bar{\Omega})$ be a probabilistically weak solution to \eqref{stochastic GNS} on $[0,\tau]$ with initial condition $(\xi^{\text{in}}, 0)$ at initial time $0$ according to Definition \ref{[Def. 5.2, Y20a]}. Suppose that there exists a Borel set $\mathcal{N} \subset \bar{\Omega}_{\tau}$ such that $P(\mathcal{N}) = 0$ and $Q_{\omega}$ from Lemma \ref{[Lem. 5.2, Y20a]} satisfies for every $\omega \in \bar{\Omega}_{\tau} \setminus \mathcal{N}$ 
\begin{equation}\label{[Equ. (143), Y20c]}
Q_{\omega} (\{ \omega' \in \bar{\Omega}: \hspace{0.5mm}  \tau(\omega') = \tau(\omega) \}) = 1. 
\end{equation} 
Then the probability measure $P\otimes_{\tau}R \in \mathcal{P} (\bar{\Omega})$ defined by 
\begin{equation}\label{[Equ. (144), Y20c]}
P \otimes_{\tau} R (\cdot) \triangleq \int_{\bar{\Omega}} Q_{\omega} (\cdot) P(d \omega)
\end{equation} 
satisfies $P\otimes_{\tau}R \rvert_{\bar{\Omega}_{\tau}} = P \rvert_{\bar{\Omega}_{\tau}}$ and it is a probabilistically weak solution to \eqref{stochastic GNS} on $[0,\infty)$ with initial condition $(\xi^{\text{in}}, 0)$ at initial time $0$. 
\end{lemma} 
Now we fix $\mathbb{R}$-valued Wiener process $B$ on $(\Omega, \mathcal{F}, \textbf{P})$ with $(\mathcal{F}_{t})_{t\geq 0}$ as its normal filtration. For $l \in \mathbb{N}, L > 1$, and $\delta \in (0, \frac{1}{24})$ we define 
\begin{subequations}\label{estimate 176}
\begin{align}
\tau_{L}^{l} (\omega) \triangleq& \inf\{t \geq 0: \lvert \theta(t,\omega) \rvert > (L - \frac{1}{l})^{\frac{1}{4}}\} \wedge \inf\{t \geq 0:  \lVert \theta(\omega) \rVert_{C_{t}^{\frac{1}{2} - 2 \delta}} > (L - \frac{1}{l})^{\frac{1}{2}} \wedge L, \label{estimate 67}\\
\tau_{L} \triangleq& \lim_{l\to\infty} \tau_{L}^{l}. \label{[Equ. (145a) and (145b), Y20c]}
\end{align}
\end{subequations} 
Comparing \eqref{stochastic GNS} and \eqref{[Equation (4.48s), HZZ19]} we see that $F (\xi(r)) = \xi(r), \theta = B$; as Brownian path is locally H$\ddot{\mathrm{o}}$lder continuous with exponent $\alpha \in (0,\frac{1}{2})$, it follows by \cite[Lem. 3.5]{HZZ19} that $\tau_{L}$ is a stopping time of $(\bar{\mathcal{B}}_{t})_{t\geq 0}$. For the fixed $(\Omega, \mathcal{F}, \textbf{P})$, we assume Theorem \ref{Theorem 2.3} and denote by $u$ the solution constructed by Theorem \ref{Theorem 2.3} on $[0,\mathfrak{t}]$ where $\mathfrak{t} = T_{L}$ for $L > 1$ sufficiently large and 
\begin{equation}\label{[Equ. (146), Y20c]}
T_{L} \triangleq \inf\{t> 0: \lvert B(t) \rvert \geq L^{\frac{1}{4}} \}  \wedge \inf\{t > 0: \lVert B \rVert_{C_{t}^{\frac{1}{2} - 2\delta}} \geq L^{\frac{1}{2}} \} \wedge L.  
\end{equation}
With $P$ representing the law of $(u,B)$, the following two results also follow immediately from previous works (\cite{HZZ19, Y20a, Y20c}) making use of the fact that 
\begin{equation}\label{[Equ. (147), Y20c]}
\theta(t, (u, B)) = B(t) \hspace{5mm} \forall \hspace{1mm} t \in [0, T_{L}] \hspace{1mm} \textbf{P}\text{-almost surely}. 
\end{equation}
\begin{proposition}\label{[Pro. 5.4, Y20a]}
\rm{(cf. \cite[Pro. 5.4]{HZZ19}, \cite[Pro. 5.4]{Y20a})} Let $\tau_{L}$ be defined by \eqref{[Equ. (145a) and (145b), Y20c]}. Then $P = \mathcal{L}(u, B)$, is a probabilistically weak solution to \eqref{stochastic GNS} on $[0, \tau_{L}]$ that satisfies  Definition \ref{[Def. 5.2, Y20a]}. 
\end{proposition}
\begin{proposition}\label{[Pro. 5.5, Y20a]}
\rm{(cf. \cite[Pro. 5.5]{HZZ19}, \cite[Pro. 5.5]{Y20a})}  Let $\tau_{L}$ be defined by \eqref{[Equ. (145a) and (145b), Y20c]} and $P = \mathcal{L}(u,b)$. Then $P\otimes_{\tau_{L}}R$ in \eqref{[Equ. (144), Y20c]} is a probabilistically weak solution to \eqref{stochastic GNS} on $[0,\infty)$ that satisfies Definition \ref{[Def. 5.1, Y20a]}. 
\end{proposition}
Similarly to Theorem \ref{Theorem 2.2}, at this point we are ready to prove Theorem \ref{Theorem 2.4}; due to its similarity to previous works \cite{HZZ19, Y20a}, we leave this in the Appendix. 

\subsection{Proof of Theorem \ref{Theorem 2.3} assuming Proposition \ref{[Pro. 5.7, Y20a]}}
We define $\Upsilon(t) \triangleq e^{B(t)}$ and $v \triangleq \Upsilon^{-1} u$ for $t\geq 0$. It follows from It$\hat{\mathrm{o}}$'s product formula that 
\begin{equation}\label{[Equ. (121), Y20a]}
\partial_{t} v + \frac{1}{2} v + (-\Delta)^{m} v + \Upsilon \text{div} (v \otimes v) +  \nabla (\Upsilon^{-1}\pi) = 0, \hspace{3mm} \nabla\cdot v = 0 \hspace{2mm} \text{ for } t > 0. 
\end{equation} 
For every $q \in \mathbb{N}_{0}$ we aim to construct $(v_{q}, \mathring{R}_{q})$ that satisfies 
\begin{equation}\label{[Equ. (122), Y20a]}
\partial_{t} v_{q} + \frac{1}{2} v_{q} + (-\Delta)^{m} v_{q} + \Upsilon \text{div} (v_{q} \otimes v_{q}) + \nabla p_{q} = \text{div} \mathring{R}_{q}, \hspace{3mm} \nabla\cdot v_{q} = 0 \hspace{2mm} \text{ for } t > 0. 
\end{equation} 
We define $\lambda_{q}$ and $\delta_{q}$ identically to the additive case in \eqref{estimate 69} but define differently
\begin{equation}\label{estimate 70}
M_{0}(t) \triangleq e^{4L t + 2L} \hspace{1mm} \text{ and } \hspace{1mm} m_{L} \triangleq \sqrt{3} L^{\frac{1}{4}} e^{\frac{1}{2} L^{\frac{1}{4}}}. 
\end{equation} 
Due to \eqref{[Equ. (146), Y20c]}, for all $L > 1, \delta \in (0, \frac{1}{24})$, and $t \in [0, T_{L}]$ we have 
\begin{equation}\label{[Equ. (123), Y20a]}
\lvert B(t) \rvert \leq L^{\frac{1}{4}} \text{ and } \lVert B \rVert_{C_{t}^{\frac{1}{2} - 2 \delta}} \leq L^{\frac{1}{2}}
\end{equation}
which implies 
\begin{equation}\label{[Equ. (124), Y20a]}
\lVert \Upsilon \rVert_{C_{t}^{\frac{1}{2} - 2\delta}} + \lvert \Upsilon(t) \rvert + \lvert \Upsilon^{-1}(t) \rvert \leq e^{L^{\frac{1}{4}}} L^{\frac{1}{2}} + 2e^{L^{\frac{1}{4}}} \leq m_{L}^{2}.
\end{equation} 
For inductive bounds we assume that $(v_{q}, \mathring{R}_{q})$ for all $q \in \mathbb{N}_{0}$ satisfy the following on $[0, T_{L}]$ with another universal constant $c_{R} > 0$ to be determined subsequently (see \eqref{estimate 97} and \eqref{estimate 110}):  
\begin{subequations}\label{estimate 149}
\begin{align}
& \lVert v_{q} \rVert_{C_{t,x}} \leq m_{L} M_{0}(t)^{\frac{1}{2}} (1 + \sum_{1\leq \iota \leq q} \delta_{\iota}^{\frac{1}{2}}) \leq 2 m_{L} M_{0}(t)^{\frac{1}{2}},  \label{[Equ. (126a), Y20a]}\\
&\lVert v_{q} \rVert_{C_{t,x}^{1}} \leq m_{L}^{4} M_{0}(t) \delta_{q}^{\frac{1}{2}} \lambda_{q},  \label{[Equ. (126b), Y20a]}\\
& \lVert \mathring{R}_{q} \rVert_{C_{t,x}} \leq c_{R} M_{0}(t) \delta_{q+1},  \label{[Equ. (126c), Y20a]}
\end{align}
\end{subequations}  
where again we follow the convention that $\sum_{1\leq \iota \leq 0} \delta_{\iota}^{\frac{1}{2}} = 0$ and assume \eqref{[Equ. (4.5), HZZ19]}, to be formally stated in \eqref{[Equ. (129), Y20a]}, so that $\sum_{1 \leq \iota \leq q} \delta_{\iota}^{\frac{1}{2}} < \frac{1}{2(2\pi)^{\frac{3}{2}}} < \frac{1}{2}$ for any $q \in \mathbb{N}$ and hence the second inequality of \eqref{[Equ. (126a), Y20a]} is justified. 

\begin{proposition}\label{[Pro. 5.6, Y20a]}
For $L > 1$, define 
\begin{equation}\label{[Equ. (127), Y20a]}
v_{0}(t,x) \triangleq (2\pi)^{-\frac{3}{2}} m_{L} e^{2L t+ L} 
\begin{pmatrix}
\sin(x^{3}) & 0 & 0 
\end{pmatrix}^{T}.  
\end{equation} 
Then together with 
\begin{equation}\label{[Equ. (128), Y20a]}
\mathring{R}_{0} (t,x) \triangleq \frac{ m_{L} (2L + \frac{1}{2}) e^{2L t + L}}{(2\pi)^{\frac{3}{2}}} 
\begin{pmatrix}
0 & 0 & - \cos(x^{3}) \\
0 & 0 & 0 \\
-\cos(x^{3}) & 0 & 0 
\end{pmatrix} 
+ \mathcal{R} (-\Delta)^{m} v_{0}, 
\end{equation} 
it satisfies \eqref{[Equ. (122), Y20a]} at level $q= 0$. Moreover, \eqref{estimate 149} at level $q=0$ is satisfied provided 
\begin{equation}\label{[Equ. (129), Y20a]}
\sqrt{3} [1 + 2(2\pi)^{\frac{3}{2}}]^{2} < \sqrt{3} a^{4\beta} \leq \frac{c_{R} e^{L}}{L^{\frac{1}{4}} ( 4L + 1 + C_{S} \sqrt{2}) e^{\frac{1}{2} L^{\frac{1}{4}}}} 
\end{equation} 
where the first inequality guaranties \eqref{[Equ. (4.5), HZZ19]}. Furthermore, $v_{0}(0,x)$ and $\mathring{R}_{0}(0,x)$ are both deterministic. 
\end{proposition} 

\begin{proof}[Proof of Proposition \ref{[Pro. 5.6, Y20a]}]
The facts that $v_{0}$ is incompressible, mean-zero, $\mathring{R}_{0}$ is trace-free and symmetric, and \eqref{[Equ. (122), Y20a]} at level $q = 0$ holds with $p_{0} \equiv 0$, as well as both $v_{0}(0,x)$ and $\mathring{R}_{0}(0,x)$ both being deterministic can be readily verified (see \cite[Pro. 5.6]{Y20a}). Concerning the three estimates \eqref{[Equ. (126a), Y20a]}-\eqref{[Equ. (126c), Y20a]} we compute 
\begin{subequations}
\begin{align}
&\lVert v_{0} \rVert_{C_{t,x}}  = (2\pi)^{-\frac{3}{2}} m_{L} M_{0}(t)^{\frac{1}{2}} \leq m_{L} M_{0}(t)^{\frac{1}{2}}, \\
& \lVert v_{0} \rVert_{C_{t,x}^{1}} =  (2\pi)^{-\frac{3}{2}}2(L+1) m_{L} M_{0}(t)^{\frac{1}{2}} \leq  m_{L}^{4} M_{0}(t) \delta_{0}^{\frac{1}{2}} \lambda_{0}, 
\end{align}
\end{subequations} 
and 
\begin{equation}\label{estimate 72} 
\lVert v_{0}(t) \rVert_{L^{2}} \overset{\eqref{[Equ. (127), Y20a]}}{=} \frac{ m_{L} M_{0}(t)^{\frac{1}{2}}}{\sqrt{2}}.
\end{equation} 
Finally, 
\begin{equation}\label{estimate 73}
\lVert \mathring{R}_{0} \rVert_{C_{t,x}} \overset{\eqref{[Equ. (128), Y20a]}}{\leq} (2\pi)^{-\frac{3}{2}} m_{L} (2L + \frac{1}{2}) e^{2L t + L} 2 + \lVert \mathcal{R} (-\Delta)^{m} v_{0} \rVert_{C_{t,x}}. 
\end{equation} 
By the same computations in \eqref{estimate 71} of the proof of Proposition \ref{[Pro. 4.7, Y20a]} we know $\lVert \mathcal{R} (-\Delta)^{m} v_{0} \rVert_{C_{t,x}} \leq C_{S} 2 \lVert v_{0} \rVert_{C_{t}L_{x}^{2}}$ for the same $C_{S} > 0$ from \eqref{[Equation (3.13b), HZZ19]} because $v_{0}$ in \eqref{[Equ. (127), Y20a]} also satisfies $\Delta v_{0} =  - v_{0}$. Therefore, applying \eqref{estimate 72} to \eqref{estimate 73} leads us to 
\begin{equation}
\lVert \mathring{R}_{0} \rVert_{C_{t,x}} \overset{\eqref{estimate 73}\eqref{estimate 72}}{\leq} \frac{ m_{L} (4L + 1) e^{2L t+ L}}{(2\pi)^{\frac{3}{2}}} + \frac{C_{S} 2 m_{L} M_{0}(t)^{\frac{1}{2}}}{\sqrt{2}} \overset{\eqref{[Equ. (129), Y20a]}}{\leq} c_{R} M_{0}(t) \delta_{1}. 
\end{equation} 
\end{proof} 

\begin{proposition}\label{[Pro. 5.7, Y20a]}
Let $L > 1$ satisfy 
\begin{equation}\label{estimate 74}
\sqrt{3} [1+ 2 (2\pi)^{\frac{3}{2}} ]^{2} < \frac{c_{R} e^{L}}{L^{\frac{1}{4}} ( 4L + 1 + C_{S} \sqrt{2}) e^{\frac{1}{2} L^{\frac{1}{4}}}}. 
\end{equation} 
Suppose that $(v_{q}, \mathring{R}_{q})$ is an $(\mathcal{F}_{t})_{t\geq 0}$-adapted process that solves \eqref{[Equ. (122), Y20a]} and satisfies \eqref{[Equ. (126a), Y20a]}-\eqref{[Equ. (126c), Y20a]}. Then there exist  a choice of parameters $a$ and $\beta$ such that \eqref{[Equ. (129), Y20a]} is fulfilled and an $(\mathcal{F}_{t})_{t\geq 0}$-adapted process $(v_{q+1}, \mathring{R}_{q+1})$ that solves \eqref{[Equ. (122), Y20a]}, satisfies \eqref{[Equ. (126a), Y20a]}-\eqref{[Equ. (126c), Y20a]} at level $q+1$, and for all $t \in [0, T_{L}]$ 
\begin{equation}\label{[Equ. (133), Y20a]}
\lVert v_{q+1} - v_{q} \rVert_{C_{t,x}} \leq m_{L} M_{0}(t)^{\frac{1}{2}} \delta_{q+1}^{\frac{1}{2}}. 
\end{equation} 
Finally, if $v_{q} (0,x)$ and $\mathring{R}_{q}(0,x)$ are deterministic, then so are $v_{q+1}(0,x)$ and $\mathring{R}_{q+1}(0,x)$. 
\end{proposition}
Taking Proposition \ref{[Pro. 5.7, Y20a]} for granted, we are able to prove Theorem \ref{Theorem 2.3} now. 

\begin{proof}[Proof of Theorem \ref{Theorem 2.3} assuming Proposition \ref{[Pro. 5.7, Y20a]}]
Given any $T> 0, K > 1$, and $\kappa \in (0,1)$, starting from $(v_{0}, \mathring{R}_{0})$ in Proposition \ref{[Pro. 5.6, Y20a]}, Proposition \ref{[Pro. 5.7, Y20a]} gives us $(v_{q}, \mathring{R}_{q})$ for all $q \geq 1$  that are $(\mathcal{F}_{t})_{t\geq 0}$-adapted and satisfy \eqref{[Equ. (122), Y20a]}, \eqref{[Equ. (126a), Y20a]}-\eqref{[Equ. (126c), Y20a]}, and \eqref{[Equ. (133), Y20a]}, as well as $a$ and $\beta$ such that \eqref{[Equ. (129), Y20a]} is fulfilled. Then for all $\gamma \in (0,\beta)$, similarly to \eqref{estimate 62}, using the fact that $2^{q+1} \geq 2(q+1)$ for all $q \in \mathbb{N}_{0}$, 
\begin{align}
\sum_{q\geq 0} \lVert v_{q+1} - v_{q} \rVert_{C_{t,x}^{\gamma}} \overset{\eqref{estimate 9}}{\lesssim}& \sum_{q\geq 0} \lVert v_{q+1} - v_{q} \rVert_{C_{t,x}}^{1-\gamma} \lVert v_{q+1} - v_{q} \rVert_{C_{t,x}^{1}}^{\gamma} \nonumber \\
&\overset{\eqref{[Equ. (133), Y20a]} \eqref{[Equ. (126b), Y20a]} }{\lesssim} m_{L}^{1+ 3 \gamma} M_{0}(t)^{\frac{1+\gamma}{2}} \sum_{q\geq 0} a^{2^{q+1} (\gamma - \beta)}  \lesssim m_{L}^{1+ 3\gamma} M_{0}(t)^{\frac{1+ \gamma}{2}}. \label{estimate 160}
\end{align} 
This implies that $\{v_{q}\}_{q=1}^{\infty}$ is Cauchy in $C([0,T_{L}]; C^{\gamma}(\mathbb{T}^{3}))$ and hence we can deduce a limiting solution $v \triangleq \lim_{q\to\infty} v_{q} \in C([0, T_{L}]; C^{\gamma}(\mathbb{T}^{3}))$ that is $(\mathcal{F}_{t})_{t\geq 0}$-adapted. Because $u = \Upsilon v = e^{B} v$, due to \eqref{[Equ. (123), Y20a]} we can deduce \eqref{estimate 2}. Because $\lim_{q\to \infty} \lVert \mathring{R}_{q} \rVert_{C_{t,x}} \leq \lim_{q\to\infty} c_{R} M_{0}(t) \delta_{q+1} = 0$ due to \eqref{[Equ. (126c), Y20a]}, we see that $v$ is a weak solution to \eqref{[Equ. (121), Y20a]} on $[0, T_{L}]$. Then $u = \Upsilon v$ is a $(\mathcal{F}_{t})_{\geq 0}$-adapted solution to \eqref{stochastic GNS}. Moreover, similarly to \eqref{estimate 10} we can show 
\begin{equation}\label{estimate 150}
\lVert v - v_{0} \rVert_{C_{t,x}} \overset{\eqref{[Equ. (133), Y20a]}}{\leq} m_{L} M_{0}(t)^{\frac{1}{2}} \sum_{q\geq 0} \delta_{q+1}^{\frac{1}{2}} \leq m_{L} M_{0}(t)^{\frac{1}{2}} \sum_{q\geq 0} a^{-2(q+1) \beta} \overset{\eqref{[Equ. (129), Y20a]}}{<} m_{L} M_{0}(t)^{\frac{1}{2}} \frac{1}{2(2\pi)^{\frac{3}{2}}}.
\end{equation} 
Therefore, 
\begin{equation}\label{estimate 75} 
\lVert v - v_{0} \rVert_{C_{t}L_{x}^{2}} \leq (2\pi)^{\frac{3}{2}} \lVert v - v_{0} \rVert_{C_{t,x}} \overset{\eqref{estimate 150}}{<} \frac{m_{L} M_{0}(t)^{\frac{1}{2}}}{2}.
\end{equation} 
Next, we take $L > 1$ sufficiently large so that not only \eqref{estimate 74} but 
\begin{equation}\label{[Equ. (134), Y20a]}
(\frac{1}{\sqrt{2}} - \frac{1}{2}) e^{2L T} > (\frac{1}{\sqrt{2}} + \frac{1}{2}) e^{2L^{\frac{1}{2}}} \text{ and } L > [ \ln (K e^{\frac{T}{2}})]^{2} 
\end{equation} 
hold. It follows that 
\begin{equation}\label{estimate 76}
e^{2L^{\frac{1}{2}}} \lVert v(0) \rVert_{L_{x}^{2}} \overset{\eqref{estimate 75} \eqref{estimate 72} }{\leq} e^{2L^{\frac{1}{2}}} m_{L} M_{0}(0)^{\frac{1}{2}} ( \frac{1}{2} + \frac{1}{\sqrt{2}}). 
\end{equation} 
This implies that on a set $\{T_{L} \geq T \}$, 
\begin{align}
\lVert v(T) \rVert_{L_{x}^{2}} \overset{\eqref{estimate 72}}{\geq}& \frac{m_{L} M_{0}(T)^{\frac{1}{2}}}{\sqrt{2}} - \lVert v(T) - v_{0}(T) \rVert_{L_{x}^{2}} \nonumber\\
&\overset{\eqref{estimate 75} \eqref{[Equ. (134), Y20a]}}{>} (\frac{1}{\sqrt{2}} + \frac{1}{2}) e^{2L^{\frac{1}{2}}} m_{L} M_{0}(0)^{\frac{1}{2}} \overset{\eqref{estimate 76}}{\geq} e^{2L^{\frac{1}{2}}} \lVert v(0) \rVert_{L_{x}^{2}}.  \label{estimate 77}
\end{align} 
This gives on the set $\{T_{L} \geq T\}$, 
\begin{equation}
\lVert u(T) \rVert_{L_{x}^{2}} \overset{\eqref{estimate 77}}{\geq} \lvert e^{B(T)} \rvert e^{2L^{\frac{1}{2}}} \lVert v(0) \rVert_{L_{x}^{2}}  
\overset{\eqref{[Equ. (123), Y20a]}}{\geq}  e^{L^{\frac{1}{2}}} \lVert u^{\text{in}} \rVert_{L_{x}^{2}} 
\overset{\eqref{[Equ. (134), Y20a]}}{>}  K e^{\frac{T}{2}} \lVert u^{\text{in}} \rVert_{L_{x}^{2}} 
\end{equation} 
which verifies \eqref{estimate 68}. Finally, taking $L > 1$ larger if necessary achieves \eqref{estimate 1} due to \eqref{[Equ. (146), Y20c]}. We also note that $u^{\text{in}}(x) = \Upsilon(0) v(0,x) = v(0,x)$ is deterministic because $v_{q}(0,x)$ is deterministic for all $q \in \mathbb{N}_{0}$ due to Propositions \ref{[Pro. 5.6, Y20a]} and \ref{[Pro. 5.7, Y20a]}.
\end{proof}

\subsection{Proof of Proposition \ref{[Pro. 5.7, Y20a]}}

\subsubsection{Mollification}
We fix $L > 1$ sufficiently large so that \eqref{estimate 74} holds, and then take $a \in \mathbb{N}$ sufficiently large while $\beta \in (0, \frac{1}{2})$ sufficiently small so that \eqref{[Equ. (129), Y20a]} holds. Now we define $l$ identically to  \eqref{[Equ. (56), Y20a]} and mollify $v_{q}$ and $\mathring{R}_{q}$ identically to \eqref{[Equ. (58), Y20a]} while 
\begin{equation}\label{estimate 152}
\Upsilon_{l} \triangleq \Upsilon \ast_{t} \varphi_{l}. 
\end{equation} 
Because $(v_{q}, \mathring{R}_{q})$ solves \eqref{[Equ. (122), Y20a]}, we see that 
\begin{equation}\label{[Equ. (138), Y20a]}
 \partial_{t} v_{l} + \frac{1}{2} v_{l} + (-\Delta)^{m} v_{l} + \Upsilon_{l}\text{div} (v_{l} \otimes v_{l}) + \nabla p_{l}  =  \text{div}( \mathring{R}_{l}  + R_{\text{com1}})
\end{equation}
where 
\begin{subequations}
\begin{align}
& R_{\text{com1}} \triangleq R_{\text{commutator1}} \triangleq  -((\Upsilon (v_{q} \mathring{\otimes} v_{q})) \ast_{x} \phi_{l} )\ast_{t} \varphi_{l} + \Upsilon_{l} (v_{l} \mathring{\otimes} v_{l}), \label{[Equ. (139b), Y20a]}\\
& p_{l} \triangleq (p_{q} \ast_{x} \phi_{l}) \ast_{t} \varphi_{l} - \frac{1}{3} (\Upsilon_{l} \lvert v_{l} \rvert^{2} - ((\Upsilon \lvert v_{q} \rvert^{2} )\ast_{x} \phi_{l} ) \ast_{t} \varphi_{l}). \label{[Equ. (139a), Y20a]}
\end{align}
\end{subequations} 
Let us compute for $N \in \mathbb{N}$, $\beta \in (0, \frac{1}{2})$, and $a\in \mathbb{N}$ sufficiently large  
\begin{subequations}\label{estimate 88}
\begin{align}
& \lVert v_{q} - v_{l} \rVert_{C_{t,x}} \overset{\eqref{[Equ. (126b), Y20a]}}{\lesssim} l m_{L}^{4} M_{0}(t) \delta_{q}^{\frac{1}{2}} \lambda_{q}  \ll m_{L} M_{0}(t)^{\frac{1}{2}} \delta_{q+1}^{\frac{1}{2}}, \label{estimate 85} \\
&\lVert v_{l} \rVert_{C_{t,x}^{N}} \overset{\eqref{[Equ. (126b), Y20a]}}{\lesssim} l^{-N+1} m_{L}^{4} M_{0}(t) \delta_{q}^{\frac{1}{2}} \lambda_{q} \ll l^{-N} m_{L} M_{0}(t),  \label{estimate 86}\\
& \lVert v_{l} \rVert_{C_{t,x}} \leq \lVert v_{q} \rVert_{C_{t,x}} \overset{\eqref{[Equ. (126a), Y20a]}}{\leq} m_{L} M_{0}(t)^{\frac{1}{2}} (1+ \sum_{1\leq \iota\leq q} \delta_{\iota}^{\frac{1}{2}}).\label{estimate 87}
\end{align}
\end{subequations} 

\subsubsection{Perturbation}
Differently from \eqref{[Equ. (5.18a), BV19b]}-\eqref{[Equ. (5.18b), BV19b]} we define $\Phi_{j}: [0, T_{L}] \times \mathbb{R}^{3}  \mapsto \mathbb{R}^{3}$ for $j \in \{0,\hdots, \lceil l^{-1} T_{L}\rceil \}$ a $\mathbb{T}^{3}$-periodic solution to 
\begin{subequations}
\begin{align}
(\partial_{t} + (\Upsilon_{l}v_{l}) \cdot \nabla) \Phi_{j} =0, \label{estimate 78}\\
\Phi_{j} (jl, x) = x.  \label{estimate 79}
\end{align}
\end{subequations}
Let us comment in Remark \ref{transport error again} on the importance of multiplying $v_{l}$ by $\Upsilon_{l}$ within \eqref{estimate 78}. We collect necessary estimates of $\Phi_{j}$. 
\begin{proposition}\label{Proposition 5.8}
For all $j \in \{0, \hdots, \lceil l^{-1} T_{L} \rceil \}$ and $t \in [l(j-1), l(j+1)]$ with appropriate modification in case $j = 0$ and $\lceil l^{-1} T_{L} \rceil$,  
\begin{subequations}\label{estimate 107}
\begin{align}
& \lVert \nabla \Phi_{j}(t) - \mathrm{Id} \rVert_{C_{x}} \lesssim l e^{L^{\frac{1}{4}}} m_{L}^{4} M_{0}(t) \delta_{q}^{\frac{1}{2}} \lambda_{q} \ll 1, \label{estimate 80} \\
& \frac{1}{2} \leq \lvert \nabla \Phi_{j}(t,x) \rvert \leq 2 \hspace{1mm} \forall \hspace{1mm} x \in \mathbb{T}^{3} \hspace{3mm} \text{ and } \hspace{3mm} \lVert \Phi_{j} (t) \rVert_{C_{x}^{1}} \lesssim 1, \label{estimate 81}\\
& \lVert \partial_{t} \Phi_{j}(t) \rVert_{C_{x}} \lesssim e^{L^{\frac{1}{4}}} m_{L} M_{0}(t)^{\frac{1}{2}}, \label{estimate 82}\\
& \lVert \nabla \Phi_{j}(t) \rVert_{C_{x}^{N}} \lesssim e^{L^{\frac{1}{4}}} m_{L}^{4} M_{0}(t) \delta_{q}^{\frac{1}{2}} \lambda_{q} l^{-N+1} \hspace{6mm} \forall \hspace{1mm} N \in \mathbb{N}, \label{estimate 83}\\
& \lVert \partial_{t} \nabla \Phi_{j}(t) \rVert_{C_{x}^{N}} \lesssim e^{2L^{\frac{1}{4}}} m_{L}^{5} M_{0}(t)^{\frac{3}{2}} \delta_{q}^{\frac{1}{2}}\lambda_{q} l^{-N}  \hspace{3mm} \forall \hspace{1mm} N \in \mathbb{N}_{0} \label{estimate 84} 
\end{align}
\end{subequations}
(cf. \cite[Equ. (5.19a) and (5.19c)]{BV19b} and \cite[Lem. 3.1]{BDIS15}). 
\end{proposition} 

\begin{proof}[Proof of Proposition \ref{Proposition 5.8}]
The proof is similar to that of Proposition \ref{Proposition 4.9} relying on \cite[Pro. D.1]{BDIS15}. First, due to \cite[Equ. (135)]{BDIS15}, $a \in \mathbb{N}$ sufficiently large gives 
\begin{equation}
\lVert \nabla \Phi_{j} (t) - \text{Id} \rVert_{C_{x}} \overset{\eqref{[Equ. (123), Y20a]} \eqref{estimate 86}}{\lesssim} e^{Cl e^{L^{\frac{1}{4}}} m_{L}^{4} M_{0}(t) \delta_{q}^{\frac{1}{2}} \lambda_{q}} - 1 \lesssim l e^{L^{\frac{1}{4}}} m_{L}^{4} M_{0}(t) \delta_{q}^{\frac{1}{2}} \lambda_{q} \ll 1. 
\end{equation} 
Second, the first estimate in \eqref{estimate 81} is an immediate consequence of \eqref{estimate 80} while the second estimate in \eqref{estimate 81} follows from \cite[Equ. (132)-(133)]{BDIS15} and \eqref{estimate 86}. Third, \eqref{estimate 82} follows from directly estimating on $\partial_{t} \Phi_{j}(t)$ from \eqref{estimate 78} via \eqref{[Equ. (123), Y20a]}, \eqref{estimate 87} and \eqref{estimate 81}. Fourth, \eqref{estimate 83} can be verified via \cite[Equ. (136)]{BDIS15} as follows: 
\begin{equation*}
\lVert \nabla \Phi_{j}(t) \rVert_{C_{x}^{N}}\overset{\eqref{[Equ. (123), Y20a]} \eqref{[Equ. (124), Y20a]}\eqref{estimate 86} }{\lesssim}  l [ e^{L^{\frac{1}{4}}} l^{-N} m_{L}^{4} M_{0}(t) \delta_{q}^{\frac{1}{2}} \lambda_{q} ] e^{Cl e^{L^{\frac{1}{4}}} m_{L}^{4} M_{0}(t) \delta_{q}^{\frac{1}{2}} \lambda_{q}} 
\lesssim l^{-N +1} e^{L^{\frac{1}{4}}} m_{L}^{4} M_{0}(t) \delta_{q}^{\frac{1}{2}} \lambda_{q}. 
\end{equation*} 
Finally, we can directly apply $\nabla$ on \eqref{estimate 78} and estimate in case $N \in \mathbb{N}$ 
\begin{align*}
&\lVert \partial_{t} \nabla \Phi_{j} (t) \rVert_{C_{x}^{N}} \overset{\eqref{[Equ. (123), Y20a]}}{\lesssim} e^{L^{\frac{1}{4}}} [ \lVert v_{l} \rVert_{C_{t}C_{x}^{N}} \lVert \nabla \nabla \Phi_{j} (t) \rVert_{C_{x}} + \lVert v_{l} \rVert_{C_{t,x}} \lVert \nabla \nabla \Phi_{j} (t) \rVert_{C_{x}^{N}} \\
&\hspace{7mm} + \lVert \nabla v_{l} \rVert_{C_{t}C_{x}^{N}} \lVert \nabla \Phi_{j} (t) \rVert_{C_{x}} + \lVert \nabla v_{l} \rVert_{C_{t,x}} \lVert \nabla \Phi_{j} (t) \rVert_{C_{x}^{N}}] 
\overset{\eqref{estimate 83} \eqref{estimate 88}}{\lesssim} 
\lesssim e^{2L^{\frac{1}{4}}} M_{0}(t)^{\frac{3}{2}} \delta_{q}^{\frac{1}{2}} \lambda_{q} m_{L}^{5} l^{-N}
\end{align*} 
while the case $N =0$ can be achieved similarly and more simply. 
\end{proof}

Let us define $\chi$ and $\chi_{j}$ for $j \in \{0, 1, \hdots, \lceil l^{-1} T_{L} \rceil \}$ identically to \eqref{[Equ. (5.20), BV19b]} in the proof of Proposition \ref{[Pro. 4.8, Y20a]} so that \eqref{[Equ. (5.21), BV19b]} continues to be satisfied. On the other hand, while we continue to define $a_{(\zeta)}$ identically to \eqref{[Equ. (5.22), BV19b]} except $M_{0}(t)$ is defined by \eqref{estimate 70} rather than \eqref{estimate 69}, we define a modified amplitude function as 
\begin{align}
\bar{a}_{(\zeta)} (t,x) \triangleq \bar{a}_{q+1, j, \zeta} (t,x) \triangleq& \Upsilon_{l}^{-\frac{1}{2}} a_{(\zeta)} (t,x) \nonumber \\
\overset{\eqref{[Equ. (5.22), BV19b]}}{=}& \Upsilon_{l}^{-\frac{1}{2}} c_{R}^{\frac{1}{4}} \delta_{q+1}^{\frac{1}{2}} M_{0}(t)^{\frac{1}{2}} \chi_{j}(t) \gamma_{\zeta} \left(\text{Id} - \frac{\mathring{R}_{l}(t,x)}{c_{R}^{\frac{1}{2}} \delta_{q+1} M_{0}(t)} \right). \label{estimate 90}
\end{align} 
Convenience of defining $\bar{a}_{(\zeta)}$ as $\Upsilon_{l}^{-\frac{1}{2}} a_{(\zeta)}$ will be clear in the derivations of \eqref{estimate 151} and \eqref{estimate 157}. As we have not changed the inductive hypothesis of $\mathring{R}_{q}$ (cf. \eqref{[Equ. (40c), Y20a]} and \eqref{[Equ. (126c), Y20a]}), the computations of \eqref{[Equ. (5.22a), BV19b]} and \eqref{estimate 20} go through without any issue so that $\text{Id} - \frac{ \mathring{R}_{l}}{c_{R}^{\frac{1}{2}} \delta_{q+1} M_{0}(t)}$ lies in the domain of $\gamma_{(\zeta)}$ from \eqref{[Equ. (5.15b), BV19b]}. Moreover, we derive the following crucial point-wise identity:   
\begin{equation}\label{estimate 89}
\Upsilon_{l}(t) (\frac{1}{2}) \sum_{j} \sum_{\zeta \in \Lambda_{j}} \bar{a}_{(\zeta)}^{2}(t,x) (\text{Id} - \zeta \otimes \zeta) + \mathring{R}_{l} (t,x) \overset{ \eqref{estimate 90}\eqref{[Equ. (5.23), BV19b]}}{=} c_{R}^{\frac{1}{2}} \delta_{q+1} M_{0}(t). 
\end{equation} 
Next, we obtain necessary estimates for $\bar{a}_{(\zeta)}$: 
\begin{proposition}\label{Proposition 5.9}
The modified amplitude function $\bar{a}_{(\zeta)}$ in \eqref{estimate 90} satisfies the following bounds on $[0, T_{L}]$: 
\begin{subequations}\label{estimate 106} 
\begin{align} 
\lVert \bar{a}_{(\zeta)} \rVert_{C_{t}C_{x}^{N}} \lesssim& e^{\frac{1}{2} L^{\frac{1}{4}}} c_{R}^{\frac{1}{4}} \delta_{q+1}^{\frac{1}{2}} M_{0}(t)^{\frac{1}{2}} \lVert \gamma_{\zeta} \rVert_{C^{N} (B_{C_{\ast}} (\mathrm{Id} ))} l^{-N} \hspace{8mm} \forall \hspace{1mm} N \in \mathbb{N}_{0},  \label{estimate 101} \\
\lVert \bar{a}_{(\zeta)} \rVert_{C_{t}^{1}C_{x}^{N}} \lesssim& e^{\frac{5}{2}L^{\frac{1}{4}}} c_{R}^{\frac{1}{4}} \delta_{q+1}^{\frac{1}{2}} M_{0}(t)^{\frac{1}{2}} \lVert \gamma_{\zeta} \rVert_{C^{N+1} (B_{C_{\ast}} (\mathrm{Id} ))} l^{-N -1} \hspace{3mm} \forall \hspace{1mm} N \in \mathbb{N}_{0}. \label{estimate 102}
\end{align}
\end{subequations} 
\end{proposition} 
\begin{proof}[Proof of Proposition \ref{Proposition 5.9}]
The first inequality \eqref{estimate 101} follows from the estimate \eqref{estimate 22} in Proposition \ref{Proposition 4.10} and the fact that $\lVert \Upsilon_{l}^{-\frac{1}{2}} \rVert_{C_{t}} \leq e^{\frac{1}{2} L^{\frac{1}{4}}}$ due to \eqref{[Equ. (123), Y20a]}. Although the definition of $M_{0}(t)$ in \eqref{estimate 70} is different from \eqref{estimate 69}, this makes no difference in the computations of \eqref{estimate 92}-\eqref{estimate 91}. Next, we can directly apply $\partial_{t}$ on \eqref{estimate 90} and estimate 
\begin{equation*}
\lVert \bar{a}_{(\zeta)} \rVert_{C_{t}^{1}C_{x}^{N}} \lesssim \lVert \Upsilon_{l}^{-\frac{3}{2}} (\partial_{t} \Upsilon_{l}) a_{(\zeta)} \rVert_{C_{t}C_{x}^{N}} + \lVert \Upsilon_{l}^{-\frac{1}{2}} \partial_{t}a_{(\zeta)} \rVert_{C_{t}C_{x}^{N}};
\end{equation*} 
then we can apply \eqref{[Equ. (123), Y20a]} and \eqref{estimate 22}-\eqref{estimate 23} and immediately obtain the desired result \eqref{estimate 102}.
\end{proof} 

Next, we recall $\bar{a}_{(\zeta)}, W_{\zeta, \lambda_{q+1}}$, and $B_{\zeta}$ respectively from \eqref{estimate 90}, \eqref{[Equ. (5.12), BV19b]}, and \eqref{[Equ. (5.11b), BV19b]}, and define 
\begin{subequations}\label{estimate 95}
\begin{align} 
w_{(\zeta)}^{(p)} (t,x) \triangleq& w_{q+1, j, \zeta}^{(p)}(t,x) \triangleq \bar{a}_{(\zeta)} (t,x) W_{\zeta, \lambda_{q+1}} (\Phi_{j}(t,x)) = \bar{a}_{(\zeta)} (t,x) B_{\zeta} e^{i\lambda_{q+1} \zeta \cdot \Phi_{j} (t,x)}, \label{estimate 93}\\
 w_{q+1}^{(p)} (t,x) \triangleq& \sum_{j} \sum_{\zeta \in \Lambda_{j}} w_{(\zeta)}^{(p)} (t,x); \label{estimate 94}
\end{align}
\end{subequations} 
we note that defining $w_{(\zeta)}^{(p)}$ this way with $\bar{a}_{(\zeta)}$ instead of $a_{(\zeta)}$ makes sure to eliminate a difficult term in $R_{\text{osc}}$, as we will subsequently see in \eqref{estimate 151}. Thus, by choosing $c_{R} \leq (2 \sqrt{2} M)^{-4}$ and using the facts that $\lVert \Upsilon_{l}^{-\frac{1}{2}} \rVert_{C_{t}} \leq e^{\frac{1}{2} L^{\frac{1}{4}}}$ and for any $s \in [0,t]$ fixed, there exist at most only two non-trivial cutoffs, we obtain 
\begin{equation}\label{estimate 97}
\lVert w_{q+1}^{(p)} \rVert_{C_{t,x}} \overset{\eqref{estimate 95} \eqref{estimate 90} \eqref{estimate 20} \eqref{[Equ. (5.17), BV19b]} \eqref{[Equ. (5.21), BV19b]}}{\leq} e^{\frac{1}{2} L^{\frac{1}{4}}} c_{R}^{\frac{1}{4}} \delta_{q+1}^{\frac{1}{2}} M_{0}(t)^{\frac{1}{2}} \sqrt{2} M \leq 2^{-1} e^{\frac{1}{2} L^{\frac{1}{4}}} \delta_{q+1}^{\frac{1}{2}} M_{0}(t)^{\frac{1}{2}}. 
\end{equation} 
Next, we define $\phi_{(\zeta)}$ identically to \eqref{[Equ. (5.27), BV19b]} so that 
\begin{equation}\label{estimate 98}
w_{(\zeta)}^{(p)} (t,x) \overset{\eqref{estimate 93}\eqref{[Equ. (5.27), BV19b]}}{=} \bar{a}_{(\zeta)} (t,x) B_{\zeta} \phi_{(\zeta)} (t,x) e^{i\lambda_{q+1} \zeta \cdot x} 
\overset{\eqref{[Equ. (5.12), BV19b]}}{=} \bar{a}_{(\zeta)} (t,x) \phi_{(\zeta)} (t,x) W_{(\zeta)} (x). 
\end{equation}
Then  
\begin{equation}\label{estimate 99}
\bar{a}_{(\zeta)} \phi_{(\zeta)} W_{(\zeta)} \overset{\eqref{eigenvector}}{=} \lambda_{q+1}^{-1} \nabla \times (\bar{a}_{(\zeta)} \phi_{(\zeta)} W_{(\zeta)}) - \lambda_{q+1}^{-1} \nabla (\bar{a}_{(\zeta)} \phi_{(\zeta)} ) \times W_{(\zeta)}.
\end{equation}
Next, let us define 
\begin{equation}\label{estimate 96}
w_{(\zeta)}^{(c)} (t,x) \triangleq \lambda_{q+1}^{-1} \nabla (\bar{a}_{(\zeta)} \phi_{(\zeta)} ) \times B_{\zeta} e^{i \lambda_{q+1} \zeta \cdot x}; 
\end{equation} 
the reason to incorporate $\bar{a}_{(\zeta)}$ within $w_{(\zeta)}^{(c)}$ is to make $w_{q+1}$ divergence-free as we will subsequently see in \eqref{estimate 153}. It follows that 
\begin{align}
w_{(\zeta)}^{(c)} (t,x) \overset{\eqref{estimate 96} \eqref{[Equ. (5.27), BV19b]}}{=}& \lambda_{q+1}^{-1} ( \nabla \bar{a}_{(\zeta)} (t,x) + \bar{a}_{(\zeta)} (t,x) i \lambda_{q+1} \zeta \cdot (\nabla \Phi_{j} (t,x) - \text{Id} )) \times B_{\zeta} e^{i \lambda_{q+1} \zeta \cdot \Phi_{j}(t,x)} \nonumber \\
\overset{\eqref{[Equ. (5.12), BV19b]}}{=}& (\lambda_{q+1}^{-1} \nabla \bar{a}_{(\zeta)}(t,x) + i \bar{a}_{(\zeta)}(t,x) \zeta \cdot (\nabla \Phi_{j}(t,x) - \text{Id} )) \times W_{(\zeta)} (\Phi_{j}(t,x)). \label{estimate 105}
\end{align} 
Thus, if we define $w_{q+1}^{(c)}$ and $w_{q+1}$ identically to \eqref{[Equ. (5.29) and (5.29a), BV19b]}, then 
\begin{equation}\label{estimate 153}
w_{q+1} \overset{\eqref{[Equ. (5.29) and (5.29a), BV19b]} \eqref{estimate 94}}{=} \sum_{j}\sum_{\zeta \in \Lambda_{j}} w_{(\zeta)}^{(p)} + w_{(\zeta)}^{(c)}  \overset{\eqref{estimate 98}-\eqref{estimate 96} \eqref{[Equ. (5.27), BV19b]} \eqref{[Equ. (5.12), BV19b]}}{=} \sum_{j} \sum_{\zeta \in \Lambda_{j}} \lambda_{q+1}^{-1} \nabla \times (\bar{a}_{(\zeta)} W_{(\zeta)} \circ \Phi_{j}) 
\end{equation}
which shows that $w_{q+1}$ is mean-zero and divergence-free because $\nabla\cdot (\nabla \times f) = 0$ for all $f$. 
Next, we can estimate using the fact that for all $s \in [0,t]$ fixed, there are only at most two non-trivial cutoffs,  
\begin{align}
& \lVert w_{q+1}^{(c)} \rVert_{C_{t,x}}  \overset{\eqref{[Equ. (5.29) and (5.29a), BV19b]} \eqref{estimate 105} }{\leq} 2 \sup_{j}\sum_{\zeta \in \Lambda_{j}} \lambda_{q+1}^{-1} \lVert \nabla \bar{a}_{(\zeta)} \rVert_{C_{t,x}} + \lVert \bar{a}_{(\zeta)} \rVert_{C_{t,x}} \sup_{s \in [0,t]} \lVert (\nabla \Phi_{j} (s) - \text{Id})1_{(l(j-1), l(j+1))}(s) \rVert_{C_{x}} \nonumber\\
& \hspace{25mm} \overset{\eqref{estimate 80} \eqref{estimate 101} \eqref{[Equ. (5.17), BV19b]}}{\lesssim}  \delta_{q+1}^{\frac{1}{2}} e^{\frac{1}{2} L^{\frac{1}{4}}} M_{0}(t)^{\frac{1}{2}} \lambda_{q}^{-\frac{1}{2}} \ll e^{\frac{1}{2} L^{\frac{1}{4}}} \delta_{q+1}^{\frac{1}{2}} M_{0}(t)^{\frac{1}{2}}. \label{estimate 103}
\end{align} 
It follows that 
\begin{equation}\label{estimate 104}
\lVert w_{q+1} \rVert_{C_{t,x}} \overset{\eqref{[Equ. (5.29) and (5.29a), BV19b]}}{\leq} \lVert w_{q+1}^{(p)} \rVert_{C_{t,x}} + \lVert w_{q+1}^{(c)} \rVert_{C_{t,x}}  
\overset{\eqref{estimate 97} \eqref{estimate 103}\eqref{estimate 70}}{\leq} \frac{3m_{L}  M_{0}(t)^{\frac{1}{2}}\delta_{q+1}^{\frac{1}{2}} }{4}. 
\end{equation} 
Thus, if we define the velocity field at level $q+1$ identically to \eqref{[Equ. (5.32), BV19b]}, then we can verify \eqref{[Equ. (133), Y20a]} as follows: 
\begin{equation}
\lVert v_{q+1} - v_{q} \rVert_{C_{t,x}} \overset{\eqref{[Equ. (5.32), BV19b]}}{\leq} \lVert w_{q+1} \rVert_{C_{t,x}}+ \lVert v_{l} - v_{q} \rVert_{C_{t,x}} \overset{\eqref{estimate 104}\eqref{estimate 85}}{\leq}  m_{L}M_{0}(t)^{\frac{1}{2}} \delta_{q+1}^{\frac{1}{2}}. 
\end{equation} 
Additionally, we can verify \eqref{[Equ. (126a), Y20a]} as follows:
\begin{equation}\label{estimate 134}
\lVert v_{q+1} \rVert_{C_{t,x}} \overset{\eqref{[Equ. (5.32), BV19b]}}{\leq} \lVert v_{l} \rVert_{C_{t,x}} + \lVert w_{q+1} \rVert_{C_{t,x}}  \overset{\eqref{estimate 87} \eqref{estimate 104}}{\leq} m_{L} M_{0}(t)^{\frac{1}{2}} (1+ \sum_{1\leq \iota \leq q+1} \delta_{\iota}^{\frac{1}{2}}). 
\end{equation} 
Next, in order to verify \eqref{[Equ. (126b), Y20a]} at level $q+1$, we compute similarly to \eqref{estimate 138} using the fact that for any fixed $s \in [0,t]$, there are at most two non-trivial cutoffs
\begin{subequations}
\begin{align}
&\lVert \partial_{t} w_{q+1}^{(p)} \rVert_{C_{t,x}} + \lVert \nabla w_{q+1}^{(p)} \rVert_{C_{t,x}} \nonumber \\
& \hspace{13mm} \overset{\eqref{estimate 95} \eqref{estimate 106}\eqref{estimate 107}}{\lesssim} M c_{R}^{\frac{1}{4}}\delta_{q+1}^{\frac{1}{2}} [ e^{\frac{5}{2} L^{\frac{1}{4}}}  M_{0}(t)^{\frac{1}{2}} l^{-1} + \lambda_{q+1} e^{\frac{3}{2} L^{\frac{1}{4}}} M_{0}(t) m_{L}], \label{estimate 108}\\
&\lVert \partial_{t} w_{q+1}^{(c)} \rVert_{C_{t,x}} + \lVert \nabla w_{q+1}^{(c)} \rVert_{C_{t,x}} \overset{\eqref{estimate 105} \eqref{estimate 106}\eqref{estimate 107}}{\lesssim}  Mc_{R}^{\frac{1}{4}}\delta_{q+1}^{\frac{1}{2}} [\lambda_{q+1}^{-1} e^{\frac{5}{2} L^{\frac{1}{4}}}  M_{0}(t)^{\frac{1}{2}} l^{-2} + e^{\frac{3}{2} L^{\frac{1}{4}}}  M_{0}(t) m_{L} l^{-1} \nonumber\\
&\hspace{41mm} + e^{\frac{7}{2} L^{\frac{1}{4}}} M_{0}(t)^{2} m_{L}^{5} \delta_{q}^{\frac{1}{2}} \lambda_{q} + \lambda_{q+1} e^{\frac{5}{2} L^{\frac{1}{4}}}  M_{0}(t)^{2} l m_{L}^{5} \delta_{q}^{\frac{1}{2}} \lambda_{q} ] \label{estimate 109}
\end{align} 
\end{subequations}
with $M$ from \eqref{[Equ. (5.17), BV19b]}. Therefore, by taking $c_{R} \ll M^{-4}$ and $a \in \mathbb{N}$ sufficiently large 
\begin{align}
&\lVert w_{q+1} \rVert_{C_{t,x}^{1}} \label{estimate 110} \\ 
\overset{\eqref{estimate 104} \eqref{[Equ. (5.29) and (5.29a), BV19b]}}{\leq}& \frac{3}{4} m_{L} \delta_{q+1}^{\frac{1}{2}} M_{0}(t)^{\frac{1}{2}} + \lVert \partial_{t} w_{q+1}^{(p)} \rVert_{C_{t,x}} + \lVert \nabla w_{q+1}^{(p)} \rVert_{C_{t,x}} +  \lVert \partial_{t} w_{q+1}^{(c)} \rVert_{C_{t,x}} + \lVert \nabla w_{q+1}^{(c)} \rVert_{C_{t,x}} \nonumber\\
\overset{\eqref{estimate 108} \eqref{estimate 109}}{\leq} & \frac{3}{4} m_{L} \delta_{q+1}^{\frac{1}{2}} M_{0}(t)^{\frac{1}{2}} + C \lambda_{q+1} \delta_{q+1}^{\frac{1}{2}} M_{0}(t) m_{L}^{4} M c_{R}^{\frac{1}{4}} \leq \frac{\lambda_{q+1} \delta_{q+1}^{\frac{1}{2}} M_{0}(t) m_{L}^{4}}{2}. \nonumber 
\end{align} 
We are now ready to verify \eqref{[Equ. (126b), Y20a]} at level $q+1$ as follows. By Young's inequality for convolution and the fact that mollifiers have mass one and $\beta \in (0, \frac{1}{2})$, 
\begin{equation} 
\lVert v_{q+1} \rVert_{C_{t,x}^{1}} 
\overset{\eqref{[Equ. (5.32), BV19b]}\eqref{estimate 110}}{\leq}  \lambda_{q+1} \delta_{q+1}^{\frac{1}{2}} M_{0}(t) m_{L}^{4}[ a^{2^{q} (-1 + \beta)} + \frac{1}{2}] 
\leq \lambda_{q+1} \delta_{q+1}^{\frac{1}{2}} M_{0}(t) m_{L}^{4}.
\end{equation} 
Subsequently, similarly to the proof of Proposition \ref{[Pro. 4.8, Y20a]}, we will rely on Lemma \ref{[Lem. 5.7, BV19b]} to estimate Reynolds stress. Due to \eqref{estimate 81} this time, by choosing $a \in \mathbb{N}$ sufficiently large we have $\frac{1}{2} \leq  \lvert \nabla \Phi_{j} (t,x) \rvert \leq 2$ for all $t \in [l(j-1), l(j+1)]$ and $x \in \mathbb{T}^{3}$ so that \eqref{[Equ. (5.34a), BV19b]} is satisfied with $C = 2$. Therefore,  \eqref{[Equ. (5.35), BV19b]} leads to \eqref{[Equ. (5.36), BV19b]}-\eqref{[Equ. (5.37), BV19b]} again. 

\subsubsection{Reynolds stress}
First, we observe that 
\begin{align}
&\text{div} \mathring{R}_{q+1} - \nabla p_{q+1} \overset{\eqref{[Equ. (122), Y20a]} \eqref{[Equ. (5.32), BV19b]} \eqref{[Equ. (138), Y20a]}}{=} - \Upsilon_{l} \text{div} (v_{l} \otimes v_{l}) - \nabla p_{l} + \text{div} (\mathring{R}_{l} + R_{\text{com1}}) \nonumber \\
& \hspace{15mm} + \partial_{t}w_{q+1}^{(p)} + \partial_{t}w_{q+1}^{(c)} + \frac{1}{2} w_{q+1} + (-\Delta)^{m} w_{q+1} + \Upsilon \text{div} (v_{q+1} \otimes v_{q+1}). \label{estimate 111}
\end{align} 
We have an identity of 
\begin{align}
& - \Upsilon_{l} \text{div} (v_{l} \otimes v_{l}) + \Upsilon \text{div} (v_{q+1} \otimes v_{q+1})  \nonumber\\
\overset{\eqref{[Equ. (5.32), BV19b]}}{=}& \Upsilon_{l} \text{div} (v_{l} \otimes w_{q+1} + w_{q+1} \otimes w_{q+1} + w_{q+1} \otimes v_{l}) + (\Upsilon - \Upsilon_{l}) \text{div} (v_{q+1} \otimes v_{q+1}). \label{estimate 137}
\end{align} 
Applying this identity \eqref{estimate 137} in \eqref{estimate 111} leads to 
\begin{align}
& \text{div} \mathring{R}_{q+1} - \nabla p_{q+1} = \underbrace{( \partial_{t} + \Upsilon_{l} (v_{l} \cdot \nabla )) w_{q+1}^{(p)}}_{\text{div} R_{\text{tran}}} + \underbrace{\text{div} (\Upsilon_{l} w_{q+1}^{(p)} \otimes w_{q+1}^{(p)} +\mathring{R}_{l})}_{\text{div} R_{\text{osc}} + \nabla p_{\text{osc}} }\nonumber\\
& \hspace{5mm} + \underbrace{\Upsilon_{l} (w_{q+1} \cdot \nabla) v_{l}}_{\text{div} R_{\text{Nash}}} + \underbrace{(\partial_{t} + \Upsilon_{l} (v_{l} \cdot \nabla) ) w_{q+1}^{(c)} + \Upsilon_{l} \text{div} (w_{q+1}^{(c)} \otimes w_{q+1} + w_{q+1}^{(p)} \otimes w_{q+1}^{(c)})}_{\text{div} R_{\text{corr}} + \nabla p_{\text{corr}}} \nonumber\\
& \hspace{5mm} + \text{div} R_{\text{com1}} - \nabla p_{l} + \underbrace{(\Upsilon - \Upsilon_{l}) \text{div} (v_{q+1} \otimes v_{q+1})}_{\text{div} R_{\text{com2}} + \nabla p_{\text{com2}}} + \underbrace{\frac{1}{2} w_{q+1} + (-\Delta)^{m} w_{q+1}}_{\text{div} R_{\text{line}}}. \label{estimate 154}
 \end{align} 
\begin{remark}\label{transport error again}
Similarly to Remark \ref{transport error}, we strategically multiplied $v_{l}$ in \eqref{estimate 78} by $\Upsilon_{l}$, and included $\Upsilon_{l} (v_{l} \cdot \nabla) w_{q+1}^{(p)}$ in $R_{\text{tran}}$ and $\Upsilon_{l} (v_{l} \cdot \nabla) w_{q+1}^{(c)}$ in $R_{\text{corr}}$. As we will see in \eqref{estimate 155} and \eqref{estimate 131}, this leads to a crucial cancellation of the most difficult term when $\nabla$ is applied on $e^{i \lambda_{q+1} \zeta \cdot \Phi_{j}}$ in which $\lambda_{q+1}$ from chain rule makes such terms too large to handle.  
\end{remark} 

Concerning $R_{\text{osc}}$ in \eqref{estimate 154}, making use of the fact that $\gamma_{\zeta} = \gamma_{-\zeta}$ from Lemma \ref{[Pro. 5.6, BV19b]} so that $\bar{a}_{\zeta} = \bar{a}_{-\zeta}$ in \eqref{estimate 90},  
\begin{align}
& \text{div} ( \Upsilon_{l} w_{q+1}^{(p)} \otimes w_{q+1}^{(p)} + \mathring{R}_{l}) \label{estimate 151}\\
\overset{\eqref{estimate 95} \eqref{estimate 98}\eqref{estimate 112} }{=}&  \text{div} (\Upsilon_{l} (\frac{1}{2}) \sum_{j} \sum_{\zeta \in \Lambda_{j}} \bar{a}_{(\zeta)}^{2} (\text{Id} - \zeta \otimes \zeta) + \mathring{R}_{l})\nonumber \\
&+ \sum_{j,j'} \sum_{\zeta \in \Lambda_{j}, \zeta' \in \Lambda_{j'}: \zeta + \zeta' \neq 0} \Upsilon_{l} \text{div} (\bar{a}_{(\zeta)} \phi_{(\zeta)} W_{(\zeta)} \otimes \bar{a}_{(\zeta')} \phi_{(\zeta')} W_{(\zeta')}) \nonumber\\
\overset{\eqref{estimate 89}}{=}& \text{div} (c_{R}^{\frac{1}{2}} \delta_{q+1} M_{0}(t)) + \sum_{j,j'} \sum_{\zeta \in \Lambda_{j}, \zeta' \in \Lambda_{j'}} \Upsilon_{l} \text{div} (\bar{a}_{(\zeta)} \phi_{(\zeta)} W_{(\zeta)} \otimes \bar{a}_{(\zeta')} \phi_{(\zeta')} W_{(\zeta')} ) \nonumber\\
\overset{\eqref{estimate 90} \eqref{estimate 113} }{=}& \nabla (\frac{1}{2} \sum_{j, j'} \sum_{\zeta \in \Lambda_{j}, \zeta' \in \Lambda_{j'}: \zeta + \zeta' \neq 0}  a_{(\zeta)} a_{(\zeta')} \phi_{(\zeta)} \phi_{(\zeta')} (W_{(\zeta)} \cdot W_{(\zeta)})) \nonumber\\
&+ \text{div} \mathcal{R} ( \sum_{j,j'}\sum_{\zeta \in \Lambda_{j}, \zeta' \in \Lambda_{j'}: \zeta + \zeta' \neq 0} (W_{(\zeta)} \otimes W_{(\zeta')} - \frac{W_{(\zeta)} \cdot W_{(\zeta')}}{2} \text{Id}) \nabla (a_{(\zeta)} a_{(\zeta')} \phi_{(\zeta)} \phi_{(\zeta')}))  \nonumber
\end{align} 
and hence $R_{\text{osc}}$ and $p_{\text{osc}}$ are same as those in \eqref{estimate 46}-\eqref{estimate 47}; therefore, the estimate \eqref{estimate 114} directly applies to the current case. Thus, besides $R_{\text{osc}}$, $p_{\text{osc}}$, and $R_{\text{com1}}$ in \eqref{[Equ. (139b), Y20a]}, we define from \eqref{estimate 154} 
\begin{subequations}\label{estimate 122}
\begin{align}
R_{\text{line}} \triangleq& R_{\text{linear}} \triangleq \mathcal{R} ( \frac{1}{2} w_{q+1} + (-\Delta)^{m} w_{q+1}), \label{estimate 127}\\
R_{\text{tran}} \triangleq& R_{\text{transport}} \triangleq \mathcal{R} ( (\partial_{t} + \Upsilon_{l} (v_{l} \cdot \nabla) ) w_{q+1}^{(p)} ), \label{estimate 123}\\
R_{\text{Nash}} \triangleq& \mathcal{R} ( \Upsilon_{l} (w_{q+1} \cdot \nabla) v_{l}),\label{estimate 124}\\
R_{\text{corr}} \triangleq& R_{\text{corrector}} \triangleq \mathcal{R} (( \partial_{t} + \Upsilon_{l} (v_{l} \cdot \nabla) ) w_{q+1}^{(c)}) + \Upsilon_{l} (w_{q+1}^{(c)} \mathring{\otimes} w_{q+1}  + w_{q+1}^{(p)} \mathring{\otimes} w_{q+1}^{(c)}), \label{estimate 125}\\
p_{\text{corr}} \triangleq& p_{\text{corrector}} \triangleq \Upsilon_{l} ( \frac{1}{3} \lvert w_{q+1}^{(c)} \rvert^{2} + \frac{2}{3} w_{q+1}^{(p)} \cdot w_{q+1}^{(c)}), \label{estimate 126}\\
R_{\text{com2}} \triangleq& R_{\text{commutator2}} \triangleq (\Upsilon - \Upsilon_{l}) (v_{q+1} \mathring{\otimes} v_{q+1}), \label{estimate 128}\\
p_{\text{com2}} \triangleq& p_{\text{commutator2}} \triangleq (\frac{\Upsilon - \Upsilon_{l}}{3}) \lvert v_{q+1} \rvert^{2},  \label{estimate 129}
\end{align}
\end{subequations}
and $p_{q+1} \triangleq p_{l} - p_{\text{osc}} - p_{\text{corr}} - p_{\text{com2}}$ while $\mathring{R}_{q+1}$ identically to  \eqref{estimate 115}. 

Now we start to work on $R_{\text{line}}$ from \eqref{estimate 127}. For any $\epsilon \in (0, 1-2m)$ fixed, we can estimate via Lemma \ref{[The. 1.4, RS16]}
\begin{equation}\label{estimate 117}
\lVert \mathcal{R} (-\Delta)^{m} w_{q+1} \rVert_{C_{t,x}} \overset{\eqref{estimate 145} \eqref{[Equ. (5.29) and (5.29a), BV19b]}}{\lesssim_{\epsilon}} \lVert \mathcal{R} w_{q+1}^{(p)} \rVert_{C_{t}C_{x}^{2m+\epsilon}} + \lVert \mathcal{R} w_{q+1}^{(c)} \rVert_{C_{t}C_{x}^{2m+\epsilon}}. 
\end{equation} 
First, we can rely on the fact that for any $s\in [0,t]$ there are at most two non-trivial cutoffs to obtain 
\begin{equation}\label{estimate 168}
\lVert \mathcal{R} w_{q+1}^{(p)} \rVert_{C_{t}C_{x}^{2m + \epsilon}} \overset{\eqref{estimate 95}}{\leq} 2\sup_{j} \sum_{\zeta \in \Lambda_{j}} \lVert \mathcal{R} (\bar{a}_{(\zeta)} W_{(\zeta)} (\Phi_{j}))\rVert_{C_{t}C_{x}^{2m+\epsilon}}.  
\end{equation} 
By \eqref{estimate 101} we see that \eqref{[Equ. (5.35), BV19b]} is satisfied with ``$C_{a}$'' $= e^{\frac{1}{2} L^{\frac{1}{4}}} \delta_{q+1}^{\frac{1}{2}} M_{0}(t)^{\frac{1}{2}} \lVert \gamma_{\zeta} \rVert_{C^{\lceil \frac{1}{2m} \rceil \vee 8} (B_{C_{\ast}} (\text{Id} ))}$ for all $N = 0, \hdots, \lceil \frac{1}{2m} \rceil \vee 8$ and thus we can choose $\beta < \frac{1}{3} (1-2m-\epsilon)$ and $a \in \mathbb{N}$ sufficiently large to deduce 
\begin{align}
\lVert \mathcal{R} w_{q+1}^{(p)} \rVert_{C_{t}C_{x}^{2m+ \epsilon}} \overset{\eqref{estimate 168} \eqref{[Equ. (5.36), BV19b]} \eqref{[Equ. (5.17), BV19b]}}{\lesssim}&  M e^{\frac{1}{2} L^{\frac{1}{4}}} \delta_{q+1}^{\frac{1}{2}} M_{0}(t)^{\frac{1}{2}} \lambda_{q+1}^{2m+ \epsilon -1} \nonumber \\ 
\lesssim&  c_{R} M_{0}(t) \delta_{q+2} [e^{\frac{1}{2} L^{\frac{1}{4}}}  a^{2^{q} (6 \beta + 2(2m + \epsilon -1) )}] \ll c_{R} M_{0}(t)\delta_{q+2}.  \label{estimate 118}
\end{align} 
Second, we see that 
\begin{equation}\label{estimate 169}
 \lVert \mathcal{R} w_{q+1}^{(c)} \rVert_{C_{t}C_{x}^{2m + \epsilon}} \overset{\eqref{[Equ. (5.29) and (5.29a), BV19b]}\eqref{estimate 105}}{\leq} 2\sup_{j} \sum_{\zeta \in \Lambda_{j}} \lVert \mathcal{R} ((\lambda_{q+1}^{-1} \nabla \bar{a}_{(\zeta)} + i \bar{a}_{(\zeta)} \zeta \cdot (\nabla \Phi_{j} - \text{Id} ) ) \times W_{(\zeta)} (\Phi_{j}) ) \rVert_{C_{t}C_{x}^{2m+ \epsilon}} 
\end{equation} 
where for any $N = 0, \hdots, \lceil \frac{1}{2m} \rceil \vee 8$, we can estimate by taking $a \in \mathbb{N}$ sufficiently large 
\begin{align}
&\lVert \lambda_{q+1}^{-1} \nabla \bar{a}_{(\zeta)} + i \bar{a}_{(\zeta)} \zeta \cdot  (\nabla \Phi_{j} - \text{Id} ) \rVert_{C_{t}C_{x}^{N}} \nonumber\\
\overset{\eqref{estimate 106}\eqref{estimate 107} }{\lesssim}& \lambda_{q+1}^{-1} e^{\frac{1}{2} L^{\frac{1}{4}}} c_{R}^{\frac{1}{4}} \delta_{q+1}^{\frac{1}{2}} M_{0}(t)^{\frac{1}{2}} \lVert \gamma_{\zeta} \rVert_{C^{N+1} (B_{C_{\ast}} (\text{Id} ))} l^{-N-1} \nonumber \\
&+ e^{\frac{3}{2} L^{\frac{1}{4}}} m_{L}^{4} \delta_{q+1}^{\frac{1}{2}} M_{0}(t)^{\frac{3}{2}} \delta_{q}^{\frac{1}{2}} \lambda_{q} \lVert \gamma_{\zeta} \rVert_{C^{N} (B_{C_{\ast}} (\text{Id} ))} l^{-N+1}  \nonumber\\
\lesssim& \delta_{q+1}^{\frac{1}{2}} \lVert \gamma_{\zeta} \rVert_{C^{N+1}(B_{C_{\ast}} (\text{Id} ))} M_{0}(t)^{\frac{1}{2}} \lambda_{q+1}^{-1} e^{\frac{1}{2} L^{\frac{1}{4}}} l^{-N-1}. \label{estimate 156}
\end{align}
Therefore, \eqref{[Equ. (5.35), BV19b]} holds with ``$C_{a}$''$= \delta_{q+1}^{\frac{1}{2}} M_{0}(t)^{\frac{1}{2}} \lVert \gamma_{\zeta} \rVert_{C^{(\lceil \frac{1}{2m} \rceil + 1) \vee 9} (B_{C_{\ast}} (\text{Id} ))} l^{-1} \lambda_{q+1}^{-1} e^{\frac{1}{2} L^{\frac{1}{4}}}$ and hence by by taking $\beta < \frac{1}{6} (\frac{5}{2} - 4m - 2\epsilon)$  and $a \in \mathbb{N}$ sufficiently large we obtain 
\begin{align}
 \lVert \mathcal{R} w_{q+1}^{(c)} \rVert_{C_{t}C_{x}^{2m+ \epsilon}} \overset{\eqref{estimate 169} \eqref{[Equ. (5.36), BV19b]} }{\lesssim}&  \sup_{j} \sum_{\zeta \in \Lambda_{j}} \frac{ \delta_{q+1}^{\frac{1}{2}} M_{0}(t)^{\frac{1}{2}} \lVert \gamma_{\zeta} \rVert_{C^{(\lceil \frac{1}{2m} \rceil + 1) \vee 9} (B_{C_{\ast}}(\text{Id} ))} l^{-1} \lambda_{q+1}^{-1} e^{\frac{1}{2} L^{\frac{1}{4}}} }{\lambda_{q+1}^{1- (2m + \epsilon)}}\nonumber\\
\overset{\eqref{[Equ. (5.17), BV19b]}}{\lesssim}& c_{R} M_{0}(t) \delta_{q+2} a^{2^{q} (6 \beta - \frac{5}{2} + 4m + 2 \epsilon)} \ll c_{R} M_{0}(t) \delta_{q+2}. \label{estimate 119}
\end{align}
Therefore, we conclude by applying \eqref{estimate 118} and \eqref{estimate 119} to \eqref{estimate 117} that   
\begin{equation}\label{estimate 120}
\lVert \mathcal{R} (-\Delta)^{m} w_{q+1} \rVert_{C_{t,x}} \overset{\eqref{estimate 117}\eqref{estimate 118} \eqref{estimate 119}}{\ll} c_{R} M_{0}(t) \delta_{q+2}. 
\end{equation}
As $\lVert \mathcal{R} (\frac{1}{2} w_{q+1} ) \rVert_{C_{t}C_{x}} \overset{\eqref{estimate 30} }{\lesssim} \lVert \mathcal{R} w_{q+1} \rVert_{C_{t}C_{x}^{2m + \epsilon}}$, we can apply the same estimates in \eqref{estimate 118} and \eqref{estimate 119} to $\mathcal{R} (\frac{1}{2} w_{q+1})$ and conclude that 
\begin{equation}\label{estimate 171}
\lVert R_{\text{line}} \rVert_{C_{t,x}} \overset{\eqref{estimate 127}}{=} \lVert \mathcal{R} ( \frac{1}{2} w_{q+1} + (-\Delta)^{m} w_{q+1}) \rVert_{C_{t,x}} \overset{\eqref{estimate 120} }{\ll} c_{R}M_{0}(t) \delta_{q+2}. 
\end{equation} 
Next, in order to work on $R_{\text{tran}}$ from \eqref{estimate 123} we make the key observation that 
\begin{align}
( \partial_{t} + \Upsilon_{l} (v_{l} \cdot \nabla) ) w_{q+1}^{(p)}(t,x) &\overset{\eqref{estimate 95}}{=}  \sum_{j} \sum_{\zeta \in \Lambda_{j}} [ \partial_{t} \bar{a}_{(\zeta)} (t,x) + \Upsilon_{l} (v_{l} \cdot \nabla) \bar{a}_{(\zeta)} (t,x)] W_{(\zeta)} (\Phi_{j} (t,x)) \nonumber\\
&+ \bar{a}_{(\zeta)} (t,x) \nabla W_{(\zeta)} (\Phi_{j} (t,x)) \cdot [ \partial_{t} \Phi_{j} (t,x) + \Upsilon_{l} (v_{l} \cdot \nabla) \Phi_{j} (t,x) ] \nonumber \\
\overset{\eqref{estimate 78}}{=}&  \sum_{j} \sum_{\zeta \in \Lambda_{j}} [ \partial_{t} \bar{a}_{(\zeta)} (t,x) + \Upsilon_{l} (v_{l} \cdot \nabla) \bar{a}_{(\zeta)} (t,x)] W_{(\zeta)} (\Phi_{j} (t,x)). \label{estimate 155}
\end{align} 
For any $\epsilon \in (\frac{1}{8}, \frac{1}{4})$ and $N = 0, \hdots, \lceil \frac{1}{\epsilon} \rceil \vee 8 = 8$, 
\begin{align}
& \lVert \partial_{t} \bar{a}_{(\zeta)} + \Upsilon_{l} (v_{l} \cdot \nabla) \bar{a}_{(\zeta)} \rVert_{C_{t}C_{x}^{N}} \\
&\overset{\eqref{estimate 102} \eqref{[Equ. (123), Y20a]} \eqref{estimate 88}}{\lesssim} e^{\frac{5}{2}L^{\frac{1}{4}}} c_{R}^{\frac{1}{4}} \delta_{q+1}^{\frac{1}{2}} M_{0}(t)^{\frac{1}{2}} \lVert \gamma_{\zeta} \rVert_{C^{N+1} (B_{C_{\ast}} (\text{Id} ))} l^{-N -1}  \nonumber \\
&\hspace{10mm} + e^{\frac{3}{2}L^{\frac{1}{4}}} \delta_{q+1}^{\frac{1}{2}} \lVert \gamma_{\zeta} \rVert_{C^{N+1} (B_{C_{\ast}} (\text{Id} ))} m_{L} M_{0}(t) l^{-N-1}  \lesssim \lVert \gamma_{\zeta} \rVert_{C^{N+1} (B_{C_{\ast}}(\text{Id} ))} \delta_{q+1}^{\frac{1}{2}} e^{\frac{3}{2} L^{\frac{1}{4}}} m_{L} M_{0}(t) l^{-N-1}. \nonumber
\end{align} 
Therefore, \eqref{[Equ. (5.35), BV19b]} is satisfied with ``$C_{a}$''$= \lVert \gamma_{\zeta} \rVert_{C^{9} (B_{C_{\ast}}(\text{Id} ))} \delta_{q+1}^{\frac{1}{2}} e^{\frac{3}{2} L^{\frac{1}{4}}} m_{L} M_{0}(t) l^{-1}$. Hence, by taking $\beta < \frac{1}{6} (\frac{1}{2} - 2 \epsilon)$ and $a \in \mathbb{N}$ sufficiently large we obtain 
\begin{align}
\lVert R_{\text{tran}} \rVert_{C_{t,x}} & \overset{ \eqref{estimate 123} \eqref{estimate 155}}{\lesssim}  \lVert \mathcal{R} (\sum_{j}\sum_{\zeta \in \Lambda_{j}} ( \partial_{t} \bar{a}_{(\zeta)} + \Upsilon_{l} (v_{l} \cdot \nabla) \bar{a}_{(\zeta)} ) W_{(\zeta)} (\Phi_{j}) ) \rVert_{C_{t}C_{x}^{\epsilon}} \label{estimate 172}\\
\overset{\eqref{[Equ. (5.36), BV19b]} \eqref{[Equ. (5.17), BV19b]}}{\lesssim}& \delta_{q+1}^{\frac{1}{2}} M_{0}(t) \lambda_{q}^{\frac{3}{2}} \lambda_{q+1}^{\epsilon -1} e^{\frac{3}{2} L^{\frac{1}{4}}} m_{L} \approx c_{R} M_{0}(t) \delta_{q+2} [a^{2^{q} (6\beta -\frac{1}{2} + 2 \epsilon)} e^{\frac{3}{2} L^{\frac{1}{4}}} m_{L}] \ll c_{R} M_{0}(t) \delta_{q+2}.  \nonumber
\end{align} 
Next, we work on $R_{\text{Nash}}$ in \eqref{estimate 124} which may be written using the fact that $\bar{a}_{(\zeta)} = \Upsilon_{l}^{-\frac{1}{2}} a_{(\zeta)}$ due to \eqref{estimate 90} as follows: 
\begin{align}
&R_{\text{Nash}} 
\overset{\eqref{estimate 124}\eqref{[Equ. (5.29) and (5.29a), BV19b]}\eqref{estimate 93}\eqref{estimate 105} }{=} \Upsilon_{l}^{\frac{1}{2}}  \sum_{j} \sum_{\zeta \in \Lambda_{j}} \mathcal{R} ( (a_{(\zeta)} W_{(\zeta)} (\Phi_{j}) \cdot \nabla) v_{l} \nonumber\\
&\hspace{25mm} +  ([ (\lambda_{q+1}^{-1} \nabla a_{(\zeta)} + i a_{(\zeta)} \zeta \cdot (\nabla \Phi_{j} - \text{Id} )) \times W_{(\zeta)} (\Phi_{j})]\cdot \nabla)v_{l}). \label{estimate 157}
\end{align} 
Now $\lVert \Upsilon_{l}^{\frac{1}{2}} \rVert_{C_{t}} \lesssim e^{\frac{1}{2} L^{\frac{1}{4}}}$ by \eqref{[Equ. (123), Y20a]} and thus considering \eqref{estimate 121}, for $\epsilon \in (\frac{1}{8}, \frac{1}{2})$ and choosing $\beta < \frac{1}{5} (1-2\epsilon)$ gives us immediately for $a \in \mathbb{N}$ sufficiently large 
\begin{align}
&\lVert R_{\text{Nash}} \rVert_{C_{t,x}} \label{estimate 173}\\
& \lesssim \lVert \Upsilon_{l}^{\frac{1}{2}} \rVert_{C_{t}} \sup_{j}  \sum_{\zeta \in \Lambda_{j}} \lVert \mathcal{R} (( a_{(\zeta)} W_{(\zeta)} \circ \Phi_{j} + ( \lambda_{q+1}^{-1} \nabla a_{(\zeta)} + ia_{(\zeta)} (\nabla \Phi_{j} - \text{Id} ) \zeta ) \times W_{(\zeta)} (\Phi_{j} )) \cdot \nabla v_{l} ) \rVert_{C_{t}C_{x}^{\epsilon}}  \nonumber \\
&\hspace{40mm} \overset{\eqref{estimate 121}}{\lesssim} e^{\frac{1}{2} L^{\frac{1}{4}}} (c_{R} M_{0}(t) \delta_{q+2} [M_{0}(t)^{\frac{1}{2}} a^{2^{q} (5 \beta-1+2\epsilon) } ]) \ll c_{R} M_{0}(t) \delta_{q+2}.  \nonumber 
\end{align} 
Next, we look at $R_{\text{corr}}$ from \eqref{estimate 125}. Again, we make the key observation that 
\begin{align}
&( \partial_{t} + \Upsilon_{l} (v_{l} \cdot \nabla ) ) w_{q+1}^{(c)} \label{estimate 131}\\
\overset{\eqref{[Equ. (5.29) and (5.29a), BV19b]} \eqref{estimate 105} \eqref{[Equ. (5.12), BV19b]}}{=}& \sum_{j} \sum_{\zeta \in\Lambda_{j}} \partial_{t} (\lambda_{q+1}^{-1} \nabla \bar{a}_{(\zeta)} + i \bar{a}_{(\zeta)} \zeta \cdot  (\nabla \Phi_{j} - \text{Id} )) \times W_{(\zeta)} (\Phi_{j}) \nonumber \\
&+ \Upsilon_{l} (v_{l} \cdot \nabla) (\lambda_{q+1}^{-1} \nabla \bar{a}_{(\zeta)} + i \bar{a}_{(\zeta)} \zeta \cdot ( \nabla \Phi_{j} - \text{Id} ) ) \times W_{(\zeta)} (\Phi_{j}) \nonumber \\
&+ (\lambda_{q+1}^{-1}\nabla \bar{a}_{(\zeta)} + i \bar{a}_{(\zeta)} \zeta \cdot (\nabla \Phi_{j} - \text{Id})) \times i \lambda_{q+1} \zeta \cdot W_{(\zeta)}(\Phi_{j}) [ \partial_{t} \Phi_{j} + \Upsilon_{l} (v_{l} \cdot \nabla) \Phi_{j} ] \nonumber \\
\overset{\eqref{estimate 78}}{=}&  \sum_{j} \sum_{\zeta \in\Lambda_{j}} \partial_{t} (\lambda_{q+1}^{-1} \nabla \bar{a}_{(\zeta)} + i \bar{a}_{(\zeta)} \zeta \cdot (\nabla \Phi_{j} - \text{Id} ) ) \times W_{(\zeta)} (\Phi_{j}) \nonumber \\
&+ \Upsilon_{l} (v_{l} \cdot \nabla) (\lambda_{q+1}^{-1} \nabla \bar{a}_{(\zeta)} + i \bar{a}_{(\zeta)} \zeta \cdot  ( \nabla \Phi_{j} - \text{Id} )) \times W_{(\zeta)} (\Phi_{j}). \nonumber 
\end{align} 
For any $\epsilon \in (\frac{1}{8}, \frac{1}{2})$ and $N = 0, \hdots, \lceil \frac{1}{2\epsilon} \rceil \vee 8 = 8$, by taking $a \in \mathbb{N}$ sufficiently large we can separately estimate 
\begin{subequations}\label{estimate 130}
\begin{align}
& \lambda_{q+1}^{-1} \lVert \partial_{t} \nabla \bar{a}_{(\zeta)} \rVert_{C_{t}C_{x}^{N}} \overset{\eqref{estimate 102}}{\lesssim} \lambda_{q+1}^{-1} e^{\frac{5}{2}L^{\frac{1}{4}}} \delta_{q+1}^{\frac{1}{2}} M_{0}(t)^{\frac{1}{2}} \lVert \gamma_{\zeta} \rVert_{C^{N+2} (B_{C_{\ast}} (\text{Id} ))} l^{-N -2}, \\
&  \lVert \partial_{t} \bar{a}_{(\zeta)} (\nabla \Phi_{j} - \text{Id} ) \rVert_{C_{t}C_{x}^{N}} \overset{\eqref{estimate 102} \eqref{estimate 107}}{\lesssim} e^{\frac{7}{2} L^{\frac{1}{4}}} \delta_{q+1}^{\frac{1}{2}} M_{0}(t)^{\frac{3}{2}} \lVert \gamma_{\zeta} \rVert_{C^{N+1} (B_{C_{\ast}} (\text{Id} ))} m_{L}^{4} \delta_{q}^{\frac{1}{2}} \lambda_{q} l^{-N},  \\
&  \lVert \bar{a}_{(\zeta)} \partial_{t} \nabla \Phi_{j} \rVert_{C_{t}C_{x}^{N}} \overset{\eqref{estimate 102} \eqref{estimate 84}}{\lesssim} e^{\frac{5}{2} L^{\frac{1}{4}}} \delta_{q+1}^{\frac{1}{2}} M_{0}(t)^{2} m_{L}^{5} \lVert \gamma_{\zeta} \rVert_{C^{N} (B_{C_{\ast}} (\text{Id} ))} \delta_{q}^{\frac{1}{2}} \lambda_{q} l^{-N},  \\
&  \lambda_{q+1}^{-1} \lVert \Upsilon_{l} (v_{l} \cdot \nabla) \nabla \bar{a}_{(\zeta)} \rVert_{C_{t}C_{x}^{N}} \overset{\eqref{estimate 88}\eqref{estimate 101}}{\lesssim} \lambda_{q+1}^{-1} e^{\frac{3}{2}L^{\frac{1}{4}}} \delta_{q+1}^{\frac{1}{2}} \lVert \gamma_{\zeta} \rVert_{C^{N+2} (B_{C_{\ast}} (\text{Id} ))} m_{L}^{4} M_{0}(t)^{\frac{3}{2}} l^{-N-2},  \\
& \lVert \Upsilon_{l} (v_{l} \cdot \nabla) \bar{a}_{(\zeta)} (\nabla \Phi_{j} - \text{Id} ) \rVert_{C_{t}C_{x}^{N}} 
\overset{\eqref{estimate 88}\eqref{estimate 107}}{\lesssim}  e^{\frac{5}{2}L^{\frac{1}{4}}} \delta_{q+1}^{\frac{1}{2}} \delta_{q}^{\frac{1}{2}} \lambda_{q} M_{0}(t)^{\frac{5}{2}} m_{L}^{5} \lVert \gamma_{\zeta} \rVert_{C^{N+1} (B_{C_{\ast}} (\text{Id} ))} l^{-N}, \\
&  \lVert \Upsilon_{l} v_{l} \bar{a}_{(\zeta)} \cdot \nabla \nabla \Phi_{j} \rVert_{C_{t}C_{x}^{N}} \overset{\eqref{[Equ. (123), Y20a]} \eqref{estimate 88} \eqref{estimate 106}\eqref{estimate 107}}{\lesssim} e^{ \frac{5}{2} L^{\frac{1}{4}}} \delta_{q+1}^{\frac{1}{2}} \delta_{q}^{\frac{1}{2}} \lambda_{q} M_{0}(t)^{\frac{5}{2}} m_{L}^{5} \lVert \gamma_{\zeta} \rVert_{C^{N}(B_{C_{\ast}}(\text{Id} ))} l^{-N}.
\end{align}
\end{subequations} 
Using \eqref{estimate 130} we can estimate for all $N = 0, \hdots, \lceil \frac{1}{\epsilon} \rceil \vee 8 = 8$,  
\begin{align}
& \lVert \partial_{t} (\lambda_{q+1}^{-1} \nabla \bar{a}_{(\zeta)} + i \bar{a}_{(\zeta)} \zeta \cdot (\nabla \Phi_{j} - \text{Id} )) + \Upsilon_{l} (v_{l} \cdot \nabla) (\lambda_{q+1}^{-1} \nabla \bar{a}_{(\zeta)} + i \bar{a}_{(\zeta)} \zeta \cdot (\nabla \Phi_{j} - \text{Id} ) ) \rVert_{C_{t}C_{x}^{N}} \nonumber  \\
\lesssim& \lambda_{q+1}^{-1} \lVert \partial_{t} \nabla \bar{a}_{(\zeta)} \rVert_{C_{t}C_{x}^{N}} + \lVert \partial_{t} \bar{a}_{(\zeta)} (\nabla \Phi_{j} - \text{Id} ) \rVert_{C_{t}C_{x}^{N}} + \lVert \bar{a}_{(\zeta)} \partial_{t} \nabla \Phi_{j} \rVert_{C_{t}C_{x}^{N}} \nonumber \\
&\hspace{3mm} + \lambda_{q+1}^{-1} \lVert \Upsilon_{l} (v_{l} \cdot \nabla) \nabla \bar{a}_{(\zeta)} \rVert_{C_{t}C_{x}^{N}} + \lVert \Upsilon_{l} (v_{l} \cdot \nabla) \bar{a}_{(\zeta)} (\nabla \Phi_{j} - \text{Id} ) \rVert_{C_{t}C_{x}^{N}} + \lVert \Upsilon_{l} v_{l} \bar{a}_{(\zeta)} \cdot \nabla^{2} \Phi_{j} \rVert_{C_{t}C_{x}^{N}} \nonumber \\
& \hspace{40mm} \overset{\eqref{estimate 130}}{\lesssim} \delta_{q+1}^{\frac{1}{2}}\lambda_{q+1}^{-1} l^{-N-2} e^{\frac{5}{2}L^{\frac{1}{4}}} \lVert \gamma_{\zeta} \rVert_{C^{N+2} (B_{C_{\ast}} (\text{Id} ))} m_{L}^{5}M_{0}(t)^{\frac{5}{2}}. \label{estimate 158}
\end{align} 
Hence, \eqref{[Equ. (5.35), BV19b]} holds with 
``$C_{a}$'' = $ \delta_{q+1}^{\frac{1}{2}}\lambda_{q+1}^{-1} l^{-2} e^{\frac{5}{2} L^{\frac{1}{4}}} \lVert \gamma_{\zeta} \rVert_{C^{10} (B_{C_{\ast}} (\text{Id} ))} m_{L}^{5} M_{0}(t)^{\frac{5}{2}}$. Therefore, by choosing $\beta < \frac{1}{6} (1-2\epsilon)$ and $a \in \mathbb{N}$ sufficiently large, relying on the fact that for all $s \in [0,t]$ fixed, there exist at most two non-trivial cutoffs, we can estimate 
\begin{align}
& \lVert \mathcal{R} ( (\partial_{t} + \Upsilon_{l} (v_{l} \cdot \nabla )) w_{q+1}^{(c)} ) \rVert_{C_{t,x}} \nonumber \\
\overset{\eqref{estimate 131}}{\lesssim}& \sup_{j} \sum_{\zeta \in \Lambda_{j}} \lVert \mathcal{R} (\partial_{t} (\lambda_{q+1}^{-1} \nabla \bar{a}_{(\zeta)} + i \bar{a}_{(\zeta)} \zeta \cdot  (\nabla \Phi_{j} - \text{Id} )) \times W_{(\zeta)} (\Phi_{j}) \nonumber \\
& \hspace{10mm} + \Upsilon_{l} (v_{l} \cdot \nabla) (\lambda_{q+1}^{-1} \nabla \bar{a}_{(\zeta)} + i \bar{a}_{(\zeta)} \zeta \cdot  ( \nabla \Phi_{j} - \text{Id} )) \times W_{(\zeta)} (\Phi_{j})) \rVert_{C_{t}C_{x}^{\epsilon}} \nonumber \\
\overset{\eqref{[Equ. (5.36), BV19b]} \eqref{[Equ. (5.17), BV19b]}}{\lesssim}& c_{R} M_{0}(t) \delta_{q+2} [ a^{2^{q}(6 \beta + 2 \epsilon - 1)} e^{\frac{5}{2} L^{\frac{1}{4}}} m_{L}^{5} M_{0}(t)^{\frac{3}{2}}]  \ll c_{R} M_{0}(t) \delta_{q+2}.  \label{estimate 133}
\end{align} 
Next, we can directly estimate within $R_{\text{corr}}$ from \eqref{estimate 125} for $\beta < \frac{1}{8}$ and $a \in \mathbb{N}$ sufficiently large 
\begin{align}
\lVert  \Upsilon_{l} (w_{q+1}^{(c)} \mathring{\otimes} w_{q+1}  + w_{q+1}^{(p)} \mathring{\otimes} w_{q+1}^{(c)}) \rVert_{C_{t,x}} \overset{\eqref{[Equ. (5.29) and (5.29a), BV19b]} \eqref{[Equ. (123), Y20a]}}{\lesssim}& e^{L^{\frac{1}{4}}} \lVert w_{q+1}^{(c)} \rVert_{C_{t,x}} [\lVert w_{q+1}^{(c)} \rVert_{C_{t,x}} + \lVert w_{q+1}^{(p)} \rVert_{C_{t,x}}]\nonumber \\
\overset{\eqref{estimate 103}\eqref{estimate 97} }{\lesssim}&  c_{R} M_{0}(t) \delta_{q+2}  e^{2L^{\frac{1}{4}}} a^{2^{q} ( 4 \beta - \frac{1}{2} )}  \ll c_{R} M_{0}(t) \delta_{q+2}.  \label{estimate 132} 
\end{align} 
Considering \eqref{estimate 133} and \eqref{estimate 132}, we now conclude that 
\begin{equation}\label{estimate 174}
\lVert R_{\text{corr}} \rVert_{C_{t,x}} \overset{\eqref{estimate 125} \eqref{estimate 132}\eqref{estimate 133}}{\ll} c_{R}M_{0}(t) \delta_{q+2}. 
\end{equation} 
Finally, we can estimate $R_{\text{com1}}$ in \eqref{[Equ. (139b), Y20a]} and $R_{\text{com2}}$ in \eqref{estimate 128} as follows. First, we can write 
\begin{align}
& (( \Upsilon (v_{q} \mathring{\otimes} v_{q}))\ast_{x} \phi_{l})\ast_{t} \varphi_{l} - \Upsilon_{l} (v_{l} \mathring{\otimes} v_{l}) \label{estimate 170} \\
=& (( \Upsilon (v_{q} \mathring{\otimes} v_{q})) \ast_{x} \phi_{l})\ast_{t} \varphi_{l} - ( \Upsilon ( v_{q} \ast_{x} \phi_{l}) \mathring{\otimes} (v_{q} \ast_{x} \phi_{l}) ) \ast_{t} \varphi_{l} \nonumber\\
&+ ( \Upsilon ( v_{q} \ast_{x} \phi_{l}) \mathring{\otimes} (v_{q} \ast_{x} \phi_{l}) ) \ast_{t} \varphi_{l}  - \Upsilon_{l} ([ ( v_{q} \ast_{x} \phi_{l})\mathring{\otimes} (v_{q} \ast_{x} \phi_{l} )] \ast_{t} \varphi_{l}) \nonumber\\
&+ \Upsilon_{l} ([ ( v_{q} \ast_{x} \phi_{l})\mathring{\otimes} (v_{q} \ast_{x} \phi_{l} )] \ast_{t} \varphi_{l})  - \Upsilon_{l} [ (v_{q} \ast_{x} \phi_{l} \ast_{t} \varphi_{l}) \mathring{\otimes} (v_{q} \ast_{x} \phi_{l} \ast_{t} \varphi_{l})], \nonumber
\end{align} 
apply standard commutator estimate to it (e.g., \cite[Pro. 6.5]{BV19b} or \cite[Equ. (5)]{CDS12}) so that we can estimate by taking $\beta < \frac{3}{182}$ and $a \in \mathbb{N}$ sufficiently large, as well as using the fact that $\delta \in (0, \frac{1}{24})$ 
\begin{align}
\lVert R_{\text{com1}} \rVert_{C_{t,x}} \overset{\eqref{[Equ. (123), Y20a]}}{\lesssim}& l e^{L^{\frac{1}{4}}} \lVert v_{q} \rVert_{C_{t,x}} \lVert v_{q} \rVert_{C_{t,x}^{1}} + l^{\frac{1}{2} - 2 \delta} m_{L}^{2} \lVert v_{q} \rVert_{C_{t,x}}^{\frac{3}{2} + 2 \delta} \lVert v_{q} \rVert_{C_{t,x}^{1}}^{\frac{1}{2} - 2 \delta} \nonumber\\ 
\overset{\eqref{[Equ. (124), Y20a]}\eqref{[Equ. (126b), Y20a]} \eqref{[Equ. (126a), Y20a]} }{\lesssim}&  c_{R} M_{0}(t) \delta_{q+2} [ m_{L}^{\frac{11}{2} - 6 \delta} M_{0}(t)^{\frac{1}{4} - \delta} a^{2^{q} (-\frac{1}{8} + \frac{91\beta}{12})}] \ll c_{R} M_{0}(t) \delta_{q+2}. \label{estimate 135}
\end{align}  
We point out that in contrast to \cite{HZZ19, Y20a}, this is where we need $\delta \in (0, \frac{1}{24})$ rather than $\delta \in (0, \frac{1}{12})$ from previous works such as \cite[p. 43]{HZZ19} and \cite[p. 30]{Y20a}, essentially due to a new choice of $l$ in \eqref{[Equ. (56), Y20a]}. Second, as $\lvert \Upsilon (t) - \Upsilon_{l}(t) \rvert \lesssim l^{\frac{1}{2} - 2 \delta} \lVert \Upsilon \rVert_{C_{t}^{\frac{1}{2} - 2 \delta}} \lesssim l^{\frac{1}{2} - 2 \delta} m_{L}^{2}$ due to \eqref{[Equ. (124), Y20a]} and $-\frac{3}{4} + 3 \delta < - \frac{5}{8}$ because $\delta \in (0, \frac{1}{24})$, by taking $\beta < \frac{5}{64}$ and $a \in \mathbb{N}$ sufficiently large we obtain 
\begin{equation}\label{estimate 175}
\lVert R_{\text{com2}} \rVert_{C_{t,x}} 
\overset{\eqref{estimate 128}\eqref{estimate 134}}{\lesssim} l^{\frac{1}{2} - 2 \delta} m_{L}^{2} (m_{L} M_{0}(t)^{\frac{1}{2}})^{2} \lesssim c_{R} M_{0}(t) \delta_{q+2} \lambda_{q}^{8\beta - \frac{5}{8}} m_{L}^{4} \ll c_{R} M_{0}(t) \delta_{q+2}.
\end{equation}
Applying \eqref{estimate 171}, \eqref{estimate 172}, \eqref{estimate 114}, \eqref{estimate 173}, \eqref{estimate 174}, \eqref{estimate 135}-\eqref{estimate 175} to \eqref{estimate 115} verifies \eqref{[Equ. (126c), Y20a]} at level $q+1$.  

The verification of how $(v_{q+1}, \mathring{R}_{q+1})$ are $(\mathcal{F}_{t})_{t\geq 0}$-adapted and that $v_{q+1}(0,x)$ and $\mathring{R}_{q+1}(0,x)$ are deterministic if $v_{q}(0,x)$ and $\mathring{R}_{q}(0,x)$ are deterministic is similar to the proof of Proposition \ref{[Pro. 4.8, Y20a]} and previous works \cite{HZZ19, Y20a}. \\

\section{Appendix}\label{Appendix}
\subsection{Further preliminaries}
\begin{lemma}\label{divergence inverse}
\rm{(\cite[Equ. (5.34)]{BV19b})} For any $v \in C^{\infty}(\mathbb{T}^{3})$ that is mean-zero, define 
\begin{equation}\label{estimate 5}
(\mathcal{R}v)_{kl} \triangleq ( \partial^{k}\Delta^{-1} v^{l} + \partial^{l} \Delta^{-1} v^{k}) - \frac{1}{2} (\delta_{kl} + \partial^{k} \partial^{l} \Delta^{-1}) \text{div} \Delta^{-1} v
\end{equation} 
for $k, l \in \{1,2,3\}$. Then $\mathcal{R} v(x)$ is a symmetric trace-free matrix for each $x \in \mathbb{T}^{3}$ that  satisfies $\text{div} (\mathcal{R} v) = v$. When $v$ does not satisfy $\int_{\mathbb{T}^{3}} vdx = 0$, we overload notation and denote $\mathcal{R} v \triangleq \mathcal{R} (v  -\int_{\mathbb{T}^{3}} v dx)$. Moreover, $\mathcal{R}$ satisfies the classical Calder$\acute{\mathrm{o}}$n-Zygmund and Schauder estimates: $\lVert (-\Delta)^{\frac{1}{2}} \mathcal{R} \rVert_{L_{x}^{p} \mapsto L_{x}^{p}} + \lVert \mathcal{R} \rVert_{L_{x}^{p} \mapsto L_{x}^{p}}  + \lVert \mathcal{R} \rVert_{C_{x} \mapsto C_{x}} \lesssim 1$ for all $p \in (1, \infty)$. 
\end{lemma} 

The following stationary phase lemma played a crucial role in our proofs. 
\begin{lemma}\label{[Lem. 5.7, BV19b]}
\rm{(\cite[Lem. 5.7]{BV19b}, \cite[Lem. 2.2]{DS17} )} Let $\lambda \zeta \in \mathbb{Z}^{3}$, $\alpha \in (0,1)$, and $p  \in \mathbb{N}$. Assume that $a \in C^{p+\alpha}(\mathbb{T}^{3})$ and $\Phi \in C^{p+1 + \alpha} (\mathbb{T}^{3})$ are smooth functions such that the phase function $\Phi$ obeys 
\begin{equation}\label{[Equ. (5.34a), BV19b]}
C^{-1} \leq \lvert \nabla \Phi \rvert \leq C
\end{equation} 
on $\mathbb{T}^{3}$, for some constant $C \geq 1$. Then 
\begin{equation}\label{[Equ. (5.34b), BV19b]}
\lVert \mathcal{R} ( a(x) e^{i \lambda \zeta \cdot \Phi(x)} ) \rVert_{C_{x}^{\alpha}} \lesssim \frac{\lVert a \rVert_{C_{x}}}{\lambda^{1-\alpha}} + \frac{ \lVert a \rVert_{C_{x}^{p+\alpha}} + \lVert a \rVert_{C_{x}} \lVert \nabla \Phi \rVert_{C_{x}^{p+\alpha}}}{\lambda^{p-\alpha}}. 
\end{equation} 
\end{lemma}

\begin{lemma}\label{[The. 1.4, RS16]}
(\cite[The. 1.4]{RS16}, \cite[The. B.1]{D19}) Let $\gamma, \epsilon > 0$ and $\beta \geq 0$ such that $2 \gamma + \beta + \epsilon \leq 1$, and let $f(t): \mathbb{T}^{3} \mapsto \mathbb{R}^{3}$. If $f \in C_{x}^{2 \gamma + \beta + \epsilon}$, then $(-\Delta)^{\gamma} f \in C_{x}^{\beta}$, and there exists a constant $C = C(\epsilon)> 0$ such that 
\begin{equation}\label{estimate 145}
\lVert (-\Delta)^{\gamma} f \rVert_{C_{t}C_{x}^{\beta}} \leq C(\epsilon) [ f ]_{C_{t}C_{x}^{2 \gamma + \beta + \epsilon}}. 
\end{equation} 
\end{lemma}

\subsection{Proof of Theorem \ref{Theorem 2.2}}
The proof is similar to those of previous works \cite{HZZ19, Y20a}. In short, we can fix $T > 0$ arbitrarily, any $\kappa \in (0,1)$ and $K > 1$ such that $\kappa K^{2} \geq 1$, rely on Theorem \ref{Theorem 2.1} and Proposition \ref{[Proposition 3.8, HZZ19]} to deduce the existence of $L > 1$ and a measure $P \otimes_{\tau_{L}} R$ that is a martingale solution to \eqref{stochastic GNS} on $[0,\infty)$ starting from a deterministic initial condition $\xi^{\text{in}}$ of Theorem \ref{Theorem 2.1} which coincides with $P = \mathcal{L} (u)$ over a random interval $[0, \tau_{L} ]$ and satisfies 
\begin{equation}\label{[Equation (3.19a), HZZ19]}
P \otimes_{\tau_{L}} R( \{ \tau_{L} \geq T \}) 
 \overset{\eqref{[Equ. (23), Y20a]} \eqref{estimate 6}}{=} \int_{\Omega_{0}} Q_{\omega} (\{ \tau_{L} (\omega) \geq T \}) P(d \omega) \overset{\eqref{[Equ. (23), Y20a]} \eqref{[Equ. (184), Y20a]}}{=} \textbf{P} ( \{ T_{L} \geq T \})  \overset{\eqref{estimate 1}}{>} \kappa.
\end{equation} 
It follows that  
\begin{equation}\label{estimate 161}
\mathbb{E}^{P \otimes_{\tau_{L}} R} [ \lVert \xi(T) \rVert_{L_{x}^{2}}^{2}] \overset{\eqref{estimate 7} \eqref{[Equation (3.19a), HZZ19]} }{>} \kappa [ K \lVert \xi^{\text{in}} \rVert_{L_{x}^{2}} + K (T \text{Tr} (GG^{\ast} ))^{\frac{1}{2}}]^{2}  
\geq  \kappa K^{2}(\lVert \xi^{\text{in}} \rVert_{L_{x}^{2}}^{2} + T \text{Tr} (GG^{\ast})). 
\end{equation} 
On the other hand, the classical method of Galerkin approximation gives us another martingale solution $\Theta$ (e.g., \cite{FR08, GRZ09}) which starts from the same initial condition $\xi^{\text{in}}$ and satisfies 
\begin{equation*}
\mathbb{E}^{\Theta} [ \lVert \xi(T)\rVert_{L_{x}^{2}}^{2}] \leq \lVert \xi^{\text{in}} \rVert_{L_{x}^{2}}^{2} + T\text{Tr}(GG^{\ast}).
\end{equation*} 
Because $\kappa K^{2} \geq 1$, this implies $P \otimes_{\tau_{L}} R\neq \Theta$ and hence \eqref{stochastic GNS} fails the uniqueness in law. 
 
\subsection{Proof of Theorem \ref{Theorem 2.4}}
The proof is similar to that of Theorem \ref{Theorem 2.2}; we sketch it for completeness. We fix $T> 0$ arbitrarily, any $\kappa \in (0,1)$, and $K > 1$ such that $\kappa K^{2} \geq 1$. The probability measure $P \otimes_{\tau_{L}} R$ from Proposition \ref{[Pro. 5.5, Y20a]} satisfies $P \otimes_{\tau_{L}} R (\{ \tau_{L} \geq T \})  > \kappa$ due to \eqref{[Equ. (144), Y20c]}-\eqref{[Equ. (147), Y20c]} and \eqref{estimate 1} which, together with \eqref{estimate 68}, implies 
\begin{equation}\label{estimate 162}
\mathbb{E}^{P \otimes_{\tau_{L}} R} [ \lVert \xi(T) \rVert_{L_{x}^{2}}^{2}] > \kappa K^{2} e^{T} \lVert \xi^{\text{in}} \rVert_{L_{x}^{2}}^{2}, 
\end{equation}
where $\xi^{\text{in}}$ is the deterministic initial condition constructed through Theorem \ref{Theorem 2.3}. On the other hand, via a standard Galerkin approximation scheme (e.g., \cite{FR08, GRZ09}), one can readily construct a probabilistically weak solution $\Theta$ to \eqref{stochastic GNS} starting also from $\xi^{\text{in}}$ such that 
\begin{equation*}
\mathbb{E}^{\Theta} [ \lVert \xi(T) \rVert_{L_{x}^{2}}^{2}] \leq e^{T} \lVert \xi^{\text{in}} \rVert_{L_{x}^{2}}^{2}. 
\end{equation*}
This implies the lack of joint uniqueness in law for \eqref{stochastic GNS} and consequently the non-uniqueness in law for \eqref{stochastic GNS} by \cite[The. C.1]{HZZ19}, which is an infinite-dimensional version of \cite[The. 3.1]{C03}. 

\section*{Acknowledgements}
The author expresses deep gratitude to Prof. Jiahong Wu and Prof. Zachary Bradshaw for valuable discussions and Prof. Michael Salins for informing the author of the works \cite{M96, MMP14, MS93}.

\end{document}